\documentclass[12pt]{amsart}
\usepackage{graphics, amsmath, amsfonts, amssymb, amsthm, amscd, mathrsfs}
\newtheorem{theorem}{Theorem}[section]
\newtheorem{proposition}[theorem]{Proposition}
\newtheorem{lemma}[theorem]{Lemma}
\newtheorem{corollary}[theorem]{Corollary}

\theoremstyle{definition}
\newtheorem{definition}[theorem]{Definition}
\newtheorem{example}[theorem]{Example}
\newtheorem{remark}[theorem]{Remark}

\numberwithin{equation}{section}

\newcommand{\A}{{\mathscr A}}

\newcommand{\C}{{\mathbb C}}

\newcommand{\NN}{{\mathcal N}}
\newcommand{\N}{{\mathbb N}}
\renewcommand{\O}{{\mathscr O}}
\newcommand{\R}{{\mathbb R}}
\newcommand{\Z}{{\mathbb Z}}
\newcommand{\F}{{\mathcal F}}
\newcommand{\LF}{{\mathcal {LF}}}
\newcommand{\Aut}{{\operatorname{Aut}}}
\newcommand{\Der}{{\operatorname{Der}}}
\newcommand{\ci}{{\mathcal C^\infty}}

\newcommand{\inv}{{^{-1}}}
\renewcommand{\sl}{{\operatorname{/\!\!/}}}
\newcommand{\pr}{{\operatorname{pr}}}
\newcommand{\ql}{{\operatorname{q\ell}}}
\newcommand{\lie}{\mathfrak}
\renewcommand{\phi}{\varphi}
\newcommand{\GL}{\operatorname{GL}}
\newcommand{\SL}{\operatorname{SL}}
\newcommand{\SO}{\operatorname{SO}}
\newcommand{\Orth}{\operatorname{O}}
\newcommand{\pt}{\partial}
\newcommand{\alg}{{\operatorname{alg}}}
\newcommand{\supp}{{\operatorname{supp}}}
\newcommand{\FF}{{\mathfrak F}}
\newcommand{\LFF}{{\mathfrak {LF}}}
\newcommand{\CC}{{\mathcal C^0}}
\newcommand{\vb}{{\operatorname{vb}}}
\newcommand{\ad}{{\operatorname{ad}}}
\newcommand{\id}{{\operatorname{id}}}
\newcommand{\gf}{{\operatorname{fin}}}
\newcommand{\red}{{\operatorname{red}}}
\newcommand{\Iso}{{\operatorname{Iso}}}
\renewcommand{\Im}{{\operatorname{Im}}}
\newcommand{\Id}{{\operatorname{id}}}
\newcommand{\Ad}{{\operatorname{Ad}}}
\newcommand{\comptensor}{{\,\widehat\otimes\,}}
\newcommand{\hr}{{\operatorname{hr}}}
\newcommand{\I}{\mathcal I}

\newcommand{\Mor}{\operatorname{Mor}}
\begin{document}

\date{August 13, 2016}
\title{Homotopy principles for equivariant isomorphisms}

\author{Frank Kutzschebauch, Finnur L\'arusson, Gerald W.~Schwarz}

\address{Frank Kutzschebauch, Institute of Mathematics, University of Bern, Sidlerstrasse 5, CH-3012 Bern, Switzerland}
\email{frank.kutzschebauch@math.unibe.ch}

\address{Finnur L\'arusson, School of Mathematical Sciences, University of Adelaide, Adelaide SA 5005, Australia}
\email{finnur.larusson@adelaide.edu.au}

\address{Gerald W.~Schwarz, Department of Mathematics, Brandeis University, Waltham MA 02454-9110, USA}
\email{schwarz@brandeis.edu}

\thanks{F.~Kutzschebauch was partially supported by Schweizerischer Nationalfond grant 200021-140235/1.  F.~L\'arusson was partially supported by Australian Research Council grants DP120104110 and DP150103442.  F.~ L\'arusson and G.~W.~Schwarz would like to thank the University of Bern for hospitality and financial support and F.~Kutzschebauch and G.~W.~Schwarz would like to thank the University of Adelaide for hospitality and the Australian Research Council for financial support.}

\begin{abstract}  
Let $G$ be a reductive complex Lie group acting holomorphically on  Stein manifolds $X$ and $Y$. Let $p_X\colon X\to Q_X$ and $p_Y\colon Y\to Q_Y$ be the quotient mappings. When is there an equivariant biholomorphism of $X$ and $Y$?  A necessary condition is that the categorical quotients $Q_X$ and $Q_Y$ are biholomorphic and that the biholomorphism $\phi$ sends the Luna strata of $Q_X$ isomorphically onto the corresponding  Luna strata of $Q_Y$.  Fix $\phi$. We demonstrate two homotopy principles  in this situation.  The first result says that if there is a $G$-diffeomorphism $\Phi\colon X\to Y$, inducing $\phi$, which is $G$-biholomorphic on the reduced fibres of the quotient mappings, then $\Phi$ is homotopic, through $G$-diffeomorphisms satisfying the same conditions, to a $G$-equivariant biholomorphism from $X$ to $Y$. The second result roughly says that if we have a $G$-homeomorphism $\Phi\colon X\to Y$  which induces a continuous family of $G$-equivariant biholomorphisms of the fibres $p_X\inv(q)$ and $p_Y\inv(\phi(q))$ for $q\in Q_X$  and if $X$ satisfies an auxiliary property (which holds for most $X$), then $\Phi$ is homotopic, through $G$-homeomorphisms satisfying the same conditions,  to a $G$-equivariant biholomorphism from $X$ to $Y$. Our results improve upon those of \cite{KLS} and use new ideas  and techniques.
\end{abstract}

\maketitle
\tableofcontents

\section{Introduction}  \label{sec:introduction}

\noindent  
Let $G$ be a reductive complex Lie group.  Let $X$ and $Y$ be  Stein manifolds (always taken to be connected) on which $G$ acts holomorphically.   We have quotient mappings $p_X\colon X\to Q_X$ and $p_Y\colon Y\to Q_Y$ where $Q_X$ and $Q_Y$ are normal Stein spaces, the categorical quotients of $X$ and $Y$. Let $q$, $q'\in Q_X$. We say that \emph{$q$ and $q'$ are in the same Luna stratum of $Q_X$\/} if the fibres $X_q=p_X\inv(q)$ and $X_{q'}=p_X\inv(q')$ are $G$-biholomorphic. The fibres are affine $G$-varieties, not necessarily reduced. The Luna strata form a locally finite stratification of $Q_X$ by locally closed smooth subvarieties.  A necessary condition for $X$ and $Y$ to be $G$-equivariantly biholomorphic is that there is a biholomorphism $\phi\colon Q_X\to Q_Y$ which preserves the Luna strata, i.e., $X_q$ is $G$-biholomorphic to $Y_{\phi(q)}$ for all $q\in Q_X$. Suppose that such a $\phi$ exists.  Our problem then is to find a $G$-equivariant biholomorphism $\Phi\colon X\to Y$ inducing $\phi\colon Q_X\to Q_Y$.  It is possible that one has made a poor choice of $\phi$ (see Example \ref{ex:badchoice}) or it could be that no choice of $\phi$ admits a lift (see Example \ref{ex:nochoice}). 

Use $\phi$ to identify the quotients, and  call the common quotient $Q$ with quotient maps $p_X\colon X\to Q$ and $p_Y\colon Y\to Q$.  
We say that \emph{$X$ and $Y$ have common quotient $Q$}. More specifically, we replace $Y$ by $\phi^*Y=\{(q,y) \in Q_X\times Y\mid  p_Y(y)=\phi(q)\}$. Then $\phi^*Y$ is a Stein $G$-manifold whose quotient mapping   is projection onto the first factor and $Q=Q_X$ is the common quotient.   Our   problem then is to find a $G$-equivariant biholomorphism $\Phi\colon X\to \phi^*Y$ which induces $\Id_Q$, the identity map of $Q$. 
So we can always reduce to the case that $X$ and $Y$ have a common quotient $Q$ and our problem is to lift $\Id_Q$ to a $G$-biholomorphism of $X$ and $Y$.   In the spirit of Gromov's work \cite{Gromov}, we show that there is a $G$-biholomorphic lift of $\Id_Q$ if there are appropriate continuous or smooth lifts of $\Id_Q$.
 
Set 
$$
\Iso(X,Y)=\prod_{q\in Q}\Iso(X_q,Y_q)
$$ 
where $\Iso(X_q,Y_q)$ denotes the set of $G$-biholomorphisms of $X_q$ and $Y_q$. Let $\pi$ denote the natural projection of $\Iso(X,Y)$ to  $Q$. Then $\Iso(X_q,Y_q)$ is a principal homogeneous space for the group $\Iso(X_q,X_q)$ and the global sections of $\Iso(X,Y)$ form a principal homogeneous space for the group of  global sections of  $\Iso(X,X)$.  In general, there is no reasonable structure of complex variety on $\Iso(X,Y)$ or $\Iso(X,X)$ (see \cite[Section 3]{KLS}).  However, we can say what the sections of $\pi$ of various kinds are. Clearly a holomorphic section of $\Iso(X,Y)$ over an open subset $U\subset Q$ should be a $G$-biholomorphism $\Phi\colon p_X\inv(U)\to p_Y\inv(U)$ inducing $\Id_U$.
 We are also able to define what a continuous section of $\Iso(X,Y)$ over $U$ is, which we call a \emph{strong $G$-homeomorphism\/} (see Section \ref{sec:strong}). 
 
 Let $\Phi\colon X\to Y$ be a $G$-diffeomorphism inducing $\Id_Q$. We  say that $\Phi$ is \emph{strict\/} if it induces a $G$-biholomorphism of $(X_q)_\red$ with $(Y_q)_\red$  for all $q\in Q$ where the subscript   means that we are considering the reduced structures on the fibres  (see Example \ref{ex:not strong}). Let $\Iso(X,Y)_\red$ denote the product of the $\Iso((X_q)_\red,(Y_q)_\red)$ with the obvious projection to $Q$. Then the smooth sections of $\Iso(X,Y)_\red$ are the strict $G$-diffeomorphisms.  A strict $G$-diffeomorphism is not necessarily a strong $G$-homeomorphism (Example \ref{ex:not strong}). Our definition of \emph{strict\/} is more general than in \cite{KLS}; see Remark \ref{rem:strict}.
 
 Here is our first main result.
 \begin{theorem}\label{thm:main1}
Let $X$ and $Y$ be Stein $G$-manifolds with common quotient $Q$. Suppose that there is a strict $G$-diffeomorphism $\Phi\colon X\to Y$. Then $\Phi$ is homotopic, through strict $G$-diffeomorphisms, to a $G$-biholomorphism from $X$ to $Y$.
\end{theorem}
The theorem says that a smooth section of $\Iso(X,Y)_\red$ is homotopic to a holomorphic section.

There is also a version of the theorem for continuous sections of $\Iso(X,Y)$, but we need an additional assumption.   Let $D$ be a vector field on $Q$. We say that $D$ is \emph{strata preserving\/} if for all Luna strata $Z$ of $Q$ and $z\in Z$, $D(z)\in T_z(Z)$. We say that $X$ has the \emph{infinitesimal lifting property\/} if every holomorphic strata preserving vector field $D$ defined on a neighbourhood $U$ of $q\in Q$ has a   lift to  a $G$-invariant holomorphic vector field $A$ on $p\inv(U')$ where $U'$ is a neighbourhood of $q$ contained in $U$. This means that $A(p^*f)=p^*(D(f))$ for all $f\in \O(U')$. The infinitesimal lifting property  really  only depends upon the isomorphism classes of the fibres of $p\colon X\to Q$; equivalently, on the slice representations of $X$ (see Section \ref{sec:background}).
For most representations of reductive groups, the infinitesimal lifting property holds (Remark \ref{rem:generic}) and for most representations all holomorphic vector fields on the quotient automatically preserve the strata \cite{Schwarz2013}.  

Here is our second main result.
\begin{theorem}\label{thm:main2}
Let $X$ and $Y$ be Stein $G$-manifolds with common quotient $Q$. Suppose that there is a strong $G$-homeomorphism $\Phi\colon X\to Y$. If $X$ has the infinitesimal lifting property, then $\Phi$ is homotopic, through strong $G$-homeomorphisms, to a $G$-biholomorphism from $X$ to $Y$.
\end{theorem}
See Section \ref{sec:strong} for the definition of a homotopy of strong $G$-homeomorphisms. The theorem says that a continuous section of $\Iso(X,Y)$ is homotopic to a holomorphic section, provided that $X$ (equivalently, $Y$) has the infinitesimal lifting property.

Our proofs of Theorems \ref{thm:main1} and \ref{thm:main2} have two steps, where we first reduce our homotopy principles to  Oka principles of the form considered by Grauert.
Let $X$, $Y$ and $Q$ be as before. We say that \emph{$X$ and $Y$ are locally $G$-biholomorphic over $Q$} if there is an open cover $\{U_i\}$ of $Q$ and $G$-biholomorphisms $\Phi_i\colon p_X\inv(U_i)\to p_Y\inv(U_i)$ inducing the identity on $U_i$. This condition says that there are no local obstructions to the existence of a global $G$-biholomorphism $\Phi\colon X\to Y$ inducing $\Id_Q$.

\begin{theorem}\label{thm:main3}
Let $X$ and $Y$ be Stein $G$-manifolds with common quotient $Q$. Suppose that one of the following holds. 
\begin{enumerate}
\item There is a strict $G$-diffeomorphism from $X$ to $Y$.
\item There is a strong $G$-homeomorphism from $X$ to $Y$  and $X$ has the infinitesimal lifting property.
\end{enumerate}
Then $X$ and $Y$ are locally $G$-biholomorphic over $Q$.
\end{theorem}

Once we have no local obstructions we are able to establish the following versions of Grauert's Oka principle.

\begin{theorem}\label{thm:main4}
Let $X$ and $Y$ be Stein $G$-manifolds locally $G$-biholomorphic over a common quotient $Q$.  
\begin{enumerate}
\item Any strict $G$-diffeomorphism $\Phi\colon X\to Y$  is homotopic, through strict $G$-diffeo\-mor\-phisms, to a $G$-biholomorphism from $X$ to $Y$.
\item  Any strong $G$-homeomorphism $\Phi\colon X\to Y$  is homotopic, through strong $G$-homeomorphisms, to a $G$-biholomorphism from $X$ to $Y$.
\end{enumerate}
\end{theorem}
Note that Theorems \ref{thm:main3} and \ref{thm:main4} establish Theorems \ref{thm:main1} and \ref{thm:main2}.
The proof of Theorem \ref{thm:main4} is along the lines of Grauert's Oka principle for principal  bundles of complex Lie groups (Section \ref{sec:Grauert}).
A main result  of our previous paper \cite{KLS} is a weaker version of Theorem \ref{thm:main4}. In   (1) and (2) we were only able to state the existence of a $G$-biholomorphism, but not that it was homotopic to $\Phi$. Also, we had to assume that $X$ (equivalently $Y$) is generic, which means that  the set of closed orbits  with trivial isotropy group is open in $X$ and that the complement (which is a closed $G$-stable subvariety of $X$) has codimension at least two.  

We briefly mention here the Linearisation Problem.
Suppose that $X=\C^n$ and that $Y$ is a $G$-module such that we have a $G$-biholomorphism of $X$ and $Y$. Then the $G$-action on $\C^n$ is linearisable, i.e., there is a biholomorphic automorphism $\Phi$ of $\C^n$ such that $\Phi\circ g\circ\Phi\inv$ is linear for every $g\in G$.
The problem of linearising actions of reductive groups on $\C^n$ has attracted much attention both in the algebraic and holomorphic settings (\cite{Huckleberry}, \cite{ Kraft1996}). 
The first counterexamples for the algebraic linearisation problem were constructed by Schwarz \cite{Schwarz1989} for $n\geq 4$.  His examples are holomorphically linearisable.  Derksen and Kutz\-sche\-bauch \cite{Derksen-Kutzschebauch} showed that for $G$ nontrivial, there is an  $N_G\in\mathbb N$ such that there are nonlinearisable actions of $G$ on $\C^n$ for all $n\geq N_G$.  Their method was to construct actions whose stratified quotients cannot be isomorphic to the stratified quotient of a linear action. In  \cite{KLSb}, we show that, most of the time,  a holomorphic $G$-action on $\C^n$ is linearisable if and only if the stratified quotient is isomorphic to the stratified quotient of a $G$-module.

Here is a brief summary of the contents of the paper. In Section \ref{sec:background} we review general results about quotients and the Luna stratification. In Section \ref{sec:strong} we recall facts about $G$-finite functions and use them to define the notion of strong $G$-homeomorphism. Section \ref{sec:examples} gives examples showing problems that can arise in finding local or global lifts of strata preserving biholomorphisms of  quotients. In Section \ref{sec:locallifting} we establish Theorem \ref{thm:main3}. Here we use two techniques: deforming an  automorphism of $Q$ to a  liftable automorphism  and lifting 
homotopies on the quotient by lifting  associated vector fields.  After establishing Theorem \ref{thm:main3} we are able to assume that $X$ and $Y$ are locally $G$-biholomorphic over $Q$. In Section \ref{sec:typeF} we define a type of $G$-diffeomorphism from $X$ to $Y$, those of type $\F$, which roughly are those $G$-diffeomorphisms inducing $\Id_Q$ whose restriction  to each fibre $X_q$ has a biholomorphic $G$-equivariant extension to a neighbourhood of $X_q$. We also define the notion of a $G$-invariant vector field on $X$ of type $\LF$. These are roughly the smooth $G$-invariant  vector fields, annihilating the $G$-invariant holomorphic functions, whose restrictions to each fibre  $X_q$ extend in a neighbourhood of $X_q$ to a $G$-invariant holomorphic vector field annihilating the $G$-invariant holomorphic functions. We establish important properties of the $G$-diffeomorphisms of type $\F$ (assuming the results of Section  \ref{sec:topologies}). In Section \ref{sec:topologies} we prove several technical results, among them the fact that the $G$-invariant vector fields of type $\LF$ are closed in the Fr\'echet space of all smooth $G$-invariant vector fields on $X$.  In Section \ref{sec:reductiontoF} we show that any strong $G$-homeomorphism from $X$ to $Y$  is homotopic, through strong $G$-homeomorphisms from $X$ to $Y$, to one of type $\F$. The analogous result for strict $G$-diffeomorphisms  follows similarly. In Sections \ref{sec:NHC} and \ref{sec:Grauert} we   modify the techniques of Cartan \cite{Cartan58} to show that any $G$-diffeomorphism from $X$ to $Y$ of type $\F$   is homotopic, through $G$-diffeomorphisms of type $\F$, to a $G$-biholomorphism from $X$ to $Y$, completing the proof of Theorem \ref{thm:main4}.  

\smallskip\noindent
\textit{Acknowledgement.}  We thank E.~Bierstone for useful discussions.

\section{Background}  \label{sec:background}

For details of what follows see \cite{Luna} and \cite[Section~6]{Snow}.  Let $X$ be a normal Stein space with a holomorphic action of a reductive complex Lie group $G$.  The categorical quotient $Q_X=X\sl G$ of $X$ by the action of $G$ is the set of closed orbits in $X$ with a reduced Stein structure that makes the quotient map $p_X\colon X\to Q_X$ the universal $G$-invariant holomorphic map from $X$ to a Stein space. When $X$ is understood, we drop the subscript $X$ in $p_X$ and $Q_X$. Since $X$ is normal, $Q$ is normal.  If $U$ is an open subset of $Q$, then 
$p^*$ induces isomorphisms of $\C$-algebras $\O_X(p^{-1}(U))^G \simeq \O_Q(U)$  and $\CC(p\inv(U))^G\simeq\CC(U)$. We say that a subset of $X$ is \textit{$G$-saturated\/} if it is a union of fibres of $p$. If $X$ is an affine  $G$-variety, then $Q$ is just the complex space corresponding to the affine algebraic variety with coordinate ring  $\O_\mathrm{alg}(X)^G$. If $V$ is a $G$-module and $p\colon V\to V\sl G$ is the quotient mapping, then the fibre $\NN(V)=p\inv(p(0))$ is the \emph{null cone of $V$\/}.

  If $Gx$ is a closed orbit in $X$, then the stabiliser (or isotropy group) $G_x$ is reductive.  We say that closed orbits $Gx$ and $G{y}$ have the same \textit{isotropy type} if $G_x$ is $G$-conjugate to $G_{y}$. Thus we get the \textit{isotropy type  stratification} of $Q$ with strata whose labels are  conjugacy classes of reductive subgroups of $G$. 

Assume that $X$ is smooth and connected, and let $Gx$ be a closed orbit. Then we can consider the \emph{slice representation\/} which is the action of $G_x$ on $T_xX/T_x(Gx)$. We say that closed orbits $Gx$ and $Gy$ have the same \emph{slice type\/} if they have the same isotropy type and, after arranging that $G_x=G_y$, the slice representations are isomorphic representations of $G_x$. The stratification by slice type is finer than that by isotropy type, but the slice type strata are unions of irreducible components of the isotropy type strata \cite[Proposition 1.2]{Schwarz1980}. Hence if the isotropy type strata are irreducible, the slice type strata and isotropy type strata are the same. This occurs for the case of a $G$-module \cite[Lemma 5.5]{Schwarz1980}. 
Let $H=G_x$ and $W=T_xX/T_x(Gx)$ be the slice representation as above. Write $W=W^H\oplus W'$ where $W'$ is an $H$-module. The Zariski tangent space to the fibre $X_{p(x)}$ at  $x$ is isomorphic to $\lie g/\lie h\oplus W'$ as $H$-module, and 
$$
\dim W^H=\dim X-\dim G+\dim H-\dim W',
$$
 so that the fibre determines the slice representation (and vice versa). Hence   the Luna stratification of the introduction is the same as the slice type stratification.  

There is a unique open stratum $Q_\pr\subset Q$, corresponding to the closed orbits with minimal stabiliser. We call this the \emph{principal stratum\/} and the closed orbits above $Q_\pr$ are called \emph{principal orbits\/}. The isotropy groups of principal orbits are called \emph{principal isotropy groups\/}. 
By definition, $X$ is generic when the principal isotropy groups are trivial and $p\inv (Q\setminus Q_\pr)$ has codimension at least two in $X$. 

\begin{remark}  \label{rem:generic}
 If $G$ is simple, then, up to isomorphism, all but finitely many $G$-modules $V$ with $V^G=0$ are generic and have the infinitesimal lifting property \cite[Corollary~11.6 (1)]{Schwarz1995}.  The same result holds for semisimple groups but one needs to assume that every irreducible component of $V$ is a faithful module for the Lie algebra   of $G$ \cite[Corollary~11.6 (2)]{Schwarz1995}.    A \lq\lq random\rq\rq\ $\C^*$-module is generic and has the infinitesimal lifting property, although infinite families of counterexamples exist.  More precisely, a faithful $n$-dimensional $\C^*$-module without zero weights is generic (and has the infinitesimal  lifting property) if and only if it has at least two positive weights and at least two negative weights and no $n-1$ weights have a common prime divisor.  Finally, $X$ is generic (or has the infinitesimal lifting property) if and only if every slice representation does. Hence these properties only depend upon  the Luna stratification of $Q$. If one is in the situation where all slice representations are orthogonal, then \cite[Theorems 3.7 and 6.7]{Schwarz1980} shows that one has the infinitesimal lifting property.
 \end{remark}

   \section{$G$-finite functions and strong homeomorphisms}\label{sec:strong}
Let $X$ and $Y$ be Stein $G$-manifolds.   Despite the fact that we can state  our main  theorems in the case that   $\phi\colon Q_X\to Q_Y$ is the identity, our proofs (especially in  Section \ref{sec:locallifting}) require us to consider the case that $\phi$ is an arbitrary strata preserving biholomorphism.
   
 If $U$ is a subset of $Q_X$, we denote $p_X\inv(U)$ by $X_U$, and $Y_U$ for $U\subset Q_Y$ is defined analogously.  The group $G$ acts on $\O(X)$, $f\mapsto g\cdot f$, where $(g\cdot f)(x)=f(g\inv x)$, $x\in X$, $g\in G$, $f\in\O(X)$.
Let $\O_\gf(X)$ denote the holomorphic functions  $f$ such that the span of $\{g\cdot f\mid g\in G\}$ is finite dimensional. They are called the  \emph{$G$-finite  holomorphic functions on $X$} and obviously form an $\O(Q)=\O(X)^G$-algebra.  If $X$ is a smooth affine $G$-variety, then the techniques of \cite[Proposition 6.8, Corollary 6.9]{Schwarz1980} show that for $U\subset Q$ open and Stein we have
$$
\O_\gf(X_U)\simeq \O(U)\otimes_{\O_\alg(Q)}\O_\alg(X).
$$

Let $H$ be a reductive subgroup of $G$ and let $B$ be an $H$-saturated neighbourhood of the origin of an $H$-module $W$. We always assume   that $B$ is Stein, in which case $B\sl H$ is also Stein. Let $G\times^HB$ (or $T_B$) denote the quotient of $G\times B$ by the (free) $H$-action sending $(g,w)$ to $(gh\inv,hw)$ for $h\in H$, $g\in G$ and $w\in B$. We denote the image of $(g,w)$ in $G\times^HB$ by $[g,w]$.  By the slice theorem,  $X$ is locally $G$-biholomorphic to  such   tubes $T_B$.   
 If $V$ is an irreducible  nontrivial $G$-module, let $\O(X)_V$ denote the elements of $\O_\gf(X)$ contained in a copy of $V$, and similarly define $\O_\alg(T_W)_V$. 
Then $\O_\alg(T_W)_V$ generates $\O(T_B)_V$ over $\O(B)^H$. By Nakayama's Lemma, $f_1,\dots,f_m\in\O(X)_V$ restrict to minimal   generators of the $\O(U)$-module $\O(X_U)_V$   for some neighbourhood $U$ of $q\in Q$  if and only if the restrictions of the $f_i$ to $X_q$ form a basis of  $\O(X_q)_V=\O_\alg(X_q)_V$.

Let $\Phi\colon X\to Y$ be a $G$-biholomorphism inducing $\phi\colon Q_X\to Q_Y$. Let $q\in Q_X$,  let $ f_1,\dots, f_m$ be  elements of $\O(X)_V$ whose restrictions to $\O(X_q)$ span $\O(X_q)_V$ and let $f_1',\dots,  f_n'$ be elements of $\O(Y)_V$ whose restrictions to $\O(Y_{\phi(q)})$ span $\O(Y_{\phi(q)})_V$. Then the $f_i$ generate $\O(X)_V$ over a neighbourhood of $q$ and the $f_j'$ generate $\O(Y)_V$ over a neighbourhood of $\phi(q)$. Over a neighbourhood $U$ of $q$  we have  $\Phi^*  f_i'=  f_i'\circ \Phi=\sum a_{ij} f_j$ where  the $a_{ij}$ are in $\O(U)\simeq\O(X_U)^G$. The $a_{ij}$ are generally not unique. However, if the $f_i$ and $f_j'$ are linearly independent when restricted to $X_q$ and $Y_{\phi(q)}$, respectively, then $m=n$ and  the matrix $(a_{ij}(q))$ is unique and  invertible. It follows that $(\Phi\inv)^*  f_i=\sum b_{ij}  f_j'$ over a neighbourhood of  $\phi(q)$ where the matrix valued function $(b_{ij})$ equals   $(a_{ij}\circ\phi\inv)\inv$.

Let $\Phi\colon X\to Y$ be  a $G$-equivariant homeomorphism inducing a strata preserving biholomorphism $\phi\colon Q_X\to Q_Y$. Let $V$ and the $ f_i$, $  f_j'$ and $q$ be as above. We say that $\Phi$ is \emph{strong for $V$ over $\phi$ at $q$\/} if $\Phi^*f_i'=\sum a_{ij}f_j$   where the $a_{ij}$   are continuous in a neighbourhood of $q$,  inducing an isomorphism $\O(Y_{\phi(q)})_V\to\O(X_q)_V$.
 Note that this condition is independent of our choice of the $f_i$ and $ f_j'$.  We say that \emph{$\Phi$ is strong over $\phi$  at $q$\/} if $\Phi$ is strong over $\phi$ for $V$ at $q$ for all irreducible nontrivial $V$. One does not actually need to worry about all $V$. Let $V_1,\dots,V_r$ be irreducible nontrivial $G$-modules such that $\O_\alg(X_q)$ is generated by  $\bigoplus_j \O_\alg(X_q)_{V_j}$. If $\Phi$ is strong over $\phi$ for the $V_j$ at $q$, then it is strong over $\phi$ at $q$. Note that if $\Phi$ is strong  over $\phi$ at $q$, then it is strong over $\phi$ at $q'$ for $q'$ sufficiently close to $q$ and that $\Phi\inv$ is strong over $\phi\inv$  at $\phi(q)$. Finally, we say that $\Phi$ is a \emph{strong $G$-homeomorphism over $\phi$\/} if it is strong over $\phi$ at  all $q\in Q_X$.  When we omit the phrase ``over $\phi$'' we mean that $Q_X=Q_Y=Q$ and $\phi=\Id_Q$ as in the introduction. The strong $G$-homeomorphisms of $X$ form a group under composition.   As we saw above, $G$-biholomorphisms inducing $\phi$ are strong.

 If $\Phi\colon X\to Y$ is a strong $G$-homeomorphism, then the various $a_{ij}(q)$ determine   $G$-isomorphisms of the  (not necessarily reduced) algebraic $G$-varieties $X_q$ and $Y_q$, so one may consider $\Phi$ as a continuous family of isomorphisms of the fibres $X_q$ and $Y_q$, $q\in Q$. 
 If $\Phi_t$ is a family of strong $G$-homeomorphisms for $t\in[0,1]$, we say that the family is a \emph{homotopy of strong $G$-homeomorphisms\/} if the corresponding $a_{ij}(t,q)$ can be chosen to be continuous in $t$ and $q$.
 
 We  consider strong $G$-homeomorphisms  to be the continuous analogues of $G$-biholo\-mor\-phisms of $X$ and $Y$. This is especially evident in the case that the automorphism group scheme associated to $X$ exists.  For $U\subset Q$ open, let $\Aut_U(X_U)^G$ denote the group of $G$-biholomorphisms of $X_U$ inducing $\Id_U$. Suppose that there is a complex space $\mathfrak G$ over $Q$ whose fibres are complex Lie  groups such that we have a canonical identification of the holomorphic sections $\Gamma(U,\mathfrak G)$ with $\Aut_U(X_U)^G$ for all $U$ open in $Q$. Then we say that  \emph{$\mathfrak G$ is the  group  scheme associated to  the Stein $G$-manifold $X$}.  Most of the time $\mathfrak G$ does not exist (see \cite[Section 3]{KLS}). However, it does exist in case  $p\colon X\to Q$ is flat (compare \cite[Chapter III, Section 2]{KraftSchReductive}).

\begin{example}\label{ex:group scheme}This example is from Kraft-Schwarz \cite[Chapter III, Section 2]{KraftSchReductive}.  Let $V=\C^2$ with coordinate functions $x$ and $y$, and let $G$ be the normaliser of the diagonal matrices in $\SL_2(\C)$. Let $s=xy$ and $t=s^2$. Then $\O_\alg(Q)=\C[t]$ (so $Q=\C$), and $\O_\alg(V)$ is generated by $\O_\alg(V)_{V^*}$ which has minimal generators $x$, $y$, $sx$ and $sy$.    An element  $\Phi\in\Aut_Q(V)^G$ sends $x$ to $\alpha x+\beta sx$ for some $\alpha$, $\beta\in \O(Q)$ and it sends $y$ to $\alpha y-\beta sy$. Then $\Phi^*(s)=(\alpha^2-\beta^2 t)s$ where $\Phi^*(s^2)=s^2$. Hence   $\alpha^2-\beta^2 t=\pm 1$. Conversely, given $\alpha$, $\beta\in\O(Q)$ such that $\alpha^2-\beta^2 t=\pm 1$, there is a corresponding $\Phi\in\Aut_Q(V)^G$ with $\Phi^*x=(\alpha+s\beta)x$ and $\Phi^*y=(\alpha-s\beta) y$. We can see our automorphisms  as sections of a group scheme $\mathfrak G$, as follows.
As variety, 
$$
\mathfrak G=\{(t,a,b)\in\C^3\mid a^2-b^2t=\pm 1\},
$$ 
with projection $\pi\colon \mathfrak G\to\C$ sending $(t,a,b)$  to $t$. The group structure on $\Aut_Q(V)^G$ induces a group structure on the fibres of $\pi$:
$$
(t,a',b')\cdot (t,a,b)=(t,aa'+\epsilon bb't,a'b+\epsilon ab') \text{ where }\epsilon=a^2-b^2t.
$$
The inverse to $(t,a,b)$ is $(t,\epsilon a,-b)$.
The group $\Aut_Q(V)^G$ is isomorphic to the group of holomorphic sections of $\pi\colon \mathfrak G\to \C$. The strong $G$-homeomorphisms of $V$ are the same thing as the continuous sections of $\pi$.  
\end{example}
\begin{example}\label{ex:not strong}
Let $(V,G)$ be as in the previous example. Choose the branch of the square root in a neighbourhood of $1\in \C$ with $\sqrt 1=1$ and let $a(t)=\sqrt{1+\bar t}$, $t\in \C$. Then   $a(t)$ is smooth in a neighbourhood of $0\in\C$. Let $\Phi$ be the   $G$-diffeomorphism   which sends $x$ to $(a(t) +\bar s)x$ and $y$ to $(a(t)-\bar s)y$ for $t$ near $0$. Then the corresponding $b(t)$ is $\bar t/|t|=\bar s/s$ which has no limit at $t=0$. The fibre of $V$ over $0$ is not reduced, and the reduced fibre is $\{xy=0\}$ on which the restriction of $\Phi$ is the identity. Thus $\Phi$ is a strict $G$-diffeomorphism, but it is not a strong $G$-homeomorphism.
\end{example}

We now establish some differentiability properties of strong $G$-homeomorphisms inducing a strata preserving biholomorphism $\phi$. Our results are a generalisation of \cite[Lemma 24]{KLS} (see Remark \ref{rem:mistake} below).

Let $H$, $W$, $B$, etc.\ be as in the beginning of this section. We always assume that $B$ is stable under multiplication by $t\in[0,1]$. 
Let $f\in\O_\alg(T_W)$. We say that $f$ has degree $n$ if $f([g,tw])=t^nf([g,w])$, $g\in G$, $w\in W$, $t\in\C$.  We denote the elements of degree $n$ by $\O_\alg(T_W)_n$. The elements of degree zero are the pullbacks to $\O_\mathrm{alg}(T_W)$ of the elements of $\O_\mathrm{alg}(Z)$ where $Z\simeq G/H$ is the zero section of $T_W$.  For the moment, we also assume that $W^H=0$.     

By   \cite[Lemma 23]{KLS} we have
 \begin{lemma}  \label{lem:gens}
$\O_\mathrm{alg}(T_W)$ is generated by $\O_\mathrm{alg}(T_W)_1$ as an $\O_\mathrm{alg}(G/H)$-algebra.
\end{lemma}

 Let $V_1,\dots,V_r$ be  nonisomorphic irreducible  nontrivial  $G$-modules which appear in  $\O_\alg(T_W)_0$ and $\O_\alg(T_W)_1$ such that $\O_\alg(T_W)$ is generated by  
 $$
 M= \O_\alg(T_W)_{V_1}\oplus\cdots\oplus \O_\alg(T_W)_{V_r}.
 $$
 We may assume that $V_1,\dots,V_r$ are minimal with the above property. Let $f_1,\dots,f_n$ minimally generate $M$ where each $f_i$ is homogeneous and in $\O_\alg(T_W)_{V_j}$ for some $j$.
  \begin{definition}\label{def:standard-generators}
 Let $V_1,\dots,V_r$ and $f_1,\dots,f_n$ be as above.  We call $\{f_i\}$  a \emph{standard set of   generators of $\O_\gf(T_B)$\/}  or a \emph{standard set of   generators of $\O_\alg(T_W)$}. 
  \end{definition}
The span of the $f_i$ is $G$-stable and the $f_i$ are linearly independent on $F=G\times^H\NN(W)$.  Let $d_i$ denote the degree of $f_i$, $i=1,\dots,n$. We arrange that $d_i=0$   for $ i\leq  \ell $, $d_i=1$   for $ \ell < i\leq m$ and   $d_i>1$ for $i>m$.
 
  Let $X$ denote $T_B$ so that $Q=T_B\sl G\simeq B\sl H$. Let $q_0\in Q$ be the image of the  point $x_0=[e,0]\in X$. Let $U$ be a 
  neighbourhood of $q_0$ and let $\phi\colon U\to \phi(U)\subset Q$ be a strata preserving biholomorphism. Then $\phi(q_0)=q_0$.  
Let $\Phi\colon X_U\to X_{\phi(U)}\subset X$ be a strong $G$-homeomorphism over  $\phi$.  

Shrinking  $U$ we may assume that there is an $(n\times n)$-matrix $(a_{ij})$   of continuous invariant functions
such that  $\Phi^*f_i=\sum a_{ij} f_j$. 
Since $\Phi$ preserves $F$, it preserves the closed orbit $Z\subset F$. 
Let $\I$ denote the ideal in $\CC(X_U)^G$  generated by $\bigoplus_{k\geq 2}\O_\alg(T_W)_k^G$. Note that since $W^H=0$, $\O_\alg(T_W)^G_1=0$. 

 \begin{lemma}\label{lem:helpfundam}
\begin{enumerate}
\item The matrix $(a_{ij}(x_0))$ is invertible, $1\leq i$, $j\leq  \ell $.
\item Perhaps shrinking $U$, we have  $a_{ij}\in\I$ for $\ell < i\leq m$ and  $1\leq j\leq  \ell$. 
\end{enumerate}
\end{lemma}

 \begin{proof}
 If we restrict $f_1,\dots,f_\ell$ and the $\Phi^*f_i$ to $Z$, then $\Phi^*$ is an isomorphism, with matrix $(a_{ij}(x_0))$,  and we have (1). 
 
 Let $\rho$ denote the Reynolds operator, i.e., the $G$-equivariant projection from $\O_\alg(T_W)$ to $\O_\alg(T_W)^G$. It extends to a projection of $\CC(X_U)^G\cdot\O_\alg(T_W)$ to $\CC(X_U)^G$.
  Let us assume that  some $a_{ij}\neq 0$ for $\ell < i\leq m$ and  $1\leq j\leq  \ell$, say   $a_{m 1}\neq 0$. Then $f_ m $ and $f_1$ correspond to the same irreducible $G$-module $V$. We may assume that $f_1,\dots,f_s$ are the polynomials corresponding to $V$ among $f_1,\dots,f_\ell$. Then $a_{mj}=0$ for $s<j\leq \ell$.  Now $V$ occurs in $\O_\alg(Z)\simeq\O_\alg(T_W)_0$ and we have a non-degenerate pairing of $\O_\alg(Z)$ with itself, which sends a pair of functions $f$, $f'$ to   $\rho(f\cdot f')$. Thus $\O_\alg(Z)_{V^*}$ has  a basis   $f_1',\dots,f_s'$ of cardinality $s$. We may assume that $\rho(f_i'\cdot f_j)=\delta_{ij}$ for $1\leq i$, $j\leq s$. There is a homogeneous minimal generating set $f_1',\dots,f_s'$, $f_{s+1}',\dots,f_{s+t}'$ of the $\O_\alg(T_W)^G$-module $\O_\alg(T_W)_{V^*}$. Then $\deg f'_i>0$ for $i>s$. There are $b_{ij}\in\CC(X_U)^G$ such that $\Phi^*f_i'=\sum b_{ij}f_j'$. Let  $1\leq  i\leq s$. Since $f_m$ has degree $1$, $\rho(f_i'\cdot f_m)=0$.    Thus   
    $$
0= \rho(\Phi^*f_i'\cdot\Phi^*f_m)= \sum_{j,\, k} b_{ij}a_{mk}\rho(f_j'\cdot f_k)\in\sum_{j,\, k=1}^s b_{ij}a_{mk}\rho(f_j'\cdot f_k)+\I=\sum_{j=1}^s b_{ij}a_{mj}+\I
 $$
 Applying (1) to the matrix $b_{ij}$, $1\leq i$, $j\leq s$, we see that $(b_{ij}(x_0))$ is invertible. In a $G$-saturated neighbourhood of $x_0$ we can invert $(b_{ij})$ and then the equation above shows that $a_{mj}\in\I$, $j=1,\dots,s$. This gives (2).
 \end{proof}
 
\begin{remark}\label{rem:mistake}
In the proof of \cite[Lemma 24]{KLS} we tacitly  assumed that $\Phi^*$ sends the $f_i$ of degree 1 to    $\sum_j a_{ij}f_j$ where $j> \ell $, i.e., to terms only involving the $f_j$ of degree at least 1. This, of course, is not justified.   The lemma   above  is what is required. Lemma \ref{lem:fundamental} and Corollary \ref{cor:fundam} below replace  \cite[Lemma 24]{KLS}.
\end{remark}

We may assume that $U$ is the image of a tube $T_{B'}$. Then we have a scalar action of $[0,1]$ on $X_U$ as follows. Let $x=[g,w]\in X_U$. Then $t\cdot [g,w]=[g,tw]$.    Let $\Phi_t(x)=t\inv\cdot \Phi(t\cdot x)$, $x\in X_U$, $t\in(0,1]$. Let   $W_g$ denote the fibre of $T_W$ at $[g,0]\in Z$. Let $a_{ij}^t(x)$ denote $t^{d_j-d_i}a_{ij}(t\cdot x)$, $t\in(0,1]$, $x\in X_U$, $1\leq i$, $j\leq n$.
Finally, let $\Aut_\vb(T_W)^G$ denote the complex 
algebraic group of $G$-vector bundle automorphisms of $T_W$ (see Corollary \ref{cor:Lvbreductive}).

   \begin{lemma}  \label{lem:fundamental}
The following hold.
\begin{enumerate}
\item  $(\Phi_t^*f_i)(x)=\sum_j  a^t_{ij}(x) f_j(x)$, $x\in X_U$, $t\in(0,1]$, $1\leq i \leq n$. 
\item  The limit as $t\to 0$ of $\Phi_t$ acting on the $f_i$, $\ell< i\leq  m$, is given by the matrix $L$ with entries $a_{ij}(x_0)$, $\ell< i$, $j\leq  m $.
\item  $\Phi$ has a normal derivative $\delta\Phi $ along $Z$, and $\delta\Phi\in\Aut_\vb(T_W)^G$. If $\phi$ is the identity, then $\delta\Phi$ fixes $\O_\alg(T_W)^G$.
\item The limit as $t\to 0$ of the $\Phi_t$ is $\delta\Phi$, uniformly on compact subsets of $X_U$.
 \end{enumerate}
\end{lemma}

\begin{proof} For (1) we compute that
\begin{align*}
(\Phi_t^*f_i)(x)&=f_i(t\inv\Phi(t\cdot x))=t^{-d_i}(\Phi^*f_i)(t\cdot x)=t^{-d_i}\sum_j a_{ij}(t\cdot x)f_j(t\cdot x) \\ &=\sum_j t^{d_j-d_i}a_{ij}(t\cdot x)f_j(x). \end{align*}
  Now let $\ell< i\leq  m $, so that $d_i=1$.  By Lemma \ref{lem:helpfundam}, for $1\leq j\leq \ell$, $a_{ij}(t\cdot x)=t^2d_{ij}(t,x)$ where $d_{ij}$ is continuous. Hence, uniformly on compact subsets of $X_U$,
\begin{equation} \label{eq:limit}
\lim_{t\to 0}\Phi_t^*f_i(x)=\lim_{t\to 0}\sum_{j>\ell}  t^{d_j-1}a_{ij}(t\cdot x) f_j(x)=\sum\limits_{\ell< j\leq  m }a_{ij}(x_0)f_j(x),  
\end{equation}
giving (2).
For now, just consider indices $i$ between $\ell+1$ and $m$. We have $\Phi(x_0)=[g,0]$ for some $g\in N_G(H)$, and $\Phi_t(x_0)=[g,0]$ for all $t\in(0,1]$. The $f_i$  are sections of $T_{W^*}$ and they span each fibre of $T_{W^*}$. If $f=\sum  c_if_i$, $c_i\in \C$,   vanishes at $[g,0]$, then $f\circ\Phi_t(x_0)=0$ for all $t\in(0,1]$. By \eqref{eq:limit}, $\sum  c_ia_{ij}(x_0)f_j(x_0)=0$. Thus the matrix $L$ induces a linear mapping from $W^*_{g}$ to $W^*_e$. Let $\delta\Phi(x_0)\colon W_e\to W_{g}$ be the dual mapping. It follows easily from \eqref{eq:limit} that $\delta\Phi(x_0)$ is the normal derivative of $\Phi$ at $x_0$, and the analogous fact at other points of $Z$ follows by equivariance of $\Phi$.
Clearly, $\delta\Phi\in\Aut_\vb(T_W)^G$. If $\phi$ is the identity, then the $\Phi_t$ and $\delta\Phi$ fix the invariants. Hence we have established 
(3).  

We have shown that   $a_{ij}^t$ converges uniformly as $t\to 0$ on compact subsets of $X_U$ for $\ell<i\leq m$, and this is trivially true for $1\leq i\leq\ell$. Since the $f_i$ for $i\leq m$ generate $\O_\alg(T_W)$, a subset of them gives local coordinates at any point of $T_W$. Hence $\Phi_t$ converges to $\delta\Phi$ uniformly on compact subsets of $X_U$ as $t\to 0$ and we have (4).
 \end{proof}
 
It is not clear that $\lim\limits_{t\to 0}a_{ij}^t(x)$ exists for all $i$ and $j$. To remedy this we replace   $(a_{ij})$ by a matrix $(b_{ij})$ for which $\lim\limits_{t\to 0}b_{ij}^t(x)$ does exist.

\begin{lemma}\label{lem:Puv}
Let $m<u\leq n$. For $1\leq v\leq n$ there are polynomials 
$$
P_{uv}(x,z_{ij})\in\O(T_W)^G[z_{11},\dots z_{1n},\dots,z_{m1},\dots,z_{mn}]
$$
such that  
$$
\Phi^*f_u =\sum_{v}P_{uv}(x,a_{ij})f_v
$$
for any strong $G$-homeomorphism $\Phi$ with associated matrix $(a_{ij})$.
 \end{lemma}

\begin{proof}
By Lemma \ref{lem:gens}, for $m<u\leq n$, $f_u= P_u(f_1,\dots,f_ m )$, where $P_u(z_1,\dots,z_m)$ is a polynomial with constant coefficients. We may assume that each monomial in $P_u(z_1,\dots,z_m)$ is homogeneous of degree $d_u$ where we give $z_i$ degree 0 for $i\leq \ell$ and degree 1 for $i>\ell$. Now $\Phi^*f_u=P_u(\Phi^*f_1,\dots,\Phi^*f_m)$. Replace   $\Phi^*f_i$ by $\sum a_{ij}f_j$, $1\leq i\leq m$. Expanding 
and re-expressing in terms of invariants times the $f_v$ we obtain
$$
\Phi^*f_u=\sum_v P_{uv}(x,a_{11},\dots,a_{1n},\dots,a_{m1},\dots,a_{mn})f_v
$$
where $P_{uv}$ is independent of the $a_{ij}$ (which we may treat as indeterminants). 
\end{proof}

\begin{remark}
By equivariance of $\Phi$, some of the $a_{ij}$ are necessarily zero.
\end{remark}

We break up $P_{uv}(x,a_{11},\dots,a_{mn})$ into a sum  of terms $P_{uv\alpha\beta}(x,a_{11},\dots,a_{mn})$ as follows.
Let  $\tau=h(x) a_{11}^{\gamma_{11}}\cdots a_{mn}^{\gamma_{mn}}$ be a term in $P_{uv}(x,a_{ij})$. Let $\alpha_0(\tau)$ be the sum of the $\gamma_{ij}$ for which $d_j-d_i=-1$ and let $\beta_0(\tau)$ be the sum of the $\gamma_{ij}(d_j-d_i)$ for which $d_j-d_i>0$. Then $h(x)$ is homogeneous of degree $d_u-d_v+\beta_0(\tau)-\alpha_0(\tau)$. For $\alpha$, $\beta\in\Z^+$ let $P_{uv\alpha\beta}$ be the sum of the terms $\tau$ in $P_{uv}$ for which $\alpha_0(\tau)=\alpha$ and $\beta_0(\tau)=\beta$. Then $P_{uv\alpha\beta}\in \O_\alg(T_W)^G_{d_u-d_v+\beta-\alpha}[a_{11},\dots,a_{mn}]$.
From Lemma \ref{lem:helpfundam} we know that $a_{ij}(t\cdot x)=t^2\tilde a_{ij}(t,x)$, $\ell<i\leq m$, $1\leq j\leq \ell$, where   $\tilde a_{ij}$ is continuous and invariant.  Let $\tilde a_{ij}(t,x)=a_{ij}(t\cdot x)$ for  $1\leq i\leq \ell$ or $\ell<i\leq m$ and $\ell<j\leq n$. One easily shows:
 \begin{lemma}\label{lem:limit-P-alpha-beta}
 Let $\alpha$, $\beta\in\Z^+$. Then 
$$
t^{d_v-d_u}P_{uv\alpha\beta}(t\cdot x,a_{ij}(t\cdot x))=t^{\alpha+\beta}P_{uv\alpha\beta}(x,\tilde a_{ij}(t,x))
$$
is 
jointly continuous in $t\in[0,1]$ and $x\in X_U$. Moreover, uniformly on compact subsets of $X_U$ we have
$$
\lim_{t\to 0} t^{d_v-d_u}P_{uv\alpha\beta}(t\cdot x,a_{ij}(t\cdot x))= \begin{cases}
0&\alpha+\beta>0\\
P_{uv00}(x,\delta_{d_id_j}a_{ij}(x_0))&\alpha+\beta=0 
\end{cases}
$$
where $P_{uv00}(x,\delta_{d_id_j}a_{ij}(x_0))\in\O_\alg(T_W)^G_{d_u-d_v}$.
 \end{lemma}
\begin{corollary}\label{cor:fundam}
Let $b_{ij}(x)$ denote $P_{ij}(x,a_{11}(x),\dots,a_{mn}(x))$ for $i>m$, $1\leq j\leq n$ and let $b_{ij}(x)=a_{ij}(x)$ for $i\leq m$, $1\leq j\leq n$. Then $\Phi^*f_i=\sum b_{ij}f_j$, $1\leq i\leq n$. The $b_{ij}^t(x)$ are 
jointly continuous in $t$ and $x$ and $(\delta\Phi)^*f_i=\sum b_{ij}^0(x)f_j$ where $b^0_{ij}(x)\in\O_\alg(T_W)^G_{d_i-d_j}$, $1\leq i$, $j\leq n$.
Hence $\Phi_t$ is a homotopy of strong $G$-homeomorphisms where $\Phi_0$ is $G$-biholomorphic.
 \end{corollary}

 We have been considering a special case, where the slice representation of the closed orbit $Gx_0$ has no nonzero fixed vectors. Now we consider the general case. So the slice representation is of the form $(\C^d\oplus W,H)$ where $H$ acts trivially on $\C^d$ and $W^H=0$. Then the slice theorem says that a neighbourhood of our closed orbit   is $G$-biholomorphic to $S\times T_B$ where $T_B$ is a tube as before and $S$ is an open subset of $\C^d$ which we may take to be Stein and connected. 

\begin{definition}\label{def:standard-neighbourhood}
Let $S$, $W$, $H$ and $B$ be as above. We call $S\times T_B$ a \emph{standard neighbourhood in $X$} or a  \emph{standard open set in $X$}.
\end{definition}

Note that our standard generators $\{f_i\}$ of $\O_\alg(T_W)$, considered as functions on $S\times T_B$, generate $\O_\gf(S\times T_B)$ as an $\O(S\times T_B)^G$-algebra. Thus we also call $\{f_i\}$ a standard generating set  for $\O_\gf(S\times T_B)$.   Now let $s_0\in S$ and consider a $G$-saturated neighbourhood of $(s_0,x_0)$ of the form $S_0\times T_{B'}$ where $B'\subset B$. Let $U$ denote $p(S_0\times T_{B'})$. Then $U$ is a neighbourhood of $q_0=p(s_0,x_0)$. Suppose that  we have a strata preserving biholomorphism $\psi\colon U\to\psi(U)\subset S\times T_B$  where $\psi(q_0)=q_0$. We may write $\psi$ as $(\theta,\phi)$ where $\theta$ maps to $S$ and $\phi$ to $T_B$. Further suppose that we have a strong $G$-homeomorphism  $\Psi=(\Theta,\Phi)\colon S_0\times T_{B'}\to S\times T_B$ inducing $\psi$. Then $\Theta(s,x)=\theta(s,p_{B'}(x))$, $s\in S_0$, $x\in T_{B'}$. Hence $\Theta$ is holomorphic. Since $\Psi$ is strong,   $\Phi^*f_i=\sum a_{ij}(s,x)f_j(x)$ where the $a_{ij}$ are continuous and invariant. For $t\in[0,1]$ and $(s,x)\in S\times T_B$ let $t\cdot(s,x)$ denote $(s,t\cdot x)$. Then we have the 
homotopy   $\Psi_t(s,x)=(\Theta_t(s,x),\Phi_t(s,x))$ where $\Theta_t(s,x)=\Theta(s,t\cdot x)$. Define $\Psi^s(x)=\Psi(s,x)$ for $(s,x)\in S_0\times T_{B'}$, and similarly define $\Phi^s$ and $\Theta^s$.
Then applying Lemma \ref{lem:fundamental} and Corollary \ref{cor:fundam} and perhaps shrinking our neighbourhoods  we have the following:

\begin{corollary}\label{cor:fundam-with-S}
Let $U$, $\Psi$, etc.\ be as above.  
\begin{enumerate}
\item  The limit as $t\to 0$ of $\Phi^s_t$ acting on the $f_i$, $\ell< i\leq  m $, is given by the matrix $L(s)$ with entries $a_{ij}(s,x_0)$, $\ell< i$, $j\leq  m $.
\item  $\Phi^s$ has a normal derivative $\delta\Phi^s$ along $Z$, and $\delta\Phi^s\in\Aut_\vb(T_W)^G$ depends continuously on $s$. If $\phi$ is the identity, then $\delta\Phi^s$ fixes $\O_\alg(T_W)^G$  for all $s$.
\item The limit as $t\to 0$ of the $\Psi_t^s$ is $(\Theta(s,x_0),\delta\Phi^s)$, uniformly on compact subsets of $S_0\times T_{B'}$.
\item  There are invariant $b_{ij}(s,x)$ such that $b_{ij}=a_{ij}$ for $i\leq m$, $\Phi^*f_i=\sum b_{ij}f_j$ for $1\leq i\leq n$, and the $b_{ij}^t(s,x)$ are 
jointly continuous in $t\in[0,1]$, $s$ and $x$. Moreover,  
$
(\delta\Phi)^*f_i(s,x) =\sum b_{ij}^0(s,x)f_j(x)
$ 
where $b^0_{ij}(s,x)$ 
lies in $\CC(S_0)\otimes\O_\alg(T_W)^G_{d_i-d_j}$, $1\leq i$, $j\leq n$. 
\end{enumerate}
\end{corollary}  

We usually denote $\delta\Phi$ by $\Phi_0$. Then $\Psi_t$ has limit $\Psi_0=(\Theta_0,\Phi_0)$ where $\Theta_0(s,x)=\Theta(s,x_0)$ for $(s,x)\in S_0\times T_{B'}$.

\begin{corollary}\label{cor:homotopy-strong}
Let $\Psi=(\Theta,\Phi)$ be a strong $G$-homeomorphism inducing $\psi=(\theta,\phi)$ as above. Then the family $t\mapsto\Psi_t$ is a homotopy of strong $G$-homeomorphisms inducing a family of biholomorphisms $\psi_t$. There is a continuous 
map $\sigma\colon S\to\Aut_\vb(T_W)^G$ such that $\Phi_0(s,x)=\sigma(s)(x)$ for $(s,x)\in S_0\times T_{B'}$. Moreover,  $\Theta_0(s,x)=\Theta(s,x_0)$ is a holomorphic map of $S_0\times T_{B'}$ into $S$ which is   biholomorphic when restricted to $S_0\times\{x_0\}$. If $\psi$ is the identity, then   $\Psi_0$ fixes $\O(S_0\times T_{B'})^G$.
\end{corollary}

\begin{remark}\label{rem:fundam}
Suppose that the $a_{ij}(s,x)$ are smooth (resp.\ holomorphic). Then $\Psi$ is smooth (resp.\ holomorphic). By Lemma \ref{lem:helpfundam} we have:

\begin{enumerate}
\item For $\ell<i\leq m$ and $1\leq j\leq \ell$, there are smooth  functions $\tilde a_{ij}(t,s,x)$ (resp.\ holomorphic in $s$ and $x$) such that   $a_{ij}(s,t\cdot x)=t^2\tilde a_{ij}(t,s,x)$.    \end{enumerate}
Let $\tilde a_{ij}(t,s,x)=a_{ij}(s,t\cdot x)$ for $1\leq i\leq\ell$ or  $\ell<j\leq n$.
Let the $P_{uv}(x,a_{ij})$ be the polynomials of Lemma \ref{lem:Puv}. For $1\leq v\leq n$, set $b_{uv}=a_{uv}$ if $u\leq m$ and set $b_{uv}(s,x)=P_{uv}(x,a_{ij}(s,x))$ for $u>m$. Then $\Phi^*f_i=\sum b_{ij}f_j$, $1\leq i\leq n$. The functions $b_{ij}^t(s,x)$ are polynomials in $t$, $x$ and the $\tilde a_{ij}(t,s,x)$. Hence:
\begin{enumerate}
\addtocounter{enumi}{1}
\item The functions $b^t_{ij}(s,x)$ are smooth   in $t$, $s$ and $x$  (resp.\ holomorphic in $s$ and $x$) and we have $\Phi_t^*f_i(s,x)=\sum b_{ij}^t(s,x)f_j(x)$, $1\leq i$, $j\leq n$.
\end{enumerate}
\end{remark}

\section{Examples}\label{sec:examples}

First we exhibit a biholomorphism  of quotients which does not have local $G$-equivari\-ant lifts.
\begin{example}\label{ex:badchoice}
(See \cite[Example 4.4]{Schwarz2014}.)\ Let $G=\SL_2(\C)$ and $H=\SL_4(\C)$ and consider the canonical action of $G\times H$ on $V=(\C^4)^*\otimes\C^2$. Then the $G$-invariant functions of $V$ are generated by determinants $d_{ij}$, $1\leq i<j\leq 4$, with the relation
$$
d_{12}d_{34}-d_{13}d_{24}+d_{14}d_{23}=0.
$$
We may identify the quotient with  the set of  $2$-forms  $\omega=\sum d_{ij}e_i\wedge e_j\in\wedge^2(\C^4)$ with the property that $\omega\wedge\omega\in\wedge^4(\C^4)\simeq\C$ vanishes. Now $\omega\mapsto\omega\wedge\omega\in\C$ is the $\SO_6(\C)\simeq(\SL_4(\C)/\{\pm I\})$-invariant bilinear form on $\C^6\simeq\wedge^2(\C^4)$. Hence $Q$ can be identified with the null cone of the action of $L^0=\SO_6(\C)$ on $\C^6$. There is an action of $L=\Orth_6(\C)$ on $Q$ as well  and $L\setminus L^0$ acts by outer automorphisms on $L^0$. Suppose that some $\ell
\in L\setminus L^0$ has a local biholomorphic lift $\Phi$ to $V$. Then $\Phi'(0)$ is a lift of $\ell$, induces an outer automorphism of $H$ and is in the normaliser of $G$ \cite[Proposition 2.9]{Schwarz2014}. But $G$ has no outer automorphisms, hence changing $\Phi'(0)$ by an element of $G$ one can assume that $\Phi'(0)$ lies in $\GL(V)^G=\GL_4(\C)$.  But no element of $\GL_4(\C)$ induces an outer automorphism of $H$. Hence we have a contradiction, and $\ell$ has no local lift.
\end{example}

By modifying the example above, we find $X$ and $Y$ such that $Q_X$ and $Q_Y$ have a unique strata preserving biholomorphism which does not have local lifts (let alone a global lift).

\begin{example}\label{ex:nochoice}
 Let $V$, $G$ and $Q$ be as above. Let $Q_0=Q\cap \Delta$ where $\Delta$ is the open ball of radius one in $\C^6$. Since  $\Delta$ is hyperbolic, so is $Q_0$, and this implies that the automorphism group of $Q_0$ is a real Lie group $H$ with the property that every isotropy group is compact \cite[Theorem 5.4.2]{Kobayashi}. Since the origin is the  unique singular point of $Q_0$, it is fixed by $H$ and  $H$  is compact.   We show that $H=\Orth_6(\R)$. We also show that there is a homogeneous complex polynomial $f$ of degree 3 on $\C^6$ which does not vanish everywhere on $Q_0$ and is fixed only by the identity of $H$. Let $0\neq z_0\in Q_0$ such that $f(z_0)=c\neq 0$ and let $Q_0'$ denote the complement 
in $Q_0$ of $\{z\in Q_0\mid f(z)=c\}$. Any holomorphic automorphism $\phi$ of $Q_0'$ extends to $Q_0$ and is an element $h\in H$. Suppose that  $\{z\in Q_0\mid (h\cdot f)(z)=c\}=\{z\in Q_0\mid f(z)=c\}$ (and $h\neq e$). Then $\{z\in Q_0\mid f(z)=c\}$ is a union of  irreducible components of $\{z\in Q_0\mid (h\cdot f)(z)-f(z)=0\}$ which are cones in $Q$ passing through the origin. Then  $f(0)=c$, which is absurd. Thus $h=e$. Now let $X$ denote $p\inv(Q_0')\subset V$ and let $Y$ denote $p\inv(\phi(Q_0'))$ where $\phi\in\Orth_6(\R)\setminus\SO_6(\R)$. Note that $\phi$ preserves $Q_0$. Now we have a unique biholomorphism $\phi\colon Q_X=Q_0'\to Q_Y=\phi(Q_0')$ and by the previous example there is no lift of $\phi$ near the origins.   
 
 We show that $H=\Orth_6(\R)$. We have the representation of $H$ on the Zariski tangent space of $Q_0$ at $0$, which is $\C^6$.  Let $G=H_\C$ be the complexification of $H$. Then $G$ acts on $Q$. By the slice theorem, the kernel of $G\to\GL(T_0(Q))=
 \GL_6(\C)$ acts trivially in a neighbourhood of $0$, hence trivially on $Q$. But $G$ has to act faithfully on $Q$ since $H$ does. Hence 
 $G\to\GL_6(\C)$ is faithful and $G$ acts linearly on $\C^6$, and it clearly has to lie in $\Orth_6(\C)$. Now the maximal compact subgroups of $\Orth_6(\C)$ are all conjugate to $\Orth_6(\R)$ where $\Orth_6(\R)\subset H$. Hence $H=\Orth_6(\R)$. 
  
 We now show that there is an $f$ as claimed above. As a real representation of $H$, the space of polynomials of degree three contains two copies of $S^3(\R^6)$. We show that the principal isotropy group for the action on one copy of $S^3(\R^6)$ is trivial. The isotropy group of the vector $e_1^3$ in  $S^3(\R^6)$ is a copy of 
 $\Orth_5(\R)$ whose slice representation is $S^3(\R^5)\oplus S^2(\R^5)\oplus \R$ with trivial action on $\R$ and the obvious actions on $S^3(\R^5)$ and $S^2(\R^5)$. The principal isotropy group of $S^2(\R^5)$ is finite (a product of copies of $\Z/2\Z$), and its action on $S^3(\R^5)$ is faithful. Hence the principal isotropy group is trivial. This means that there is an open set of homogenous polynomials $f$ of degree 3 on $\C^6$ whose $H$-isotropy is trivial, and we can choose such an $f$ which does not vanish on $Q_0$. Then as above one constructs a Stein open set $Q_0'\subset Q_0$ with trivial holomorphic automorphism group.
  \end{example}
 
\begin{example}\label{ex:badchoice2}
Let $G=\C^*$ and $V=\C^4$ with basis $\{a_1,b_1,a_2,b_2\}$ where the $a_i$ have weight 1 and the $b_i$ have weight $-1$.
Let $\Phi\colon V\to V$ sending $(a_1,b_1,a_2,b_2)$ to $(b_1,a_1,b_2,a_2)$. Then 
$\Phi(\lambda v)=\lambda\inv\Phi(v)$ for $\lambda\in G$ and $v\in V$. Hence  $\Phi$ is not equivariant in the usual sense. Let $\phi$ be the automorphism of $Q$ induced by $\Phi$. If $\Psi$ is a lift of $\phi$, then so is its derivative $\Psi'(0)$ and one can show that $\Psi'(0)$ has to have the same equivariance property as $\Phi$ does. Hence $\phi$ has a lift, but there is no lift which is equivariant in the usual sense.
\end{example}

The $G$-modules in the examples above have the infinitesimal lifting property. Here is an 
example where this property fails.
\begin{example}
This is due to H.\ M.\ Meyer. Let $G=\C^*$ act on $V=\C^3$ with coordinate functions $x$, $y$, $z$ of weights $-1$, $1$, $2$, respectively. Let $u=xy$ and $w=x^2z$. Then $u$ and  $w$ are coordinate functions on $Q=\C^2$ and the Luna strata are $\{0\}$ and $\C^2\setminus \{0\}$. The vector field $u\pt/\pt w$ is strata preserving. A lift would have to send $x^2z$ to $xy$ but no smooth vector field with this property exists.
\end{example}

 \section{Local lifting of automorphisms} \label{sec:locallifting}

 Let $X$ and $Y$ be Stein $G$-manifolds with common quotient $Q$. We establish Theorem \ref{thm:main3} which gives sufficient conditions for the existence of local $G$-biholomorphic lifts of $\Id_Q$.  The main idea behind the proof is the lifting of vector fields defined on the quotients. We begin with results about this problem. For the moment we only deal with $X$.

For $U$ open in $Q$, let $\Der(U)$ denote the holomorphic vector fields on $U$, i.e., the derivations of $\O(U)$. Let $\Der(X_U)$ denote the holomorphic vector fields on $X_U$. Then we have an $\O(U)$-module homomorphism   $p_*\colon\Der(X_U)^G\to\Der(U)$ which is just restriction of $A\in \Der(X_U)^G$ to $\O(X_U)^G\simeq\O(U)$. If $D=p_*A$, we say that \emph{$A$ is a lift of $D$}. 

\begin{remark}\label{rem:llifting-property}
Recall that an element $D\in \Der(U)$ is strata preserving if for each Luna stratum $Z$ of $U$, $D(z)\in T_zZ$ for every $z\in Z$. Equivalently, $D$ preserves the ideals of the closures of the strata. 
Recall that  $X$ has the infinitesimal lifting property if any strata preserving $D\in\Der(U)$ has local lifts to $\Der(X_U)^G$. It is easy to see that if $A\in\Der(X_U)^G$, then $p_*A$ is strata preserving. 
Hence, when we have the infinitesimal lifting property, the locally  liftable vector fields are precisely the strata preserving vector fields.  
\end{remark}

\begin{lemma}\label{lem:sheaves-derivations}
\begin{enumerate}
\item The sheaves of $\O_Q$-modules $U\mapsto \Der(X_U)^G$ and $U\mapsto\Der(U)$   are  coherent. 
\item Suppose that $U\subset Q$ is Stein, $D\in\Der(U)$ and $\{U_i\}$ is an open cover of $U$ such that $D|_{U_i}$ lifts to $\Der(X_{U_i})^G$ for all $i$. Then $D$ lifts to $\Der(X_U)^G$.
\item Let $W$ be an $H$-module where $H$ is a reductive subgroup of $G$. Then $T_W$ has the infinitesimal lifting property if and only if $W$ does.
\end{enumerate}
\end{lemma}

\begin{proof} By \cite{Roberts1986}, $U\mapsto\Der(X_U)^G$ is coherent 
and by \cite[Chapter 2]{Fischer}, so is $U\mapsto\Der(U)$. Hence we have (1).

Let $\Der_U(X_U)^G$ denote the kernel of $p_*\colon \Der(X_U)^G\to\Der(U)$. Then the corresponding sheaf of $\O_Q$-modules is coherent and part  (2) follows from Cartan's  Theorem B.

Since $G\times W\to T_W$ is a principal $H$-bundle, hence locally trivial, the mapping $\Der(G\times W)^{G\times H}\to\Der(T_W)^G$ is surjective. Now 
$$
\Der(G\times W)^{G\times H}=(\lie g\otimes\O(W))^H\oplus(1_G\otimes \Der(W)^H)
$$
 where $1_G$ denotes the constant function 1 on $G$ and $\lie g$ denotes the  Lie algebra of $G$. Since $(\lie g\otimes\O(W))^H$   obviously acts trivially on $\O(G\times W)^{G\times H}$, we see that the images of $\Der(T_W)^G$ and $\Der(W)^H$ in $\Der(Q)$ are the same. The same reasoning works for $Q$ replaced by an open subset $U\subset Q$. Hence we have (3).
\end{proof}

 \begin{remark}\label{rem:smooth-derivations}
We say that $T_W$ (resp.\ $W$) has the \emph{smooth infinitesimal lifting property\/} if elements in $\Der(U)$, $U$ open in $Q$, have local lifts to smooth invariant vector fields on open sets of  $T_W$ (resp.\ $W$). Then (3) above holds with ``infinitesimal'' replaced by ``smooth infinitesimal.''
 \end{remark}

We now show that the smooth infinitesimal lifting property implies the infinitesimal lifting property.
 \begin{lemma}\label{lem:lift-if-smooth-lift}
 Let $W$ be an $H$-module where $H$ is a reductive subgroup of $G$.
Let $D\in\Der(U)$ where $U$ is open and Stein in   $Q=W\sl H$. Suppose that $D$ has a smooth $H$-invariant lift 
 on $p\inv(U)$. Then $D$ has an $H$-invariant holomorphic lift over $U$.
 \end{lemma}

\begin{proof} Let $p=(p_1,\dots,p_m)\colon W\to\C^m$ where the $p_i$ are homogeneous generators of $\O_\alg(W)^H$. Then we may identify $Q$ with $\Im\, p$. By the slice theorem, Lemma \ref{lem:sheaves-derivations} and Remark \ref{rem:smooth-derivations} it is enough to consider the case that $U$ is a neighbourhood of  $0\in Q$ and to produce a local lift over a perhaps smaller neighbourhood of $0$. 
Give $\O_\alg(Q)\simeq\O_\alg(W)^H$ the grading of $\O_\alg(W)^H$.
The $\O_\alg(Q)$-module $\Der_\alg(Q)$ is graded   where $E\in\Der_\alg(Q)$ has degree $k$ if it sends polynomials of degree $s$ to ones of degree $s+k$ for all $s$. By a Taylor series argument, one sees that $E\in\Der_\alg(Q)$ has a local $H$-invariant holomorphic lift if and only it has a global $H$-invariant polynomial  lift.  Let  $\widehat \O_{Q,0}$ denote the completion of the local ring $\O_{Q,0}$.    If $E$ is a derivation of $\widehat \O_{Q,0}$ and $E_j$ is the component of $E$ of degree $j$, then $E$ has a  lift to an  $H$-invariant formal power series vector field on $W$ if and only if the same is true for each $E_j$. The $\O_\alg(Q)$-submodule of $\Der_\alg(Q)$ of algebraically liftable vector fields  is homogeneous and finitely generated, say by the homogeneous elements $D_1,\dots,D_k$.

 Now let $A$ be a smooth $H$-invariant lift of $D$.  Then
$$
A=\sum_{j=1}^n a_j(z,\bar z)\pt/\pt z_j+b_j(z,\bar z)\pt/\pt \bar z_j,
$$
where $n=\dim W$ and the $z_j$, $\bar z_j$ are the usual holomorphic and antiholomorphic coordinate functions. For any $H$-invariant holomorphic function $f$, $A(f)$ is holomorphic, hence
$$
A(f)=\sum_{j=1}^n a_j(z,0)\pt f/\pt z_j.
$$
Let   $\widehat A= \sum_{j=1}^n \hat a_j(z,0)\pt  /\pt z_j$ where $\hat a_j(z,0)$ is the Taylor series of $a_j(z,0)$ at $0$. Then $\widehat A$ is an $H$-invariant formal vector field lifting $D$. Any such $\widehat A$  is of the form $\sum \hat a_iA_i$ where the $A_i$ are homogeneous $\O_\alg(W)^H$-module generators of $\Der_{Q,\alg}(W)^H$ and the $\hat a_i$ are invariant formal power series. Thus $D=p_*\widehat A=\sum_i \hat b_i p_*A_i$ where the $\hat b_i$ are elements of $\widehat \O_{Q,0}$ which pull back to the $\hat a_i$.  Since the $p_*A_j$ are obviously liftable, we see that  $D$ is  in the  $\widehat\O_{Q,0}$-module generated by the $D_i$.  Hence $D=\sum_{i=1}^k d_iD_i$ where the $d_i\in \O_{Q,0}$. This follows from the fact that the inclusion of $\O_{Q,0}$ into  $\widehat \O_{Q,0}$ is   faithfully flat. Hence $D$ has a holomorphic lift over a neighbourhood of $0\in Q$.  
\end{proof}

We now use some parts of the theory of topological tensor products. Let $N$ and $P$ be smooth manifolds.
Then $\ci(N\times P)$ is isomorphic to a completion of $\ci(N)\otimes\ci(P)$. The completed tensor product is denoted by $\ci(N)\comptensor\ci(P)$. Since both factors are nuclear, the completion is the same for either of the two main topologies on the tensor product ($\pi$ and $\varepsilon$)   \cite[Theorem\ 44.1 and Theorem\ 50.1]{Treves}. If $P$ is a complex manifold, then the  space of functions smooth on $N\times P$ and holomorphic for fixed $x\in N$ is isomorphic to $\ci(N)\comptensor\O(P)$ and is Fr\'echet and nuclear.

\begin{proposition}\label{prop:lift-family}
Let $U\subset Q=Q_X$ be open    and let $D_t$ be a smooth vector field on $[0,1]\times U$ which is in $\Der(U)$ for each fixed $t$. If each $D_t$ is liftable, then there is a smooth vector field $A_t$ on $[0,1]\times X_U$ such that,  for each fixed $t$,  $A_t\in\Der(X_U)^G$ and $p_*A_t=D_t$. 
\end{proposition}

\begin{proof}
The sections of coherent sheaves of $\O_Q$-modules have natural Fr\'echet space structures \cite[Ch.\ V, \S 6]{SteinSpaces}.   Let $\Der_{\ell}(U)=p_*\Der(X_U)^G$ denote the liftable vector fields. Then they form a closed  subspace of $\Der(U)$ \cite[p.\ 169, Closedness Theorem]{SteinSpaces}. Now our smooth vector field $D_t$ is the same thing as an  element of $\ci([0,1],\Der(U))$, the smooth mappings of $[0,1]$ into $\Der(U)$ \cite[Sec.\ 44]{Treves}. Since each $D_t$ is liftable, we are actually in $\ci([0,1],\Der_{  \ell}(U))\simeq\ci([0,1])\comptensor\Der_{  \ell}(U)$ \cite[Theorem 44.1]{Treves}. We have a surjection $\Der(X_U)^G\to\Der_{  \ell}(U)$, hence 
$$
\ci([0,1])\comptensor\Der(X_U)^G\to\ci([0,1])\comptensor\Der_{\mathrm \ell}(U)
$$
 is surjective \cite[Proposition 43.9]{Treves}. 
 It follows that there is a smooth  family $A_t$ in $\Der(X_U)^G$ covering $D_t$.
\end{proof}

 Let $U\subset Q$ and let $\psi\colon U\to\psi(U)\subset Q$ be a strata preserving biholomorphic map.  We say that a $G$-diffeomorphism $\Psi\colon X_U\to X_{\psi(U)}$  inducing $\psi$ is a \emph{strict $G$-diffeo\-mor\-phism over $\psi$\/} if it induces a biholomorphism on reduced fibres. As before, when we say that $\Psi$ is  a strict $G$-diffeomorphism,  we mean that it is over (or induces) the identity on the quotients.
 We will prove Theorem \ref{thm:main3} using the following result.
 
 \begin{theorem}\label{thm:holomorphic-lift}
 Let $\psi$ be as above.
 Suppose that there is a strict $G$-diffeomorphism $\Psi\colon X_U\to X_{\psi(U)}$ over $\psi$ or that there is a strong $G$-homeomorphism $\Psi\colon X_U\to X_{\psi(U)}$ over $\psi$ and $X$ has the infinitesimal lifting property. Then for every $q_0\in U$ there is a neighbourhood $U'$ of $q_0$ and  a $G$-biholomorphism $\widetilde\Psi\colon X_{U'}\to X_{\psi(U')}$ inducing $\psi|_{U'}$.
 \end{theorem} 
 
 Assume the theorem for the moment. Then we have:
 
 \begin{proof}[Proof of Theorem \ref{thm:main3}]
Let $q_0\in Q$. Then $X_{q_0}$ and $Y_{q_0}$ are $G$-isomorphic and by the slice theorem there is a neighbourhood $U$ of $q_0$ and a $G$-biholomorphism $\Psi\colon X_U\to Y_{\psi(U)}$ where $\Psi$ induces $\psi\colon U\to\psi(U)$ and $\psi(q_0)=q_0$.  Suppose that we have a strict $G$-diffeomorphism $\Psi'\colon X\to Y$. Then $(\Psi')\inv\circ\Psi$ is a strict $G$-diffeomorphism of $X_U$ with $X_{\psi(U)}$ over $\psi$. By Theorem \ref{thm:holomorphic-lift}  we can find a $G$-biholomorphism  $\widetilde\Psi\colon X_{U'}\to X_{\psi(U')}$ covering $\psi$ where $q_0\in U'$. Then $\Psi\circ\widetilde\Psi\inv$ is a $G$-biholomorphism of $X_{\psi(U')}$ with $Y_{\psi(U')}$ inducing the identity on $\psi(U')$.  Hence $X$ and $Y$ are locally $G$-biholomorphic over $Q$. One gets the same conclusion if   $\Psi'$ is a strong $G$-homeomorphism   and $X$ has the infinitesimal lifting property.
 \end{proof}
 
To prove Theorem \ref{thm:holomorphic-lift} we may assume that $X=S\times T_B$ is a standard open set (Definition \ref{def:standard-neighbourhood}).  Abusing notation a bit, we will let $Q$ denote $T_W\sl G$.  As before, let  $p_1,\dots,p_m$ be homogeneous generators of $\O_\alg(T_W)^G\simeq\O_\alg(W)^H$  and let $p=(p_1,\dots,p_m)\colon T_W\to \C^m$. Then    $p\colon T_W\to Q=\Im\ p\subset \C^m$ is the quotient mapping for $T_W$. We may assume that $q_0=(s_0,0)$ for some $s_0\in S$. Then $U$ is a neighbourhood of $q_0$ which is sent to the neighbourhood $\psi(U)$ of $q_0$ which we may assume still lies in $S\times Q_B$.  Let $F$ denote $G\times^H\NN(W)$.  As we go along we will keep shrinking $U$ which we can always take to be of the form $S_0\times U_0$ where $S_0$ is a neighbourhood of $s_0\in S$ and $U_0$ is a neighbourhood of $0\in Q$. Write $\psi=(\theta,\phi)$ as before.

\begin{lemma}
To prove Theorem \ref{thm:holomorphic-lift} we may reduce to the case that 
 $\theta(s,q)=s$ for $(s,q)\in U$.
\end{lemma}

\begin{proof} 
 Since $\psi$ is strata preserving, $\theta$ induces a biholomorphism of $S_0\times\{0\}$ into $S\times\{0\}$ fixing $(s_0,0)$. 
  Let  $\tilde\psi(s,q)=(\theta(s,q),q)$ for  $(s,q)\in U$. Then $\tilde\psi$ is   a biholomorphism in a neighbourhood of $(s_0,0)$ and $\tilde\psi\inv$ has the form $(s,q)\mapsto(\tilde\theta(s,q),q)$ where $\tilde\theta(s_0,0)=s_0$.  We have an equivariant biholomorphism  $\widetilde\Psi$ defined over a $G$-saturated neighbourhood $\Omega$ of $\{s_0\}\times F$ given by  $\widetilde\Psi(s,x)=(\tilde\theta(s,p(x)),x)$ for $(s,x)\in\Omega$. Then $\widetilde\Psi$ induces $\tilde\psi\inv$  so  we may replace $\psi$ by $\psi\circ \tilde\psi\inv$ to reduce  to the case that $\theta(s,\cdot)=s$. 
  \end{proof}
For $s\in S_0$  we now have that  $\phi(s,\cdot)$ is a  biholomorphism  defined on $U_0$  such that $\phi(s,0)=0$. We also denote $\phi(s,\cdot)$ by $\phi^s(\cdot)$.
We have a $\C^*$-action on $Q$ induced by the scalar $\C^*$-action on $T_W$. The action of $t\in\C^*$   sends $(q_1,\dots,q_m)\in Q$ to $(t^{e_1}q_1,\dots,t^{e_m}q_m)$ where $e_i$ is the degree of $p_i$, $i=1,\dots,m$.  
We have automorphisms $\phi_t^s(q)=t\inv\cdot\phi^s(t\cdot q)$ which we may assume are defined for $t\in(0,1]$, $s\in S_0$ and $q\in U_0$. We want to show that $\phi^s_0=\lim\limits_{t\to 0}\phi^s_t$ exists. If $\phi^s_0$ exists, then it commutes with the action of $\R^*$, hence also the action of $\C^*$, so we say that it is \emph{quasilinear}.   Let $\Aut_\ql(Q)$ denote   $\Aut(Q)^{\C^*}$, the group of quasilinear automorphisms of $Q$.  Since  $Q\sl\C^*$ is a point, $\Aut_\ql(Q)$   is a linear  algebraic group.  As before,  $\Aut_\vb(T_W)^G$ denotes the   group of $G$-vector bundle automorphisms of $T_W$.  Then any $\ell\in\Aut_\vb(T_W)^G$ is a $G$-biholomorphism which commutes with $\C^*$, hence it induces an element $p_*\ell\in\Aut_\ql(Q)$.

  Let $\gamma\colon  U_0\to\gamma(U_0)$ be a strata preserving biholomorphism (e.g., one of the $\phi^s$) and
 suppose that   $\gamma$ has a lift $\Gamma\colon p\inv(U_0)\to p\inv(\gamma(U_0))$ which is a   strict $G$-diffeomorphism  over $\gamma$. Define $\Gamma_t=t\inv\circ\Gamma\circ t$ where $t\in(0,1]$ acts by scalar multiplication on the fibres of $T_W$.  We may assume that $p\inv(U_0)$ is stable under the scalar action restricted to $[0,1]$. Since $\Gamma$ preserves the fibre over $q_0$, it preserves  the closed orbit $Z\simeq G/H$. Since $\Gamma$ is smooth, $\Gamma_0=\lim\limits_{t\to 0}\Gamma_t$ exists and is the normal derivative $\delta\Gamma$ of $\Gamma$ along $Z$.
 
 \begin{lemma}\label{lem:delta-strict}
Let $\Gamma$ be a strict $G$-diffeomorphism as above. Then $\delta\Gamma\in\Aut_\vb(T_W)^G$.
\end{lemma}

\begin{proof}
We have to show that $\delta\Gamma$ is complex linear. Now $\Gamma([e,0])=[g,0]$ for some $g\in N_G(H)$. Let $W_e$ and $W_g$ denote the fibres of $T_W$ at $[e,0]$ and $[g,0]$, respectively.   Then $\delta\Gamma$ restricts to an $H$-equivariant real linear map of $W\simeq W_e$ to $W_g$   where the $H$-action on   $W_g$ is twisted: $h\cdot [g,w]=[g,(g\inv h g) w]$. Write $W_e\simeq W=W_1\oplus W_2$ where $W_2$ is the fixed space of $H^0$ and $W_1$ is an $H$-stable complement to $W_2$. We have a corresponding direct sum decomposition  of $W_g$.  Any $H$-equivariant real linear map of $W_e$ to $W_g$ has to preserve the direct sum decompositions.  It is easy to see that the Zariski tangent space of $\NN(W)_\red$ at $0$  is $W_1$.  Since $\Gamma$ restricted to $G\times^H\NN(W)_\red$ is   $G$-equivariant and biholomorphic,  $\delta\Gamma$ is complex linear on $W_1$. Hence we only need to show that $\delta\Gamma$ is complex linear on $W_2$. Let $H_2$ denote the image of $H$ in $\GL(W_2)$. Then $H_2$ is finite and the set of principal orbits $W_{2,\pr}$ relative to the $H_2$-action is open and dense in $W_2$. Let $H_0$ denote the kernel of $H\to H_2$ and let $S$ denote the stratum of $Q$ corresponding to $H_0$. Then $W_{2,\pr}\to S$ is a covering map onto its (open) image $U$ and $G\times^HW_{2.\pr}\to U$ is a fibre bundle with fibre $G/H_0$. Then $\Gamma$ is holomorphic on $G\times^HW_{2,\pr}$ since it is equivariant and covers the holomorphic map $\gamma$ restricted to $U$. It follows that $\delta\Gamma$ is complex linear on $W_2$. Hence $\delta\Gamma\in\Aut_\vb(T_W)^G$.
\end{proof}

\begin{remark}\label{rem:strict}
Our definition of strict $G$-diffeomorphism is more general than the one in \cite{KLS}. The 
main use of strictness in \cite{KLS} is to show that if  $\Gamma$ is strict, then $\delta\Gamma$ is complex linear. Using Lemma \ref{lem:delta-strict} one can substitute our definition of strict $G$-diffeomorphism for the one  in   \cite{KLS} and obtain the same theorems.
\end{remark}

    \begin{lemma} \label{lem:limit-t-to-0}
 Let $\gamma$ be one of the $\phi^s$,  $s\in S_0$.
 Suppose that     $\Gamma\colon p\inv(U_0)\to p\inv(\gamma(U_0))$   is a   strict $G$-diffeomorphism  over $\gamma$  or a strong $G$-homeomorphism over $\gamma$.  Then $\gamma_0=\lim\limits_{t\to 0}\gamma_t$ exists and is an element of $p_*(\Aut_\vb(T_W)^G)$.
 \end{lemma}

 \begin{proof}   
Let $\Gamma_t$ and $\gamma_t$ be as above. Then $\Gamma_t$ covers $\gamma_t$ and by Lemmas  \ref{lem:delta-strict} and \ref{lem:fundamental}, $\lim\limits_{t\to 0}\Gamma_t= \delta\Gamma\in\Aut_\vb(T_W)^G$.
Then $p\circ \delta\Gamma$ is the limit of $p\circ \Gamma_t$ as $t\to 0$, hence $\gamma_0$ exists and is covered by $\delta\Gamma$. 
 \end{proof}
 
\begin{proof}[Proof of Theorem \ref{thm:holomorphic-lift}] We have reduced to the case that $\Psi(s,x)=(s,\Phi(s,x))$ for $(s,x)$ in the set $S_0\times U_0$. Let $y_1,\dots,y_m$ be the coordinate functions on $\C^m$. Then near $0\in Q$,
$$
(\phi^s_t)^*y_i=\sum_{\alpha\in\N^m} t^{-e_i}t^{|\alpha|}c_\alpha(s)y^\alpha
$$
where the $c_\alpha(s)$ are holomorphic in $s\in S_0$ and $|\alpha|=\sum e_i\alpha_i$ for $\alpha=(\alpha_1,\dots,\alpha_m)\in\N^m$. Since $\lim\limits_{t\to 0}\phi^s_t$ exists, we must have that $\sum_{|\alpha| =e}c_\alpha(s) y^\alpha$ vanishes on $Q$ for each $e<e_i$. Hence we can throw away these terms and reduce to a sum over the $\alpha$ with $|\alpha|\geq e_i$. Thus the family  $\phi^s_t(q)$  is holomorphic where defined in $s$, $q$ and $t\in\C$. Now $p_*\colon\Aut_\vb(T_W)^G\to\Aut_\ql(Q)$ is a homomorphism of linear algebraic groups, hence the image $\Lambda$ is an algebraic subgroup   of $\Aut_\ql(Q)$.  
 By Lemma  \ref{lem:limit-t-to-0}, the $\phi^s_0$ are in $\Lambda$. The surjection   $p_*\colon \Aut_\ql(T_W)^G\to\Lambda$ has local holomorphic sections, so we can   lift the 
holomorphic family $\phi^s_0$ near $s_0\in S_0$ to a holomorphic family in $\Aut_\vb(T_W)^G$. Hence, shrinking $S_0$, we can  reduce to the case that  $\phi^s_0$ is the identity for all $s\in S_0$. Then we have 
a  homotopy $\phi_t$ with domain  $U=S_0\times U_0$, starting at the identity. 
 Let 
 $$
 \Delta=\{(t,s,q)\in [0,1]\times U\mid(s,q)\in \phi_t(U)\}.
 $$ 
 Then $\Delta$ is open in $[0,1]\times U$ and contains $[0,1]\times \{(s_0,0)\}$. Let $U'=S_0'\times U_0'$ be a neighbourhood of $(s_0,0)$ such that $[0,1]\times U'\subset\Delta$. Our local  
homotopy $\phi_t$ is obtained by integrating a time dependent vector field $D_t$ and the domain of $D_t$ contains $[0,1]\times U'$. 
 Since the family $\phi_t$ is strata preserving, the $D_t$ are strata preserving. If $X$ has the infinitesimal lifting property, then Proposition \ref{prop:lift-family} shows that we have a smooth family $A_t\in \Der(X_{U'})^G$ which  lifts $D_t$.
 
 In the case that we have a strict $G$-diffeomorphism $\Psi(s,\cdot)=(s,\Phi(s,\cdot))$ covering $\psi$, the family $\Phi^s_t(x)$ is smooth in $t$,  $s$ and $x$ and covers $\phi^s_t$.  Since each $\Phi^s_0$ is in $\Aut_\vb(T_W)^G$ and $\phi^s_0$ is the identity, we may replace each $\Phi^s$ by $(\Phi^s_0)\inv\circ\Phi^s$ in which case we have reduced to the case that $\Phi_t$ is   a smooth 
 homotopy of strict $G$-diffeomorphisms over $\phi_t$ starting at the identity. We have $\Phi_t(X_{U'})\subset X_U$ for $t\in[0,1]$ where $U'$ is as above. Now $\Phi_t$ is obtained by integrating  a smooth   time dependent vector field  $B_t$, and $B_t$  is defined on $[0,1]\times X_{U'}$. Moreover, $B_t$ lifts $D_t$ for each $t$. By Lemma \ref{lem:lift-if-smooth-lift} each $D_t$ is holomorphically liftable, and by Proposition \ref{prop:lift-family} we can find a smooth family $A_t\in\Der(X_{U'})^G$ which lifts $D_t$.
 
Let 
$$
\Delta'=\{(t,s,q)\in[0,1]\times U'\mid\phi_t(s,q)\in U'\}.
$$
 Then $\Delta'$ is again a neighbourhood of $[0,1]\times\{(s_0,0)\}$ and we can find a neighbourhood $U''$ of $(s_0,0)$ such that $\phi_t(U'')\subset  U'$ for all $t\in[0,1]$. Starting at any point $(s,q)\in U''$, the flow  of $D_t$ at time 1 gives $\phi(s,q)$.  By \cite[proof of Theorem 3.4]{Schwarz2014} or \cite[Lemma 3.1, Corollary 3.3]{KLSb} the flow 
 $\widetilde\Phi_t(s,x)$ of $A_t$ exists  for $(s,x)\in X_{U''}$ and $t\in[0,1]$.
Then $(s,x)\mapsto  (s, \widetilde\Phi_1(s,x))$ is a lift of $\psi$ to a  $G$-biholomorphism of $X_{U''}$ with $X_{\psi(U'')}$.   
 \end{proof}

Now that we have Theorem 1.3, so that we may assume that  $X$ and $Y$ are locally $G$-biholomorphic over $Q$, we can give a cohomological interpretation of $Y$ and $G$-biholomorphisms over $\id_Q$. Let $\A$ be the sheaf of groups on $Q$ such that $\A(U)=\Aut_U(X_U)^G$ for $U$ open in $Q$. Let 
$\{U_i\}$ be an open cover of $Q$ such that we have $G$-biholomorphisms $\Psi_i\colon X_{U_i}\to Y_{U_i}$, inducing the identity on $U_i$. Let $\Phi_{ij}=\Psi_i\inv\circ\Psi_j\in\A(U_i\cap U_j)$. Then $\Phi_{ij}$ is a cocycle giving an element of $H^1(Q,\A)$.  If $Y'$ is also locally $G$-biholomorphic to $X$ over $Q$ and has cocycle $\Phi'_{ij}$, then $Y$ and $Y'$ are $G$-biholomorphic over $Q$ if and only if the two cocycles represent the same cohomology class. Given a cocycle $\Phi_{ij}$, one constructs a complex $G$-manifold $Y$ by glueing $X_{U_i}$ and $X_{U_j}$ over $U_i\cap U_j$ using $\Phi_{ij}$. However, it is not clear that $Y$ is Stein.

\begin{theorem}\label{thm:Y-is-Stein}
Let $Y$ be a complex $G$-manifold locally $G$-biholomorphic to $X$ over $Q$. Then $Y$ is Stein. Hence there is a one-to-one correspondence  between isomorphism classes of Stein $G$-manifolds locally 
$G$-biholomorphic to $X$ over $Q$ and the cohomology set $H^1(Q,\A)$.
\end{theorem}

\begin{proof}
Consider the sheaf of $\O_Q$-modules given by $U\mapsto\O (Y_U)_V$ where $V$ is an irreducible $G$-module. The sheaf is coherent since it is locally isomorphic to the corresponding sheaf of $G$-finite holomorphic functions on $X$. 

Let $q\in Q$ and let   $f_1,\dots,f_m$ be $G$-finite functions on $Y$ whose restrictions to $Y_q$ generate $\O_\alg(Y_q)$. Then the $f_i$ generate the $G$-finite functions on an open set $Y_U$, as a module over $\O(U)$, where $U$ is a sufficiently small 
neighbourhood of $q$.   Shrink $U$ so that it is Runge in $Q$ and so that $Y_U\simeq X_U$. The $G$-finite functions on $Y_U$ are dense in $\O(Y_U)$. Since the $G$-finite functions are generated by the $f_i$ as $\O(U)$-algebra and since $\O(Y)^G=\O(Q)$ is dense in $\O(U)$, the algebra $\O(Q)[f_1,\dots,f_m]$ is dense in $\O(Y_U)$. Hence $\O(Y)$ is dense in $\O(Y_U)$. Moreover, for $U'$ a relatively compact 
 neighbourhood of $q$ in $U$, there are $h_1,\dots,h_n$ in $\O(Q)$ such that the mapping 
$$
\rho=(h_1,\dots,h_n, f_1,\dots,f_m)\colon p_Y\inv(\overline{U'})\to \C^{n+m}
$$
 is a closed embedding.

Let $K\subset Y$ be compact.  Let $y_n\in\widehat K$, the holomorphically convex hull of $K$ in $Y$. 
Then $\{p_Y(y_n)\}$ has a convergent subsequence, so we can assume that $p_Y(y_n)\to q$ where $q$ sits inside our open set $U'$ as above. The $f_i$ are bounded on $K$, hence the $f_i(y_n)$ are bounded, which implies that there is a subsequence of 
 $\{y_n\}$ such that $f_i(y_n)$ converges for all $i$. We already know that the $h_j(y_n)$ converge. Hence $\rho(y_n)$ converges which implies that $y_n$ converges to a point in $Y_q$. Hence $\widehat K$ is compact and 
 it follows that $Y$ is Stein.
\end{proof}

 Let $\Psi\colon S\times T_B\to S\times T_B$ be a strong $G$-homeomorphism  or strict 
 $G$-diffeomor\-phism. Then $\Psi(s,x)=(s,\Phi(s,x))$ where $\Phi(s,x)\colon S\times T_B\to T_B$. From now on we will identify $\Psi$ with the family $\Phi(s,\cdot)$ and will say that $\Phi$ is a strong $G$-homeomorphism (or strict $G$-diffeomorphism).

    \section{Sections of type $\F$}\label{sec:typeF} In Theorem \ref{thm:main4} we are assuming that $X$ and $Y$ are locally $G$-biholomorphic over $Q$ and that there is a strict $G$-diffeomorphism $\Phi\colon X\to Y$  or that there is a strong $G$-homeomorphism $\Phi\colon X\to Y$. 
  In this section we define the notion of a $G$-diffeomorphism from $X$ to $Y$ of type $\F$.   In Section \ref{sec:reductiontoF} we show that our (strict or strong) $\Phi$ is homotopic to a $G$-diffeomorphism of type $\F$. In Sections \ref{sec:NHC} and \ref{sec:Grauert} we will show that any $\Phi$   of type $\F$  can be  deformed to a $G$-biholomorphism.  In this section  we investigate the local structure of $G$-diffeomorphisms of type $\F$ and vector fields of type $\LF$ (defined below). The main result   is Theorem \ref{thm:Dexists} which says that a $G$-diffeomorphism of type $\F$, which is sufficiently close to the identity over a neighbourhood $U$ of a compact subset $K\subset Q$, has a canonically associated   vector field $D=\log \Phi$  of type $\LF$, defined over a smaller 
  relatively compact neighbourhood $U'$ of $K$, such that $\exp D=\Phi$ over $U'$. The definition of $\log\Phi$ depends upon a choice of standard generating set for $\O_\gf(X_U)$ (Definition \ref{def:standard-generators2}). We will use Theorem \ref{thm:closed} of the next section which says  that the vector fields of type $\LF$   are closed in the space of smooth vector fields on $X$ and that, for any irreducible $G$-module $V$, the $\ci(X)^G$-module    generated by    $\O(X)_V$ is closed in $\ci(X)$.

\begin{definition}\label{def:typeF}
Let $\Phi\colon X\to Y$ be a  $G$-diffeomorphism inducing $\Id_Q$.
We say that \emph{$\Phi$ is of type $\F$\/} if for every $x_0\in X$ there is a $G$-saturated neighbourhood $U$ of $x_0$
and a map $\Psi\colon  U\times U \to Y$ such that:
\begin{enumerate}
\item For $x\in U$ fixed, $\Psi(x,y)$ is a biholomorphic $G$-equivariant map of $\{x\}\times U$ into $Y$, inducing the identity on the quotient.
\item $\Psi$ is smooth in $x$ and $y$ and $G$-invariant in $x$.
\item $\Phi(x)=\Psi(x,x)$, $x\in U$.
\end{enumerate}
We call $\Psi$ a \emph{local holomorphic extension of $\Phi$.}
\end{definition}
Note that if $\Phi$ is holomorphic, then it is of type $\F$ by setting $\Psi(x,y)=\Phi(y)$. A $G$-diffeomorphism  of  $X$ of type $\F$ is obviously strict, and in Proposition \ref{prop:aij-exist} we will see that it is also strong.
For an open subset $U\subset Q$ let $\F(U)$ denote the $G$-diffeomorphisms of type $\F$ on $X_U$. Then $\F$ is a sheaf of groups. The corresponding sheaf of Lie algebras is the sheaf of smooth vector fields of type $\LF$,  as follows. Let $\Der_Q(X)^G$ denote  the Lie algebra of holomorphic $G$-invariant vector fields on $X$ which annihilate $\O(X)^G$. They are the kernel of the push-forward mapping $p_*\colon \Der(X)^G\to\Der(Q)$. Define $\ci(X)^G\cdot\Der_Q(X)^G$ to be the Lie algebra of vector fields on $X$ which are locally finite sums $\sum a_i(x)A_i(x)$ where the $A_i$ are in $\Der_Q(X)^G$ and the $a_i$ are locally defined $G$-invariant smooth  functions.  
\begin{definition}\label{def:typeLF}
A   vector field \emph{$D$ is of type $\LF$\/} if $D\in\ci(X)^G\cdot\Der_Q(X)^G$. 
\end{definition}

For $U\subset Q$ an open set, we define $\LF(U)$ to be the Lie algebra of vector fields of type $\LF$ on $X_U$. If $D\in\LF(U)$, then
$D$ is locally of the form $\sum a_i(x)A_i(x)$, and we automatically get a family $D(x,x')=\sum a_i(x)A_i(x')$ where for $x$ fixed, $D(x,x')$ is an element of $\Der_Q(X)^G$. This hints  why the sections of $\LF$ should be thought of as the Lie algebra of the sections of $\F$.

\begin{remark}
Since $\LF(Q)$ is closed in the space of $\ci$  vector fields on $X$ (Theorem \ref{thm:closed}), it is a Fr\'echet space. We doubt that the analogous result is true if we replace $\ci$ by $\CC$ in Definition \ref{def:typeLF}. This explains our need for smoothness assumptions.
\end{remark}

Recall that a vector field $D$  is \emph{complete\/} if the flow $\exp(tD)$ exists for all $t\in\R$.
\begin{lemma}\label{lem:LFcomplete}
Let $D$ be a $G$-invariant smooth vector field on $X$ which is tangent to the fibres of $p\colon X\to Q$. Then $D$ is complete. In particular,   vector fields   of type $\LF$ are complete. 
\end{lemma}

\begin{proof}
Let $F$ be a fibre of $p$. Then $D$ is tangent to $F$ and gives an element of the Lie algebra of $\Aut(F)^G$, which is a linear algebraic group. Hence $D|_F$ can be integrated for all time.
\end{proof}

\begin{corollary}\label{cor:expLFistypeF}
Let $U\subset Q$ be open and $D\in\LF(U)$. Then $\exp(D)\in\F(U)$.
\end{corollary}

 \begin{example}
 Let $D$ be a smooth vector field on the $G$-module $V$ such that $D$ annihilates $\O(V)^G$ and lies in $\ci(V)^G\cdot\Der_\alg(V)^G$. One can show that if all the isotropy groups of closed orbits of $V$ are connected, then $D$ is of type $\LF$. Here is an  example where $D$ is not of type $\LF$.
 
 Let $V$ be the direct sum of two copies of $\C^2$ with the diagonal action of the group $G$ of Example \ref{ex:group scheme}. We have $V=\C^4$ with coordinate functions $x_1,x_2,y_1,y_2$ where the $x_i$ have weight 1 and the $y_i$ have weight $-1$ for the action of $G^0=\C^*$. Then $G$ is generated by $G^0$ and the element sending $x_i$ to $y_i$ and $y_i$ to $-x_i$, $i=1$, $2$.   The $G$-invariant linear vector fields have as basis the fields
 $$
 X_{ij}=x_i\pt/\pt x_j+y_i\pt/\pt y_j,\ 1\leq i, j\leq 2
 $$
and the   polynomials transforming by the sign representation of $G$ are generated by  
$$
f_{ij}=x_iy_j+x_jy_i,\ 1\leq i\leq j\leq 2.
$$
Up to scalars, there is one  quadratic invariant $f=x_1y_2-x_2y_1$. Let $A$ denote the generator of the Lie algebra of $\C^*$. Then 
$$     
A=x_1\pt/\pt x_1-y_1\pt/\pt y_1+x_2\pt/\pt x_2-y_2\pt/\pt y_2
$$
 and $A$ transforms by the sign representation of $G$. Thus $\Der_{Q,\alg}(V)^G$ is generated by the $f_{ij}A$. We have the curious relation
$$
fA=f_{12}X_{11}+f_{22}X_{12}-f_{11}X_{21}-f_{12}X_{22}.
$$
Let $h(\bar z)$ be an antiholomorphic quadratic polynomial which transforms by the sign representation of $G$. Let $D=h(\bar z)fA$.  Using the curious relation we may express $D$ as a sum of real  invariant polynomial functions times the $X_{ij}$. Hence $D\in\ci(V)^G\cdot\Der_\alg(V)^G$ and $D$ annihilates the $G$-invariant holomorphic functions.
If $D$  were of type $\LF$ then we would have 
$$
h(\bar z)fA=\sum_{1\leq i\leq j\leq 2}k_{ij}(\bar z)f_{ij}A
$$ 
where the   $k_{ij}$ are quadratic and invariant. This is clearly not possible. Hence $D$ is not of type $\LF$.
 \end{example}

A simple but useful result:

\begin{lemma}\label{lem:pi}
Let $X$ be a Stein $G$-manifold and let $\{X_\alpha\}$ be a  cover of $X$ by $G$-saturated open sets. Then there is a partition of unity by smooth $G$-invariant functions $f_\alpha$ where $\supp f_\alpha\subset X_\alpha$. 
\end{lemma}

\begin{proof}
It is enough to do the case of a $G$-module $V$. Let $  p\colon V\to\C^d$ be the quotient morphism with image $Q$. Then the $p(V_\alpha)$ give an open cover of $Q$. Let $U_\alpha$ be open in $\C^d$ such that $U_\alpha\cap  Q= p(V_\alpha)$. Let 
$\{f_\alpha\}$ be a partition of unity subordinate to $\{U_\alpha\}$. Then 
$\{p^*f_\alpha\}$ is the required partition of unity on $V$.
\end{proof}

Let $S\times T_B$ be a standard neighbourhood in $X$ (which we now just assume is $X$) and let $\{f_i\}_{i=1}^n$ be standard generators of $\O_\gf(T_B)$ corresponding to the distinct irreducible nontrivial $G$-modules $V_1,\dots,V_r$ (Definition \ref{def:standard-generators}). Then the $f_i$ are linearly independent on  $F=G\times^H\NN(W)$. Let  $p'\colon T_B\to Q'$ denote  the quotient mapping so that $p(s,x)=(s,p'(x))\colon S\times T_B\to Q=S\times Q'$ is the quotient mapping of $X$.   
\begin{proposition}\label{prop:aij-exist}
Let $\Phi\in\F(S\times Q')$ where the $f_i$,  etc.\  are as above.   Then there are $G$-invariant smooth functions $a_{ij}(s,x)$ such that 
$$
(\Phi^*f_i)(s,x)=\sum_j a_{ij}(s,x) f_j(x),\  s\in S,\ x\in T_B.
$$
Hence $\Phi$ is strong.   If $\Phi$ is the identity on $\{s\}\times F$, then the matrix $(a_{ij}(s,x))$ is the identity for $x\in F$.
\end{proposition}

\begin{proof}
 Since the $f_i$ are linearly independent on $F$   the last claim is clear. First suppose that $S$ is a point. If we can produce smooth $a_{ij}$ locally over $Q$, then using a partition of unity   we can produce the desired $a_{ij}$ globally. Thus we may shrink $Q'$  in our proof.   Let ${\mathcal M}$ denote
$$
\O(Q)\cdot f_1+\cdots+\O(Q)\cdot f_n\simeq\O(X)_{V_1}\oplus\cdots\oplus\O(X)_{V_r}.
$$  
Then ${\mathcal M}$ is a Fr\'echet space   \cite[Ch.\ V, \S 6]{SteinSpaces}.
We have a   surjection $\pi$ from $\O(Q)^n$ onto  ${\mathcal M}$  sending $(a_1,\dots,a_n)$ to $\sum_i a_i f_i$. Let $q\in Q$. Then over a neighbourhood of $X_q$ we have a local holomorphic extension  $\Psi$ of $\Phi$ as in Definition \ref{def:typeF}. Replacing $Q$ by a neighbourhood of $q$, we may assume that $\Psi$ is defined on $X\times X$. Then 
$(\Psi^*f_i)(x,y)$ is smooth and lies in ${\mathcal M}$ for each fixed $x$. Hence 
$(\Psi^*f_i)(x,y)$ is an element of $\ci(X)^G\comptensor {\mathcal M}$. The spaces $\ci(X)^G$, $\O(Q)\simeq \O(X)^G$ and ${\mathcal M}$ are all Fr\'echet and nuclear since closed subspaces of Fr\'echet nuclear spaces are Fr\'echet and nuclear \cite[Proposition 50.1]{Treves}. Since $\pi\colon\O(Q)^n\to {\mathcal M}$ is surjective, the induced mapping of $\ci(X)^G\comptensor\O(Q)^n$ to $\ci(X)^G\comptensor {\mathcal M}$ is surjective \cite[Proposition 43.9]{Treves}. Hence there are smooth functions $a_{ij}(x,y)$ on $X\times X$ which are $(G\times G)$-invariant and holomorphic in  $y$ such that 
$$
(\Psi^*f_i)(x,y)=\sum_j a_{ij}(x,y)f_j(y),\ i=1,\dots,n.
$$
Set $a_{ij}(x)=a_{ij}(x,x)$. This gives the case where $S$ is a point. 
In the general case, since $\Phi$ has to  induce the identity on $S\times Q'$,  it gives a smooth family of $G$-diffeomorphisms of $T_B$ of type $\F$ parameterised by $S$. One now uses another  topological tensor product argument.
\end{proof}

\begin{remark}\label{rem:aij-holomorphic} 
Suppose that $\Phi$ is holomorphic. Then we don't have to introduce a $\Psi$ and the proof above shows that $\Phi^*f_i=\sum a_{ij}f_j$ where the $a_{ij}\in\O(Q)$. 
 \end{remark} 
We will consider the $a_{ij}$ both as functions on $X$ and as functions on $Q$. As functions on $Q$ they may not be smooth, however. See Example \ref{ex:not-smooth}.

If  $(a_{ij})$ is near the identity, we can take its logarithm $(d_{ij})$. We want to have a $G$-invariant vector field $D$, of type $\LF$, such that  $D(f_i)=\sum d_{ij}f_j$. Then $\exp(D)=\Phi$. Note that $D$, if it exists, is uniquely determined by the $d_{ij}$. We show that $D$ exists if  $(a_{ij})$ is  
 close to the identity in the $\ci$-topology (not just in the $\CC$-topology).

We consider   explicit seminorms on $\ci(X)$ for $X$   a general Stein $G$-manifold. Let $\{M_i\}$ be a locally finite collection of compact sets which cover $X$ such that $M_i\subset U_i$ where $U_i$ is the domain of a coordinate chart. Let $M$ be a compact subset of $X$ and $k$ a non-negative integer. For $f\in\ci(X)$ we define $||f||_{M,k}$ to be the maximum of the partial derivatives of $f$ up to order $k$ on $M\cap M_i$ relative to the coordinate functions of $U_i$. We will abbreviate $||\cdot||_{M,0}$ by $||\cdot||_M$. A sequence $f_n\in\ci(X)$ converges to $f\in\ci(X)$ if and only if $||f_n-f||_{M,k}\to 0$ as $n\to\infty$ for all $M$ and $k$. This is the $\ci$-topology, which does not depend upon the choices made. We give the closed subspace $\ci(X)^G\subset\ci(X)$ the induced topology. We give spaces of smooth functions $f\colon X\to\C^m$ the topology of convergence of each component function in the $\ci$-topology.  We fix     norms   $\vert\,\cdot\,\vert_n$   on ${\mathrm M}(n,\C)$, $n\in\N$,  with the property that $\vert CD\vert_n\leq \vert C\vert_n\cdot \vert D\vert_n$ for $C$, $D\in {\mathrm M}(n,\C)$ and such that the identity matrix has norm 1. We usually drop the subscript $n$ and just write  $\vert\,\cdot\,\vert$.
  If $(f_{ab})$ is a square matrix of smooth functions on $X$, then we define the seminorm $ ||(f_{ab})||_{M,k}$ to be the maximum of $\vert\,\cdot\,\vert$ applied to the partial derivatives of $(f_{ab})$  up to order $k$ on $M\cap M_i$ relative to the coordinate functions of $U_i$. Again, we abbreviate $||(f_{ab})||_{M,0}$ by $||(f_{ab})||_M$.
 Since $X$ is Stein, it can be embedded as a closed complex submanifold $\widetilde X$ of some $\C^m$. Then we may consider diffeomorphisms $\Phi$ of $X$ as mappings $\widetilde \Phi$ of $X\to\widetilde X\subset \C^m$. We say that a sequence of diffeomorphisms  $\Phi_n$ converges to the diffeomorphism $\Phi$  if the  mappings $\widetilde \Phi_n$ converge to  $\widetilde \Phi$ in the $\ci$-topology. Again, this topology on the diffeomorphisms does not depend upon the choices we have made. Similarly, smooth vector fields can be considered as mappings of $X\to\C^m$, giving the $\ci$-topology on smooth vector fields. 

  The following lemma will come in handy.
  
   \begin{lemma}\label{lem:series}
 Let $\alpha(z)= 
 \sum\limits_{i=1}^\infty a_iz^i$ be a power series without constant term and radius of convergence $R>0$. Let $A$ be a square matrix of elements of $\ci(X)$. Then
 \begin{enumerate}
\item For $||A||_{M}<R$, $\alpha(A)$ converges absolutely 
  and uniformly on a neighbourhood of $M$ to a matrix of smooth functions.
\item For $k> 0$, 
$$
 ||\alpha(A)||_{M,k} \leq \sum_{i=1}^k\beta_i(||A||_{M})\cdot (||A||_{M,k})^i
$$
where each $\beta_i$ is a series with radius of convergence $R$.
\item Given $k\geq 0$ and $\epsilon>0$ there is a $\delta>0$ such that $||A||_{M,k}<\delta$ implies that $||\alpha(A)||_{M,k}<\epsilon$.
\end{enumerate}
 \end{lemma}
 
 \begin{remark}\label{rem:log-exp-inverses}
 Suppose that $A\in{\mathrm M}(n,\C)$. If  $\vert A\vert<\log 2$, then $\vert\exp(A)-I\vert<  1$ and $\log\exp A=A$. If   $\vert A-I\vert<1/2$, then $\vert\log A\vert<\log 2$ and $\log A$ is the unique matrix of norm less than $\log 2$ whose exponential is $A$.
  \end{remark}
  
  \begin{definition}\label{def:standard-generators2}
  Let   $V_1,\dots,V_r$ be nonisomorphic  irreducible nontrivial  $G$-modules appearing in $\O_\gf(X)$ such that   the $\O(X)_{V_j}$   generate $\O_\gf(X)$ as $\O(Q)$-algebra.    Further suppose that the $\O(X)_{V_j}$ are finitely generated $\O(Q)$-modules and that $f_1,\dots,f_n$ are a minimal generating set of $\oplus\O(X)_{V_j}$ with each $f_i$ in some $\O(X)_{V_j}$. Then we call   $\{f_i\}$ a \emph{standard set of generators  of $\O_\gf(X)$}. When $X=T_B$, as before, we always assume that our standard generators are   in $\O(T_W)$ and are homogeneous.
  \end{definition}

Assume that we have a standard generating set $\{f_i\}_{i=1}^n$  for $\O_\gf(X)$ corresponding to irreducible $G$-representations $V_1,\dots,V_r$.
The span of the $f_i$    contains $c_1$ copies of $V_1$, \dots, $c_r$ copies of $V_r$ for some $c_j\in\N$.  Assume that $f_1,\dots,f_k$ are a basis for the $c_1$ copies of   $V_1$. Then we have a mapping $\gamma_1\colon X\to (V_1^{\oplus c_1})^*$ where $\gamma_1(x)$ is the element of $(V_1^{\oplus c_1})^*$ that sends a linear combination $f$ of $f_1,\dots,f_k$  to $f(x)$. Similarly there are $\gamma_j\colon X\to(V_j^{\oplus c_j})^*$, $j=2,\dots,r$. 
Let $\gamma$ denote the product of the $\gamma_j$. Then $\gamma\colon X\to V^*$ where $V= \oplus V_j^{\oplus c_j}$. Since $\O_\gf(X)$ is dense in $\O(X)$, the mapping $(p,\gamma)\colon X\to Q\times V^*$ is an equivariant embedding. It follows that any slice representation of $X$ is a subrepresentation of a slice representation of $V^*$, hence $X$ has finitely many slice types (equivalently, a finite Luna stratification). By \cite[Einbettungssatz I]{HeinznerEmbedding} this implies that $X$ equivariantly embeds into a $G$-module, hence $Q$ has a closed embedding $\sigma\colon Q\to\C^m$ for some $m$.  Then $\gamma$ and   $\sigma\circ p\colon X\to \C^m$ give us an equivariant closed embedding
\begin{equation}\label{eq:embedding}
\Gamma\colon X\to \C^m\times V^*.
\end{equation}
Conversely, given a closed equivariant embedding $\Gamma$ of $X$ into a $G$-module   (which we   allow to include the trivial representation), it is clear that $\O_\gf(X)$ has a standard generating set.

\begin{lemma}\label{lem:embeds-over-K}
Let $X$ be a Stein $G$-manifold and let 
 $U$ be a relatively compact Stein domain in $Q$. Then there is an equivariant closed embedding $X_U\to V^*$ for some  $G$-module $V$. Hence $\O_\gf(X_U)$ has a standard generating set.
\end{lemma}

\begin{proof}
Since the Luna stratification of $Q$ is locally finite, $U$ intersects only finitely many Luna strata of $Q$. Then $X_U$ admits a proper holomorphic $G$-equivariant embedding in a $G$-module  \cite[Einbettungssatz I]{HeinznerEmbedding}.
\end{proof}

We continue to assume that $\O_\gf(X)$ has a standard generating set.     We define two topologies on $\F(Q)$ equivalent to the $\ci$-topology.

 Let $\Phi\in\F(Q)$. Then it follows from Proposition \ref{prop:aij-exist} that $\Phi^*f_i=\sum a_{ij}f_{j}$ where the $a_{ij}\in\ci(X)^G$ and $a_{ij}=0$ if $f_i$ and $f_j$ do not correspond to the same $V_\ell$. Let  
 $$
 \mathcal M=\ci(X)^G\cdot\O(X)_{V_1}\oplus\cdots\oplus\ci(X)^G\cdot\O(X)_{V_r}\subset\ci(X)^r.
 $$
 Then each $\ci(X)^G\cdot\O(X)_{V_j}$ is generated 
 over $\ci(X)^G$ by the $f_i$ transforming by the representation $V_j$.
 By Theorem \ref{thm:closed}, $\mathcal M$ is closed in $\ci(X)^r$.  Let $E$ be the space of  endomorphisms of $\mathcal M$ as $\ci(X)^G$-module and as $G$-module. Then $\alpha\in E$ is determined by 
  $(\alpha(f_i))\in {\mathcal M}^n$. This gives a topology on $E$ and it is easy to see that the image of $E$ in ${\mathcal M}^n$ is closed. Hence $E$ is a Fr\'echet space. Let $E_0$ denote the space of $(n\times n)$-matrices   $(a_{ij})$ of elements of $\ci(X)^G$     such that $f_{i}\mapsto\sum a_{ij}f_{j}$ is an element of $E$. Then $E_0$ is closed in the space of $(n\times n)$-matrices with entries in $\ci(X)^G$ and obviously maps onto $E$.
 
  \begin{proposition}\label{prop:equivalent-topologies-phi}
  Let $f_1,\dots,f_n$, $X$, etc.\  be as above. Let 
  $\Phi\in\F(Q)$ with corresponding matrix $(a_{ij})\in E_0$. Then the following give equivalent neighbourhood bases of   
  $\Phi$ in $\F(Q)$.
  \begin{enumerate}
\item The neighbourhoods of $\Phi$ in the $\ci$-topology.
\item The set of all $\Phi'$ such that $(\Phi')^*$ is in a neighbourhood of $\Phi^*$ in $E$.
\item The set of all $\Phi'$ such that some choice of 
corresponding matrix $(a_{ij}')$ is in a neighbourhood of  $(a_{ij})$ in $E_0$. 
\end{enumerate}
 \end{proposition}
  
  \begin{proof}
 If $\Phi'$ is close to $\Phi$ in the $\ci$-topology, then the $(\Phi')^*f_{i}$ are close to the $\Phi^*f_{i}$, i.e., $(\Phi')^*$ is close to $\Phi^*$ in $E$. Now suppose that $(\Phi')^*$ lies in a neighbourhood $U$ of $\Phi^*$ in $E$. Let $U_0$ be a neighbourhood of $(a_{ij})\in E_0$ which maps into $U\subset E$. By the open mapping theorem, the image of $U_0$ 
is a neighbourhood    of  $\Phi^*$ in $E$. Hence if $(\Phi')^*$ is close to $\Phi^*$ we can choose   $(a_{ij}')$ close to $(a_{ij})$. 
 Finally, suppose that      $(a_{ij}')$ is close to  $(a_{ij})$.   Let $\Gamma\colon X\to \widetilde X\subset \C^m\times V^*$ be the embedding of \eqref{eq:embedding}.  Then $\Phi$ acts on $\widetilde X$ by
  $$
  \tilde x=(z,v^*_1,\dots,v^*_n)\mapsto (z,\sum_j   a_{j1}(\tilde x)v^*_{j},\dots,\sum_j   a_{jn}(\tilde x)v^*_j)
  $$ 
  where  $ z\in\C^m$,  $ (v^*_1,\dots,v^*_n)\in V^*$. Since $(a_{ij}')$ is close to   $(a_{ij})$, we see that $\Phi'$ is close to $\Phi$ in the $\ci$-topology. Thus the three neighbourhood bases of $\Phi$ are equivalent.
  \end{proof}

Let $D\in\LF(Q)$. Then $D(f_i)=\sum d_{ij}f_j$ with $d_{ij}\in\ci(X)^G$. Thus $D$ gives us an element $D^*$ of $E$ and an element $(d_{ij})\in E_0$. As above we have the following result.
 \begin{proposition}\label{prop:equivalent-topologies-D}
 Let $f_1,\dots,f_n$, $X$, etc.\  be as above. Let 
 $D\in\LF(Q)$ with corresponding   matrix $(d_{ij}) \in E_0$. Then the following give equivalent neighbourhood bases of   $D$ in   $\LF(Q)$.
  \begin{enumerate}
\item The neighbourhoods of $D$ in the $\ci$-topology.
\item The set of all $D'$ such that $(D')^*$ is in a neighbourhood of $D^*$ in $E$.
\item The set of all $D'$ such that some choice of 
corresponding matrix $(d_{ij}')$ is in a neighbourhood of  $(d_{ij})$ in $E_0$. 
\end{enumerate}
 \end{proposition}

We now come to some key definitions.

 \begin{definition}\label{def:propertyLF} Let $X$ be a Stein $G$-manifold with a standard generating set $f_1,\dots,f_n$ of $\O_\gf(X)$. Let $q\in Q$. We say that the fibre $X_q$ \emph{has property (LF)\/} if there is an open neighbourhood $U_0$ of $q\in Q$ with the following property.  For  every open neighbourhood $U$ of $q$ contained in $U_0$ there is a smaller   neighbourhood  $U'$ of $q$ with compact closure in $U$ and a neighbourhood   $\Omega$ of the identity in  $\F(U)$ such that for $\Phi\in\Omega$ we have:
\begin{enumerate}
\item $\Phi$ corresponds to a matrix $(a_{ij})$ with $||(a_{ij})-I||_{\overline{U'}}<1/2$.
\item There is a $D\in\LF(U')$ such that  $D(f_i)=\sum d_{ij}f_j$ on $X_{U'}$ where $(d_{ij})=\log (a_{ij})$. 
\end{enumerate}
As shorthand for conditions (1) and (2) we say that \emph{$\Phi$ admits a logarithm in $\LF(U')$}. Note that $(a_{ij})$ is not unique. The condition is that some $(a_{ij})$ corresponding to $\Phi$ satisfies (1) and (2).
\end{definition}
Of course, we have to show that the choice of standard generating set does not matter. First we make some remarks.
In what follows, when we write $(a_{ij})$ we will always mean that $(a_{ij})$ corresponds to the $\Phi$ or $\Psi$ of type $\F$ that we are considering.

 \begin{remark}\label{rem:(1)-is-automatic}
 The condition (1) above can always be arranged. Choose a compact subset $M\subset X$ whose image in $Q$ is $\overline{U'}$. 
  Then  $\{\Phi\in\F(U): ||(a_{ij})-I||_{M}<1/2\}$ is a neighbourhood of the identity in $\F(U)$.   Moreover, by Lemma \ref{lem:series}, 
  the series $\log (a_{ij})$ converges uniformly on compact subsets of a $G$-saturated neighbourhood of $M$ to a matrix of smooth invariant functions.
 \end{remark}
 
\begin{remark}\label{rem:matrix-not-matter}\label{rem:standard-basis-not-matter}
The formal series $\log\Phi^*$, when applied to any $f_i$, converges to $D(f_i)$. Hence $D$ is independent of the choice of $(a_{ij})$. 
\end{remark}

\begin{remark}\label{rem:exp-D=Phi}
Properties (1) and (2) imply that $\exp D=\Phi$ over $U'$. Note that $||(d_{ij})||_{\overline U'}<\log 2$ and $(d_{ij})$ is the unique matrix satisfying this property whose exponential is $(a_{ij})$. See Remark \ref{rem:log-exp-inverses}.
\end{remark}

 \begin{lemma}\label{lem:(LF)-well-defined}
Let $X$, etc.\ be as above. Let $f_1',\dots,f_m'$ be a second standard generating set of $\O_\gf(X)$. 
If (1) and (2) of Definition \ref{def:propertyLF} hold for the $f_j'$, then they hold for the $f_i$.
 \end{lemma}

\begin{proof} 
Let $\Omega$ be the neighbourhood of the identity in $\F(U)$ such that (1) and (2) hold for the $f_j'$. 
By Propositions \ref{prop:equivalent-topologies-phi} and \ref{prop:equivalent-topologies-D}, making $\Omega$ smaller we can guarantee that any $\Phi\in\Omega$ has a matrix $(a_{ij})$ relative to the $f_i$ such that $||(a_{ij})-I||_{\overline{U'}}<1/2$ and we can guarantee that $D$ satisfies $D(f_i)=\sum d_{ij}'f_j$ where $||(d_{ij}')||_{\overline{U'}}<\log 2$. Then since $\exp D=\Phi$, the series $(\log\Phi^*)f_i$ converges to $D(f_i)=\sum d'_{ij}f_j$ as well as to $\sum d_{ij}f_j$ where $\log(a_{ij})=(d_{ij})$.
Hence   (1) and (2) hold for the $f_i$.
\end{proof}

\begin{proposition}\label{prop:local-LF-implies-global} Let $X$ be a Stein $G$-manifold and let $K\subset Q$ be compact. 
Suppose that every fibre of $X$ has property (LF).  Let $U$ be a neighbourhood of $K$ such that we have a standard generating set  $\{f_i\}$ for $\O_\gf(X_U)$. Then for any   neighbourhood $U'$ of $K$ with compact closure in $U$ there is  a neighbourhood $\Omega$ of the identity in $\F(U)$ such that  every $\Phi\in\Omega$ admits a logarithm in $\LF(U')$.
\end{proposition}

\begin{proof} We may assume that $U$ is Stein and we can replace $X$ by $X_U$.
By Remark \ref{rem:(1)-is-automatic} we may choose $\Omega$ such that every $\Phi\in\Omega$ satisfies (1) of Definition \ref{def:propertyLF}.  Let $q\in \overline{U'}$. Since $X_q$ has property (LF), there are neighbourhoods $U_q\supset U_{q}'$ of $q$ and a neighbourhood $\Omega_q$ of the identity in $\F(U_q)$ such that if the restriction of $\Phi$ to $X_{U_q}$ lies in $\Omega_q$, then $\Phi$ admits a logarithm   $D_q\in\LF(U_q')$. By Remark \ref{rem:matrix-not-matter}, $D_q$ is uniquely determined by $\Phi$ restricted to $X_{U_q'}$. Since $\overline{U'}$ can be covered by finitely many $U'_{q_i}$, we can shrink $\Omega$ so that it maps into $\Omega_{q_i}$ for all $i$ under the continuous restriction maps $\F(U)\to\F(U_{q_i})$. Then every $\Phi\in\Omega$ admits a logarithm in $\LF(U')$.
  \end{proof}

Our goal now is to prove Proposition \ref{prop:all-type-LF} which says   that every fibre of every Stein $G$-manifold $X$ has property (LF). 

\begin{remark}\label{rem:LF-problem-local}
Suppose that we have another Stein $G$-manifold $X_0$ with quotient $Q_0$ and that we have a $G$-equivariant biholomorphism $\Psi\colon X\to X_0$ inducing $\psi\colon Q\to  Q_0$. Then $\Psi$ induces an isomorphism  of $\Aut_Q(X)^G$ with $\Aut_{Q_0}(X_0)^G$, sending $\Phi\in\Aut_Q(X)^G$ to $\Psi\circ\Phi\circ\Psi\inv$. Similarly, we obtain isomorphisms of $\Der_Q(X)^G$ and $\Der_{Q_0}(X_0)^G$, of $\F(Q)$ and $\F(Q_0)$, etc. It follows that a fibre $X_q$ has property (LF) in $X$ if and only if $\Psi(X_q)$ has property (LF) in $X_0$.   Hence by the slice theorem  we are reduced to showing that  fibres $G\times^H\NN(W)$ in  $G$-vector bundles $T_W$ have property (LF).
\end{remark}

  \begin{lemma}\label{lem:product-with-S}
Let $S$ be a Stein manifold with trivial $G$-action and let  $\{f_i\}$ be  a standard set  of generators of $\O_\gf(X)$. Let $X_q$ be a fibre of $X$ with property (LF) and let $s_0\in S$. Then $\{s_0\}\times X_q$ has property (LF) in $S\times X$.
\end{lemma}

\begin{proof}
There are arbitrarily small neighbourhoods of $(s_0,q)$ of the form $U_S\times U_q$ where $U_S$ is a neighbourhood of $s_0$ in $S$ and $U_q$ is a neighbourhood of $q\in Q$. Any $\Phi\in\F(U_S\times U_q$) gives us a smooth family $\Phi_s\in\F(U_q)$, $s\in U_S$. Since $X_q$ has property (LF), there is a neighbourhood  $U'$ of $q$ with compact closure in $U_q$ and a   neighbourhood $\Omega$ of the identity in $\F(U_q)$ such that any $\Phi\in\Omega$ admits a logarithm in $\LF(U')$. Let $U_S'$ be a neighbourhood of $s_0$ with compact closure in $U_S$. Let $\widetilde \Omega=\{\Phi\in\F(U_S\times U_q)\mid\Phi_s\in\Omega\text{ for all }s\in \overline{U_S'}\}$. Then $\widetilde \Omega$ is a neighbourhood of the identity in $\LF(U_S\times U_q)$. Shrinking $\widetilde\Omega$ we can arrange that  any $\Phi\in\widetilde\Omega$ has a matrix $(a_{ij})$ with $||(a_{ij})-I||_{\overline{U_S'}\times \overline{U'}}<1/2$. By the choice of $\widetilde\Omega$, for each $s\in  U_S'$ there is a $D_s\in\LF(U')$ with matrix $\log(a_{ij}(s,\cdot))$ 
where the $a_{ij}(s,\cdot)$ are smooth by Proposition \ref{prop:aij-exist}. Thus $D_s$ varies smoothly in $s$ and we have an element in $\LF(U_S'\times U')$ which is the logarithm of $\Phi$.
\end{proof}

 We are reduced to the case that $F=G\times^H\NN(W)$ where $W^H=0$ and $X=T_W$. 
 \begin{lemma}
Let $F$ and $T_W$ be as above. If $F$ is a principal fibre, then it has property (LF).
\end{lemma}

\begin{proof} Let $\{f_i\}$ be a standard generating set for $\O_\alg(T_W)$ such that the $f_i$ are linearly independent when restricted to $F$.
Since $F$ is principal, the only closed orbit  in the $H$-module $W$ is the origin, hence $X=F=T_W$ and $Q$ is a point. Thus $\F(Q)=\Aut(F)^G$.
By choice of the $f_i$, any $\Phi\in\Aut(F)^G$ has a unique matrix $(a_{ij})$. Any sufficiently small    neighbourhood  $\Omega$ of the identity in $\Aut(F)^G$   is the isomorphic image of a neighbourhood of the origin in its Lie algebra via the exponential map.  We may define such an $\Omega$ by the condition that $\Phi^*f_i=\sum a_{ij}f_j$ and $\vert(a_{ij})-I\vert<\epsilon\leq 1/2$ for $\epsilon$ sufficiently small.
\end{proof}

Let $W$ and $H$ be as above.  Even though we have reduced to considering the fibre $F\in T_B\subset T_W$, it is useful to prove lemmas concerning the case that $X=S\times T_B\subset S\times T_W$. Recall that for $\Phi\in\F(S\times Q_B)$, $\Phi_t(s,x)=t\inv\cdot\Phi(s,t\cdot x)$, $t\in(0,1]$, $(s,x)\in S\times T_B$.  Using the argument of Remark \ref{rem:LF-problem-local} it is easy to see that each $\Phi_t$ is of type $\F$, $0<t\leq 1$, and  $\Phi_0=\lim\limits_{t\to 0}\Phi_t$ is a smooth family of $G$-biholomorphisms, hence of type $\F$.  Let $\F([0,1]\times S\times Q_B)$ denote smooth families $\Phi(t,s,x)$  of elements of  $\F(Q_B)$ parameterised by $[0,1]\times S$.  
  
 Let $\{f_1,\dots,f_n\}$ be a standard generating set for $\O_\alg(T_W)$ as in Definition \ref{def:standard-generators}.  The $f_i$ are linearly independent when restricted to $F$. Let $d_i$ be the degree of $f_i$, $i=1,\dots,n$. We assume that $d_i=0$ for $1\leq i\leq \ell$, $d_i=1$ for $\ell<i\leq m$ and $d_i>1$ for $i>m$.  Let $\Phi\in\F(S\times Q_B)$ correspond  to the matrix $(a_{ij})$. By   Remark \ref{rem:fundam} we know that  
 \begin{equation}\label{eq:limit-zero}
a_{ij}(s,t\cdot x)=t^2\tilde a_{ij}(t,s,x),\   \ell<i\leq m,\  1\leq j\leq \ell,\  s\in S,\ x\in T_B,\  t\in[0,1]
 \end{equation}
where $\tilde a_{ij}(t,s,x)$ is smooth in $t$, $s$ and $x$. We also know that 
$\Phi_t^*f_i=\sum b^t_{ij}f_j$ where the $b_{ij}^t(s,x)$ are smooth in $t$, $s$ and $x$.

 \begin{lemma}\label{lem:Phi_t-is-continuous}
The mapping from $\F(S\times Q_B)$ to $\F([0,1]\times S\times Q_B)$ sending 
$\Phi$ to $\Phi_t$ is    continuous.
\end{lemma}

\begin{proof} The reasoning above shows that $\Phi_t\in\F([0,1]\times S\times Q_B)$. From  Lemma \ref{lem:limit-P-alpha-beta} and Corollary \ref{cor:fundam} we know that the $b^t_{ij}$ are polynomials, with coefficients in $\O_\alg(T_W)^G[t]$, in the $\tilde a_{ij}(s,t,x)$, $1\leq i\leq m$, where the $\tilde a_{ij}(t,s,x)$ are as given above or are just simply $a_{ij}(s,t\cdot x)$. Thus to show that $\Phi\mapsto\Phi_t$ is continuous it is enough to show that $a_{ij}(s,x)\mapsto t^{-2}a_{ij}(s,t\cdot x)$ is continuous in the $\ci$-topologies for $\ell<i\leq m$ and $1\leq j\leq m$. But this follows from  
Taylor's theorem. 
\end{proof}

 \begin{corollary}\label{cor:open-sets-of-families}
 Let $\Omega$ be a neighbourhood of the identity in $\F(S\times Q_B)$. Then there is a neighbourhood $\Omega_0$ of the identity  such that $\Phi_t\in\Omega$, $t\in[0,1]$, for all $\Phi\in\Omega_0$.
 \end{corollary}
  
 \begin{lemma}\label{lem:bij}
 Let $K$ be a compact subset of $S\times Q_B$. Then there is a neighbourhood $\Omega$ of the identity in $\F(S\times Q_B)$ such that any $\Phi\in\Omega$  
 corresponds to a matrix $(b_{ij})$ 
 that, as a matrix function on $S\times Q_B$, satisfies   $||(b_{ij}^t)-I||_K<1/2$, $0\leq t\leq 1$.  
 \end{lemma}
 \begin{proof}

Let $E_0$ be as in the discussion before Proposition \ref{prop:equivalent-topologies-phi} with $X=S\times T_B$.  Let $E_0'$ denote the closed linear subspace of $E_0$ consisting of matrices $(a_{ij})$ such that \eqref{eq:limit-zero} holds. Construct   $(b_{ij})$ as in Corollary \ref{cor:fundam-with-S}. Then 
by the proof of Lemma \ref{lem:Phi_t-is-continuous} the corresponding function 
$\rho\colon (a_{ij})\mapsto (b^t_{ij})$ is continuous for the $\ci$-topologies on $E_0'$ and $E_0\comptensor\ci([0,1])$ and, by construction, $\rho(I)=I$. Let 
 $$
 \Gamma=\{(b_{ij}^t)\in E_0\comptensor\ci([0,1]) :  ||(b_{ij}^t)-I||_K<1/2,\ 0\leq t\leq 1\}.
 $$
 Then there is a neighbourhood $\Delta'$ of the identity in $E_0'$ such that  $\rho(\Delta')\subset \Gamma$. Let $\Delta$ be a neighbourhood of the identity in $E_0$ such that $\Delta\cap E_0'\subset\Delta'$.   By the open mapping theorem (see the proof of Proposition \ref{prop:equivalent-topologies-phi}), if $\Phi\in\F(S\times Q_B)$ is sufficiently close to the identity, then it has a corresponding matrix $(a_{ij})\in \Delta$. Since $(a_{ij})$ is automatically in $E_0'$, we are in the open set $\Delta'$ and $(b_{ij}^t)$ satisfies the desired inequality. 
 \end{proof}
 
The graded $\O_\alg(Q)$-module $\Der_{Q,\alg}(T_W)^G$ has a generating set $A_1,\dots,A_k$ such that $A_i(t\cdot x)=t^{n_i} A_i(x)$ for some $n_i\in\Z^+$, $i=1,\dots,k$,   $x\in T_W$, $t\in\C^*$. If $n_i=0$, then $A_i([e,w])=A_i([e,0])$ is fixed by $H$, hence $A_i$ belongs to the Lie algebra of 
the centraliser $C_G(H)$ where $c\in C_G(H)$ acts on $T_W$ sending $[g,w]$ to $[gc,w]$. Since the $C_G(H)$-action commutes with the $\C^*$-action, $A_i$ is fixed by $\C^*$.  For $n_i\geq 1$, $A_i$ acts as an endomorphism of $\O_\alg(T_W)$ of degree $n_i-1$. Define $m_i$ to be $0$ if $n_i=0$, otherwise set $m_i=n_i-1$. Then $t\inv\cdot A_i(t\cdot x)=t^{m_i}A_i(x)$, $t\in\C^*$, $x\in T_W$, $i=1,\dots,k$.

Let $p=(p_1,\dots,p_d)$ where  the $p_i$ are  generators of $\O_\alg(T_W)^G$ and $p_i$ is homogeneous of degree $e_i$, $i=1,\dots,d$. Then we can identify $Q=T_W\sl G$ with the image of $p$. Let   
$K=\{(q_1,\dots,q_d)\in Q\mid \sum_i |q_i|^{2c_i}=1\}$
where we choose the $c_i\in\N$ such that $e=c_ie_i$ is independent of $i$. Then $K$ is compact and the points $t\cdot q$, $q\in K$, $t\in[0,1)$, form a neighbourhood $U'$ of $0$ in $Q$ where $\overline{U'}$ is compact. Changing the $p_i$ by scalars we may arrange that $\overline{U'}\subset Q_B$. Let $U$ be a relatively compact neighbourhood of the origin in $Q_B$ which contains $\overline{U'}$.
 By induction we can assume that all fibres of $X\setminus F$ have property (LF). Let $U''$ denote $U'\setminus\{0\}$.

\begin{lemma}\label{lem:LF-in-neighbourhood} Let $X=T_B$, etc.\ be as above. Then
there is a neighbourhood $\Omega$ of the identity in $\F(Q_B)$ such that any $\Phi\in\Omega$   admits a logarithm in $\LF(U'')$.
\end{lemma}

\begin{proof}  
Let $U_0$ be a neighbourhood of $K$  with compact closure in $U$. 
By Proposition \ref{prop:local-LF-implies-global}, there is a   neighbourhood  $\Omega_0$ of the identity in $\F(Q_B)$ such that any $\Phi\in\Omega_0$   admits a logarithm in $\LF(U_0)$.  
By Corollary \ref{cor:open-sets-of-families} and Lemma \ref{lem:bij} there is a neighbourhood $\Omega$ of the identity in $\F(Q_B)$  such that for  $\Phi\in\Omega$, $\Phi_t\in\Omega_0$, $t\in[0,1]$, and
$||(b_{ij}^t)-I||_{\overline U}<1/2$,  $t\in[0,1]$, where $(b_{ij})$ is a matrix corresponding to $\Phi$.    We have a smooth family $D_t\in\ci([0,1])\comptensor\LF(U_0)$ where $D_t=\log\Phi_t$. Hence $D_t
=\sum a_i(t,x) A_i(x)$, $t\in[0,1]$, $x\in X_{U_0}$, where the $a_i$ are smooth and invariant. Let $(d_{ij})=
\log (b_{ij})$.  Then a straightforward calculation shows that the family $(d^t_{ij})$ is the logarithm of the family 
$(b_{ij}^t)$, $0\leq t\leq 1$. Since $||(b_{ij}^t)-I||_{\overline U}<1/2$,  the vector field   $D_t$ has matrix     $(d_{ij}^t)$.  

Let $t\in(0,1]$ and define  $U''_{t}=\{q\in U''\mid t\inv \cdot q\in U_0\}$. 
For 
$x\in X_{U''_{t}}$  
$$
D_0(x):=t\cdot D_{t}(t\inv x)=\sum a_i(t,t\inv\cdot x)(t\inv)^{m_i}A_i(x)
$$
is of type $\LF$. Since $D_{t}$ has matrix $(d_{ij}^{t})$, the vector field $D_0$ has matrix $(d_{ij})$. Hence $\Phi$ admits a logarithm on $U''_{t}$ with matrix $(d_{ij})$.
Since $U''$ is the union of the $U''_{t}$ it follows that $\Phi$ admits a logarithm $D\in\LF(U'')$.
\end{proof}

We will have shown that all fibres of any $X$ have property (LF) if we can extend our $D\in\LF(U'')$ to $\LF(U')$. Now there is one case where this is automatic. Suppose that our $\Phi$ is holomorphic. Then $D\in\Der_{U''}(X_{U''})^G$. Near any point of $F$, there is a system of local coordinates consisting of invariant functions (which $D$ annihilates) and some of the $f_i$ (which $D$ sends into bounded linear combinations of the $f_j$). By the Riemann extension theorem, $D$ extends to an element of $\Der(X_{U'})$, which actually has to be in $\Der_{U'}(X_{U'})^G$. Now  the $\CC$-topology and the $\ci$-topology are the same on   
$\Aut_Q(X)^G$. Hence we have the following result.

\begin{theorem}\label{thm:log-holomorphic-diffeomorphism}
Let $X$ be a Stein $G$-manifold and assume that we have  a standard generating set $\{f_i\}$ for $\O_\gf(X)$. Let $U$ be a neighbourhood of the compact subset $K\subset Q$ and let     $U'$ be a neighbourhood of $K$  with compact closure in $U$. Then there is  a compact subset $K'\subset U$  and $0<\epsilon<1/2$, with the following property. If $\Phi\in \Aut_U(X_U)^G$  has an associated matrix $(a_{ij})$ satisfying $||(a_{ij})-I||_{K'}<\epsilon$, then there is a $D\in\Der_{U'}(X_{U'})^G$ with associated matrix $\log (a_{ij})$ such that $\exp D=\Phi$. 
\end{theorem}

Let $\Phi\in\Omega\subset\F(T_B)$ as in Lemma \ref{lem:LF-in-neighbourhood} so that $\Phi$ has a (unique) logarithm   $D\in \LF(U'\setminus \{0\})$. We now show that if $\Phi$ is sufficiently close to the identity, then $\Phi$ has a logarithm over \emph{some\/} neighbourhood of $0\in U'$. By uniqueness of logarithms, this shows that $D$ extends to $\LF(U')$ and we will have shown that $F$ has property (LF).

We need to study various $G$-automorphisms of $F=G\times^H\NN(W)$.
As before $\Aut_\vb(T_W)^G$ denotes the $G$-vector bundle automorphisms of $T_W$. Let $\rho\colon H\to\GL(W)$ be the representation associated to $W$. Then one easily shows:

\begin{lemma}\label{lem:G-vector-bundle-automorphism-structure}
Let $\Psi\in\Aut_\vb(T_W)^G$.  Then there is  a $g_0\in N_G(H)$ and   $\gamma\in \GL(W)$ normalising $\rho(H)$ such that:
\begin{enumerate}
\item For $g\in G$ and $w\in W$, $\Psi([g,w])=[gg_0,\gamma(w)]$.
\item For $h\in H$, $\rho(g_0\inv hg_0)=\gamma\circ \rho(h)\circ\gamma\inv$.
\end{enumerate}
\end{lemma}

 Let $L_\vb$ denote the subgroup of  $\Aut_\vb(T_W)^G$ fixing $\O_\alg(W)^H$.
\begin{corollary}\label{cor:Lvbreductive}
The groups $\Aut_\vb(T_W)^G$ and $L_\vb$ are reductive.
\end{corollary}

\begin{proof}
We have a homomorphism $\sigma\colon\Aut_\vb(T_W)^G\to N_G(H)/H$, $(g_0,\gamma)\mapsto g_0H$. The image clearly contains $C_G(H)H/H$, which in turn contains the identity component of $N_G(H)/H$. Hence $\Im\ \sigma\subset N_G(H)/H$ is reductive. The kernel of $\sigma$ is $\rho(H)\cdot \GL(W)^H$ which is isomorphic to   $\rho(H)\times\GL(W)^H$ divided by the centre of $\rho(H)$. Hence $\Aut_\vb(T_W)^G$ is reductive.

Let $(g_0,\gamma)$ represent an element $\Psi$ of  $\Aut_\vb(T_W)^G$. Then $\Psi$ acts on the homogeneous elements of $\O_\alg(W)^H$ of degree $d$, sending $f$ to $f\circ\gamma\inv$. This gives us representations $\sigma_d$ of $\Aut_\vb(T_W)^G$ such that $L_\vb$ is the joint kernel of (finitely many of) the $\sigma_d$. Hence $L_\vb$ is reductive.
\end{proof}

Let $L$ denote $\Aut(F)^G=\Aut(\O_\alg(F))^G$, a linear algebraic group. 
Let $\ell\in L$. Then $\ell$ preserves the closed orbit in $F$, which is $Z$. We have the deformation $\ell_t\in L$ where $\ell_t(x)=t\inv\cdot\ell(t\cdot x)$, $t\in (0,1]$, $x\in F$. The limit as $t\to 0$ is the normal derivative $\delta\ell$ of $\ell$ along $Z$, so $\delta\ell\in\Aut_\vb(T_W)^G$.  Let $\ell'\in L$. Since $\ell'_t\circ\ell_t=(\ell'\circ\ell)_t$, the map $\delta\colon L\to\Aut_\vb(T_W)^G$ is a homomorphism of algebraic groups. If $\ell\in \Aut_\vb(T_W)^G$, then $\delta(\ell|_F)=\ell$.

Let $L_\hr$ denote the subgroup   of automorphisms that extend to be   $G$-equivariant biholomorphisms over the identity in a $G$-saturated neighbourhood of $F$ (the ``holomorphically reachable points'').   Then $L_\vb\subset L_\hr\subset L$.  Let $A_1,\dots,A_k$ be our minimal homogeneous generators of $\Der_{Q,\alg}(T_W)^G$ as before, where $t\inv\cdot A_i(t\cdot x)=t^{m_i}A_i(x)$, $t\in\C^*$, $x\in T_W$ and $m_i\in\Z^+$. Since $L_\vb$ is reductive and acts by conjugation in a degree preserving way on $\Der_{Q,\alg}(T_W)^G$, we may assume that $L_\vb$ preserves the span of the $A_i$ for any fixed $m_i$.  It is easy to see that the Lie algebra of $L_\vb$ is the span of the $A_i$ with $m_i=0$. Each $A_i$ acts on $\O_\alg(T_W)$ increasing degree by $m_i$. The restrictions $\widetilde A_i$ of the $A_i$ to $F$ are linearly independent (by minimality). 
Let $\tilde{\lie r}$ be the span of the $\widetilde A_i$ with $m_i>0$. Let $s$ be the maximum degree of the $f_i$. Then any element $\widetilde A$ of $\tilde{\lie r}$ acts nilpotently on $\O_\alg(F)$ since the 
 $(s+1)$st power annihilates all the $f_i$.  It follows that $\exp \widetilde A$ is a unipotent  algebraic $G$-automorphism of $F$. Let $R$ denote $\exp(\tilde{\lie r})$, an algebraic subgroup of $L$. By our choice of the $A_i$, the conjugation action of $L_\vb$ preserves $\tilde{\lie r}$. We have a morphism $\sigma$ of algebraic groups 
 $$
 L_\vb\ltimes R\ni (\ell,\exp\widetilde A)\mapsto\ell\exp\widetilde A\in L 
 $$
with image in $L_\hr$.

\begin{theorem}\label{thm:reachable}
The homomorphism $\sigma\colon L_\vb\ltimes R\to L$ is injective with image $L_\hr$. Hence $L_\hr$ is an algebraic subgroup of $L$  and its Lie algebra is the restriction of $\Der_{Q,\alg}(T_W)^G$ to $F$.
\end{theorem}

 \begin{proof}
Let $\Phi\in L_\hr$. We want to show that $\Phi$ is in the image of $\sigma$. Now $\Phi$ is the restriction of some $\Psi\in\Aut_U(X_U)^G$ where $U$ is a neighbourhood of the origin in $Q$. We have the deformation $\Psi_t$ with limit $\Psi_0\in L_\vb$.  Replacing $\Psi$ by $\Psi_0\inv\circ\Psi$ and $\Phi$ by $\Phi_0\inv\circ\Phi$ we may reduce to the case that $\Psi_0$ and $\Phi_0$ are the identity. 
Since $\Psi_0$ is the identity,  for $t_0\in(0,1]$ sufficiently close to $0$, $\Psi_{t_0}$ admits a logarithm $D_{t_0}$ over a neighbourhood $U'$ of $0\in Q_B$ (Theorem \ref{thm:log-holomorphic-diffeomorphism}). The restriction $E$ of $D_{t_0}$ to $F$ is a sum $\sum a_i\widetilde A_i$.   Then $\Phi_{t_0t}=\exp E_t=\exp  (\sum t^{m_i}a_i\widetilde A_i)$. Letting $t$ tend to zero we see that $\sum_{m_i=0}a_i\widetilde A_i$ exponentiates to the identity. Since $\exp$ is injective on a neighbourhood of $0$ in the Lie algebra of $L_\vb$, we see that all the $a_i$ for which $m_i=0$ vanish, provided they are sufficiently small, which we can always arrange. Hence $\Phi=\exp(\sum_{m_i>0}t_0^{-m_i}a_i\widetilde A_i)$. Thus $\Phi\in R$, the image of $\sigma$ is $L_\hr$  and $L_\hr$ is algebraic. Finally, 
$\sigma$ is injective, since $\delta\circ \sigma(\ell,\exp \widetilde A)=\ell$ 
for $\ell\in L_\vb$ and $\widetilde A\in \tilde{\lie r}$.
 \end{proof}

\begin{corollary}\label{cor:extendL0}
There is a group homomorphism $\nu \colon L_\hr\to\Aut_{Q}(T_W)^G$ such that $\ell=\nu(\ell)|_F$ for all $\ell\in L_\hr$.
\end{corollary}

\begin{proof}
Let $\lie r$ be the span of the $A_i$ such that $m_i>0$. By our choice of the $A_i$,  $\lie r$ is stable under conjugation by $L_\vb$.  Let $\ell\in L_\vb$ and let $\widetilde A\in\tilde{\lie r}$ be the restriction of $A\in\lie r$. 
Let $\nu(\ell\cdot\exp\widetilde A)=\ell\cdot\exp A$. Then $\nu$ has the required properties.
\end{proof}

 \begin{remark}\label{rem:extendL0}
Let $S\times T_B$ be a standard neighbourhood in $X$ and let $F$ denote the fiber $(s_0,G\times^H\NN(W))$ for some $s_0\in S$.   Let $U$ denote $S\times Q_B$. Then we   have an extension mapping from $L_\hr$ to $\Aut_U(X_U)^G$ where the action of $\ell\in L_\hr$ on $(s,x)\in S\times T_B$ gives $(s,\nu(\ell)(x))$.
 \end{remark}

  \begin{proposition}\label{prop:all-type-LF}
Let $X$ be a Stein $G$-manifold. Then all fibres have property (LF).
\end{proposition}

 \begin{proof}
We continue with the situation of Lemma \ref{lem:LF-in-neighbourhood}. We may assume that our standard generators $f_i$ are linearly independent when restricted to $F$. We have only to show that, for some neighbourhood $\Omega$ of the identity of $\F(Q_B)$, any $\Phi\in\Omega$ is $\exp D$, $D\in\LF(U)$, where $U$ is a neighbourhood  of $0$ (which may depend upon $\Phi$). By choosing $\Omega$ small, we may assume that the restriction  of $\Phi$ to $F$ is the restriction to $F$ of $\exp D'$ where $D'\in\Der_{Q,\alg}(T_W)^G$. We may assume that the corresponding matrix $(d_{ij}')$ has norm at most $(\log 2)/2$ over a neighbourhood of $0\in Q$. Thus, replacing $\Phi$ by $\exp(-D')\Phi$, we may reduce (temporarily) to the case that $\Phi$ restricted to $F$ is the identity.

Since $\Phi$ is of type $\F$,   there is a   holomorphic extension $\Psi(x,y)$ of $\Phi$ on $T_{B'}\times T_{B'}$ where $0\in B'\subset B$. Then $\Psi$ is a smooth family $\Psi^x$ of elements of $\Aut_{Q_{B'}}(T_{B'})^G$, $x\in T_{B'}$. The restriction of $\Psi^{x_0}$ to $F$ is the identity. We have the deformations $\Psi^x_t$, $t\in[0,1]$, where $\Psi^{x_0}_0$ is the identity. Since $\Psi^x_t(y)$ is smooth in $t$, $x$, $y$ and holomorphic in $y$, the corresponding matrix $(b_{ij}^{t,x}(y))$ is smooth in $t$, $x$ and $y$, holomorphic in $y$ and   $(G\times G)$-invariant. Since $(b_{ij}^{t,x_0}(x_0))=I$, $0\leq t\leq 1$, there is a $(G\times G)$-saturated neighbourhood of $(x_0,x_0)$ on which we have $|b_{ij}^{t,x}(y)-I|<1/2$, $0\leq t\leq 1$. Shrinking $B'$ we may thus assume that 
$||(b_{ij}^{t,x}(y))-I||<1/2$, $0\leq t\leq 1$, for $(x,y)\in T_{B'}\times T_{B'}$. Since $(b^{t,x}_{ij}(y))$ is $G$-invariant in $y$, we may also write $(b^{t,x}_{ij}(q))$.  Let $U'$ be a neighbourhood of $0\in Q$ with compact closure in $Q_{B'}$. By Theorem \ref{thm:log-holomorphic-diffeomorphism} there is an $\epsilon>0$   and a compact subset $K\subset Q_{B'}$ such that $||(b^{t,x}_{ij}(q))-I||_K<\epsilon$ implies that $\Psi^x_t(y)$ admits a logarithm in $\Der_{U'}(X_{U'})^G$. Since the inequality is true for $t=0$ and $x=x_0$, there is a neighbourhood $\Lambda$ of $(0,x_0)\in [0,1]\times  T_{B'}$ such that $\Psi_t^x$ admits a logarithm in $\Der_{U'}(X_{U'})^G$ for $(t,x)\in\Lambda$. Now $\Lambda$ contains a   subset of the form $[0,t_0]\times\Delta$ where $t_0>0$ and $\Delta$ is a neighbourhood of $x_0$ in $T_{B'}$. Let $U_0$ denote $t_0\cdot U'$. Then the usual trick shows that $\Psi^x$ admits a logarithm in $\Der_{U_0}(X_{U_0})^G$  for all $x\in\Delta$. Since $\Psi(x,y)$ is $G$-invariant in $x$, we may assume that $\Delta$ is $G$-saturated. Shrinking $\Delta$ we arrive at the situation where  $\Psi^x(y)$ admits a logarithm $D^x(y)$ for $(x,y)\in\Delta\times\Delta$. Then $\Phi(x)=\Psi^x(x)=\exp D^x(x)$ where $D'':=D^x(x)$ is of type $\LF$ on $\Delta$. Since $\Phi$ is the identity  on $F$, $D''$ vanishes on $F$.
  
Now $D''$ corresponds to $(d''_{ij})$ where $||(d''_{ij})||<(\log 2)/2$ on a neighbourhood of $0\in Q$. We have shown that our original $\Phi=\exp D'\exp D''$. Then $\Phi=\exp D$ where $D$ is given by the Campbell-Hausdorff series
$$
D=D'+D''+\frac12[D',D'']+\frac1{12}([D',[D',D'']]+[D'',[D'',D']])- \dots
$$
 Since the   coefficient matrices $(d'_{ij})$ and $(d''_{ij})$ have norm at most $(\log 2)/2$  in a neighbourhood of $F$, the series converges there. It is a series of elements of type $\LF$, hence the limit is of type $\LF$ 
 by Theorem \ref{thm:closed}. The coefficient matrix $(d_{ij})$ of $D$ has norm less than $\log 2$ in a neighbourhood of $F$, which we may assume is $T_{B'}$, and     $\Phi$ has matrix $(a_{ij})$ where $|(a_{ij}(x))-I|<1/2$ for  $x\in T_{B'}$. Since $\log\exp (d_{ij})=(d_{ij})$ on $T_{B'}$, the series $\log\Phi^*$, applied to $f_i$, gives $\sum d_{ij}f_j$. But $(\log\Phi^*)f_i$ is also $\sum e_{ij}f_j$ where $e_{ij}=\log (a_{ij})$. Hence $D$ also has matrix $(e_{ij})$ and is indeed the logarithm of $\Phi$. Hence, in the situation of Lemma \ref{lem:LF-in-neighbourhood}, choosing $\Omega$ sufficiently small, for all $\Phi\in\Omega$, the vector field $D$ constructed there extends to an element of $\LF(U')$. Thus $F$ has property (LF).
\end{proof}

 Combining our results so far (including Lemma \ref{lem:series}) we have: 
 
 \begin{theorem}\label{thm:Dexists}
 Let $X$ be a Stein $G$-manifold, let $K\subset Q$ be compact and choose a standard generating set of $\O_\gf(X_U)$ for $U$   a neighbourhood of $K$.  Let $U'\subset U$ be a  neighbourhood  of $K$ with compact closure in $U$. Then there is a neighbourhood $\Omega$ of the identity in $\F(U)$  such that any $\Phi\in\Omega$ admits a (unique) logarithm in $\LF(U')$. Moreover, the mapping $\log\colon\Omega\to\LF(U')$ is continuous.
  \end{theorem}

\begin{corollary}\label{cor:typeFclosed}
Let $\{\Phi_n\}$ be a sequence of $G$-diffeomorphisms of $X$ of type $\F$ and suppose that $\Phi_n$ converges to   a $G$-diffeomorphism $\Phi$. Then $\Phi$ is of type $\F$.
\end{corollary}

\begin{proof} Since this is a local question, we can assume that we have a standard generating set $\{f_i\}$ for $\O_\gf(X)$. Let $q\in Q$  and let $U$ be a neighbourhood of $q$ with compact closure. Then  there is a neighbourhood $\Omega$ of the identity in $\F(Q)$  such that any $\Psi\in\Omega$   admits a logarithm in $\LF(U)$. Let $\Omega_0$ be a smaller neighbourhood of the identity with $\overline{\Omega}_0\subset\Omega$. There is an $N\in\N$ such that $n\geq N$ implies that $\Phi_N\inv\Phi_n\in\Omega_0$, hence $\log(\Phi_N\inv\Phi_n)=D_n\in\LF(U)$, and $D_n$ converges  to a vector field 
$D$ which is in $\LF(U)$ by Theorem \ref{thm:closed}.  Since $\exp D_n=\Phi_N\inv\Phi_n$ over $U$, we have   $\exp D=\Phi_N\inv\Phi$ over $U$. Hence $\Phi=\Phi_N\exp D$ is of type $\F$ over $U$. 
\end{proof}

  \section{Topology}\label{sec:topologies}

Let $X$  be a Stein $G$-manifold with quotient $Q$.  Let $R$ be a $G$-module where  $R=\oplus R_i^{\oplus c_i}$ and the $R_i$ are irreducible. Define $\O(X)_R$ to be $\oplus(\O(X)_{R_i})^{\oplus c_i}$ which is an $\O(X)^G$-submodule of $\O(X)^c$, $c=\sum c_i$.   We   have the space $\ci(X)^G\cdot\O(X)_R\subset\ci(X)^c$ of smooth functions which are locally (over $Q$) finite sums $\sum a_j f_j$ where the $a_j$ are smooth and invariant and the $f_j\in\O(X)_R$. The main point of this section is Theorem \ref{thm:closed} which  shows  that $\ci(X)^G\cdot\O(X)_R$ is closed in $\ci(X)^c$ and that $\ci(X)^G\cdot\Der_Q(X)^G$ is closed in $\Der_Q^\infty(X)^G$, the space of smooth $G$-invariant vector fields on $X$ which annihilate $\O(X)^G$.  Then it follows that our spaces are Fr\'echet and nuclear (see \cite[Sections 10, 50]{Treves}).   Theorem \ref{thm:closed} is a technical underpinning of our results in Section \ref{sec:typeF}.

\begin{theorem}\label{thm:closed} Let $X$ be a Stein $G$-manifold and let $R=\oplus R_i^{\oplus c_i}$   be a    $G$-module where the $R_i$ are irreducible.
\begin{enumerate}
\item The space  
$\LF(Q)=\ci(X)^G\cdot\Der_Q(X)^G$ is closed in $\Der_Q^\infty(X)^G$.
\item The space $\ci(X)^G\cdot\O(X)_R$ is closed in   $\ci(X)^c$ where $c=\sum c_i$.
\end{enumerate}
\end{theorem}

 If the quotient mapping $p\colon X\to Q$ were proper, then 
 the theorem would be a consequence of theorems of Bierstone and Milman \cite{BierstoneMilman},  Theorems A, C, D; \cite{BierstoneMilmanb}. Unfortunately, $p\colon X\to Q$ is proper only when $G$ is finite. Note that (1) and (2) above are equivalent to the subspaces   being complete in the induced topology.

Let $R=\oplus c_iR_i$ and $c=\sum c_i$ be as above. Define $\ci(X)_{R_i}$ to be the  sum of the elements of $\ci(X)$ which transform by the representation $R_i$ of $G$ and let $\ci(X)_R=\oplus(\ci(X)_{R_i})^{\oplus c_i}$.  Then $\ci(X)_R$ is closed in $\ci(X)^c$, hence Fr\'echet, and $\ci(X)_R$ contains $\ci(X)^G\cdot\O(X)_R$. Let $\Mor(X,R)$  (resp.\ $\Mor^\infty(X,R)$) denote the holomorphic (resp.\ smooth) maps of $X$ to $R$. We have the $G$-equivariant maps $\Mor(X,R)^G$ and $\Mor^\infty(X,R)^G$.  
Let $\sigma\colon \ci(X)_R\to\Mor^\infty(X,R^*)^G$ where $\sigma(f)$, $f\in\ci(X)_R$,  sends $x\in X$ to the element of $R^*$ sending $f$ to $f(x)$.   Let $\Gamma\in\Mor^\infty(X,R^*)$. Define $\tau(\Gamma)\in\ci(X)_R$  by $\tau(\Gamma)(x) = \Gamma(x)(r)$,  $x\in X$, $r\in R$.
\begin{lemma} 
\begin{enumerate}
\item The maps $\sigma$ and $\tau$ are inverses of each other, hence $\ci(X)_R$  and  $\Mor^\infty(X,R^*)^G$  are isomorphic Fr\'echet spaces.
\item The isomorphisms  $\sigma$ and $\tau$ restrict to isomorphisms of $\ci(X)^G\cdot\O(X)_R$ and $\ci(X)^G\cdot \Mor(X,R^*)^G$.
 \end{enumerate}
Hence $\ci(X)^G\cdot \O(X)_R$ is complete if and only if $\ci(X)^G\cdot\Mor(X,R^*)^G$ is complete.
\end{lemma}

We say that $X$ is \emph{good\/} if Theorem \ref{thm:closed} holds for $X$ 
and every $G$-module $R$.

\begin{proposition}\label{prop:slice}
Suppose that every slice representation of $X$ is good. Then $X$ is good.
\end{proposition}

\begin{proof} Since we have  invariant partitions of unity (Lemma \ref{lem:pi}), it is clear that goodness is a local condition over $Q$.
Take an open cover of $X$ by tubes $X_i=T_{B_i}$ corresponding to slice representations $(W_i,H_i)$ of $X$. If each $T_{W_i}$ is good, it is clear that each $T_{B_i}$ is good, hence $X$ is good. Thus we only need to show that $X=T_W=G\times^HW$ is good if $W$ is a good $H$-module and $H$ is a reductive subgroup of $G$. Note that $\ci(X)^G\simeq\ci(W)^H$.

First 
we show that $\ci(X)^G\cdot\Mor(X,R^*)^G$ is complete. Let $\rho\colon\Mor^\infty(X,R^*)^G\to\Mor^\infty(W,R_H^*)^H$ be the restriction mapping where $R^*_H$ is $R^*$ considered as an $H$-module and $W=W_e$ is the fibre of $T_W$ at $[e,0]$. We have an inverse $\eta$ to $\rho$ where $\eta(\Gamma)([g,w])=g\cdot\Gamma(w)$ for $\Gamma\in\Mor^\infty(W,R_H^*)^H$, $g\in G$ and $w\in W$. Clearly $\eta$ and $\rho$ induce   isomorphisms of $\ci(X)^G\cdot\Mor(X,R^*)^G$ and $\ci(W)^H\cdot\Mor(W,R_H^*)^H$. Since the latter space is complete by hypothesis, we have that $\ci(X)^G\cdot\Mor(X,R^*)^G$ is complete. 

Now  $\Der^\infty_Q(X)^G$ is the image of 
$$
\Der_{Q}^\infty(G\times W)^{G\times H}\simeq(\lie g\otimes\ci(W))^{H}\oplus \Der^\infty_Q(W)^H
$$
 with kernel the $\ci(G\times W)$-multiples of $\lie h$ which are $(G\times H)$-invariant where $\lie h$ sits diagonally inside $\Der_{Q}(G\times W)$ (since the action of $H$ is on both factors of $G\times W$).  We may write $\lie g$ as a direct sum $\lie h\oplus \lie m$ where $\lie m$ is $H$-stable, and then $\Der^\infty_{Q}(X)^G$ is the isomorphic  image of the direct sum $(\lie m\otimes\ci(W))^H\oplus \Der^\infty_{Q}(W)^H$ where $(\lie m\otimes\ci(W))^H\simeq\ci(W)_{(\lie m^*)}$.  We have an induced direct sum decomposition   
 $$
 \ci(X)^G\cdot\Der_Q(X)^G\simeq\ci(W)^H\cdot\O(W)_{(\lie m^*)}\oplus \ci(W)^H\cdot\Der_Q(W)^H.
 $$
Since $W$ is good, both $\ci(W)^H\cdot \O(W)_{(\lie m^*)}$ and $\ci(W)^H\cdot\Der_{Q}(W)^H$ are complete, hence $\ci(X)^G\cdot\Der_Q(X)^G$ is complete and $X$ is good.
\end{proof}

We want to show that every representation of every reductive group is good. So we need to show that every $G$-module $V$ is good.   Our inductive assumption will be that every proper slice representation $(W,H)$ of $V$ is good. We first reduce to the case that $V^G=0$. 

 Let $N$ and $P$ be smooth manifolds as before and let $N_0$ be a closed subset of $N$. Let $\ci(N,N_0)$ denote the smooth functions on $N$ which are flat on $N_0$, i.e., have vanishing Taylor series  on $N_0$. It is easy to see that $\ci(N\times P,N_0\times P)\simeq\ci(N,N_0)\comptensor\ci(P)$.

\begin{lemma}\label{lem:VG}
Write $V=V^G\oplus V'$  where $V'$ is a $G$-submodule of $V$. If $V'$ is good, then $V$ is good.
\end{lemma}
\begin{proof}
Note that $\ci(V)^G\simeq\ci(V^G)\comptensor\ci(V')^G$ and that
  $$
  \Der_{Q_V}(V)^G\simeq \O(V^G)\comptensor\Der_{Q_{V'}}(V')^G.
  $$
   If $\ci(V')^G\cdot\Der_{Q_{V'}}(V')^G$ is complete, then so is $\ci(V)^G\cdot \Der_{Q_{V'}}(V')^G$ since we are just taking completed tensor product with $\ci(V^G)$.  Similarly, for $R$ a $G$-module, 
$$
\O(V)_R\simeq\O(V^G)\comptensor  \O(V')_R.
$$ 
If $\ci(V')^G\cdot\O(V')_R$ is complete, then   
$$
\ci(V^G)\comptensor(\ci(V')^G\cdot \O(V')_{R})
$$ 
is complete. 
\end{proof}

The slice representation at a principal orbit of $V$ has the form $(W,H)$ where the closed $H$-orbits in $W$ are fixed points. Hence we start our induction with the following.

\begin{lemma}\label{lem:principal}
Let $V$ be a $G$-module whose closed orbits are all fixed points. Then $V$ is good.
\end{lemma}

\begin{proof}
By Lemma \ref{lem:VG} we may reduce to the case that  $Q_V$ is a point. Then $\ci(V)^G=\C$ and there is nothing to prove.
\end{proof}

 For the next few results, let $V_\R$ denote $V$ considered as a real vector space and let $G_\R$ denote $G$ considered as a real Lie group. 
 
 \begin{lemma}
 The algebra  $\R[V_\R]^{G_\R}$ is finitely generated.
 \end{lemma} 
 For lack of a good reference, we give the following:
 \begin{proof}
 It is enough to show that the invariants of $(G^0)_\R$ are finitely generated. Now $G^0$ is a product $G_sT$ where $G_s$ is connected semisimple, $T$ is a torus and the two factors commute. The invariants of $(G_s)_\R$ are the same as those of $(\lie g_s)_\R$ which is a real semisimple Lie algebra. Since any representation  of such a Lie algebra is completely reducible, we can assume that $G=T\simeq (\C^*)^d$ for some $d\geq 1$.
 
Now $T_\R$ is a product $(S^1)^d\times(\R^{>0})^d$. Since $(S^1)^d$ is compact, it acts completely reducibly, so we need only consider the factor $(\R^{>0})^d$. Let $W$ be an irreducible subspace of $V$. Then $T$ acts on $W\simeq \C$ by a character $\chi$ where $\chi(z_1,\dots,z_d)=z_1^{a_1}\cdots z_d^{a_d}$ for $(z_1,\dots,z_d)\in T$ and some $(a_1,\dots,a_d)\in\Z^d$. It follows that $(r_1,\dots,r_d)\in (\R^{>0})^d$ acts on $W_\R\simeq\R^2$  as multiplication by $\prod r_i^{a_i}$. It is now clear that $(\R^{>0})^d$ acts completely reducibly on $\R[V_\R]$.
 \end{proof}

 We need to consider a different quotient space $\widetilde Q$ for our consideration of smooth invariant functions on $V$. We have homogeneous generators $p_1,\dots,p_m$  of $\O_\alg(V)^G$. Let $\tilde p_1,\dots,\tilde p_{2m}$ be the real and imaginary parts of $p_1,\dots,p_m$. Choose homogeneous real invariant polynomials $\tilde p_{2m+1},\dots,\tilde p_n$ such that $\R[V_\R]^{G_\R}$ is generated by  the $\tilde p_j$. Then we have a quotient mapping 
 $$
 \tilde p=(\tilde p_1,\dots,\tilde p_n)\colon V\to  \R^{n}
 $$
 with image $\widetilde Q$. 
 \begin{remark}
 Luna    \cite{LunaSmooth} shows    that $\tilde p^*\ci(\R^n)$ is the set of smooth functions on $V$ constant on the fibres of $\tilde p$. In our case  the fibres of $\tilde p$ and $p$ are the same (as sets) and are of the form $G\times^H\NN(W)$. Any $G$-invariant continuous function must be constant on such a set. Hence $\tilde p^*\ci(\R^n)=\ci(V)^G$.
  \end{remark}

 \begin{example}\label{ex:not-smooth}
 Let $V=\C$ and $G=\{\pm 1\}$ acting by scalar multiplication. Then $p=z^2\colon V\to\C$. In real coordinates the real and imaginary parts of $z^2$ are $x^2-y^2$ and $2xy$.  The invariant smooth function $x^2+y^2$ is not of the form $h(x^2-y^2,2xy)$ where $h$ is smooth, so the real and imaginary parts of $p$ are not enough. One can take $\tilde p$  to be
  $$
 (x,y)\mapsto (x^2-y^2,2xy,x^2+y^2)\colon V\to\R^3.
 $$
 \end{example}

We need to use a form of polar coordinates on $\R^n$ (see \cite{LunaSmooth}). Let $e_i$ be the degree of $\tilde p_i$ and choose $b_i\in\N$ such that $b_ie_i=e$ is independent of $i=1,\dots, n$.
 Let $P=\{y\in\R^n\mid \sum y_i^{2b_i}=1\}$. Then $P$ is a compact smooth submanifold of $\R^n$. Consider the mapping $\tau\colon\R^+\times P\to \R^n$ sending $(t,y)$ to $t\cdot y$ where, as usual, $t\cdot y=(t^{e_1}y_1,\dots,t^{e_n}y_n)$.  Then $\R^+\cdot P=\R^n$.  From \cite[Lemma 1.5]{LunaSmooth} we have:
  \begin{lemma}\label{lem:flat}
 Pullback by $\tau$ gives an isomorphism $\ci(\R^n,0)\simeq \ci(\R^+,0)\comptensor\ci(P)$.
 \end{lemma}
 
Let $\Upsilon$ denote $\tilde p\inv(P)$. Then $\Upsilon$ is a smooth $G$-submanifold of $V$ since the differential of the homogeneous polynomial $\sum \tilde p_i^{2e_i}$ does not vanish on $\Upsilon$. We have a kind of polar coordinates on $V$ given by $\sigma\colon \R^+\times \Upsilon\to V$, $(t,v)\mapsto tv$.

 Now we need another version  of surjectivity of $\tilde p^*$ which can   be found in \cite[Section 3.3]{LunaSmooth}. 
  \begin{lemma}\label{lem:Luna}
 We have the equality
 $$
 (\id\times\tilde p|_{\Upsilon})^*(\ci(\R^+,0)\comptensor\ci(P))=\ci(\R^+,0)\comptensor\ci(\Upsilon)^G.
$$
 \end{lemma}

 \begin{corollary}\label{cor:flat}
Pullback by $\sigma$ induces an isomorphism $$
\sigma^*\colon \ci(V,\NN(V))^G\to\ci(\R^+,0)\comptensor\ci(\Upsilon)^G.
$$
 \end{corollary}
 
 \begin{proof}
 It is clear that $\sigma^*$ is injective. Let $f\in\ci(\R^+,0)\comptensor\ci(\Upsilon)^G$. Then $f$ is the pullback of some $h\in\ci(\R^+,0)\comptensor\ci(P)$. By
 Lemma \ref{lem:flat},   $h$ corresponds to an element of $\ci(\R^n,0)$ which in turn  pulls back to an element of $\ci(V,\NN(V))^G$. Hence $\sigma^*$ is surjective.
  \end{proof}

We need a version of E.\ Borel's lemma  \cite[Chapter IV, proof of Lemma 2.5]{Guillemin}.

\begin{lemma}\label{lem:Borel}
Let $f_0$, $f_1,\dots$ be a sequence of smooth functions on $P$. Let $\rho$  in $\ci([0,1))$  be real valued, have compact support and equal $1$ on a neighbourhood of $0$ in $[0,1)$.
Then there are positive increasing numbers $\lambda_j$, $\lim\limits_{j\to\infty}\lambda_j=\infty$, such that for $t\in\R^+$ and $x\in P$,
$$ 
\sum_{j=0}^\infty\rho(\lambda_j t) \frac {t^j}{j!}f_j(x)
$$
converges to a smooth function $f$ on $\R^+\times P$ such that
$$
\frac {\pt^j f}{\pt t^j}(0,x)=f_j(x),\ j=0,1,\dots
$$
\end{lemma}

\begin{proposition}\label{prop:Vgood}
Let $V$ be a $G$-module all of whose proper slice representations are good. Then $V$ is good.
\end{proposition}

\begin{proof}
We may assume that $V^G=0$. Let $A_1,\dots,A_k$ generate the $\O_\alg(Q_V)$-module $\Der_{Q_V,\alg}(V)^G$. Then they generate $\Der_{Q_V}(V)^G$.  We may assume that the $A_i$ are homogeneous. 
 We first show that the set of all sums $\sum_i h_i(v)A_i(v)$ is complete, $v\in \Upsilon$, $h_i\in\ci(\Upsilon)^G$. There is a retraction $\eta$ from $V\setminus\NN(V)$ to $\Upsilon$ which sends $v$ to $v/\rho(v)$ where $\rho(v)^{2d}=\sum_i\tilde p_i(v)^{2e_i}$ and  we choose $\rho(v)$ as the positive $2d$th root. If $\sum_i h_{i,\ell}(v) A_i(v)$ for $v\in \Upsilon$ has a limit $A$ as $\ell\to\infty$, then  we see that $A\circ\eta$ is   the limit of the $\sum_i h_{i,\ell}(v/\rho(v))\rho(v)^{-n_i}A_i(v)$ for $v\in V\setminus\NN(V)$ where $n_i$ is the degree of homogeneity of $A_i$, $i=1,\dots,k$. Since all the slice representations of $V\setminus\NN(V)$ are good, we find that $A\circ\eta$ is a sum $\sum f_i(v)A_i(v)$ where $v\in V\setminus\NN(V)$. Restricting back to $\Upsilon$ we see that our original $A$ can be expressed in the form $\sum_i h_i A_i$. Let $D_1$ denote the restriction of $\Der_{Q_V}(V)^G$ to $\Upsilon$. We have shown that $\ci(\Upsilon)^G\cdot D_1$ is complete. It follows that $\ci(\R^+)\comptensor\ci(\Upsilon)^G\cdot D_1$ is complete.

Consider an element $A$  in the 
closure of $\ci(V)^G\cdot\Der_{Q_V}(V)^G$ in $\Der^\infty_{Q_V}(V)^G$. Then the Taylor series of $A$ at $0$ is in the formal power series module generated by the $A_i$ over $\R[[V_\R]]^{G}$. Thus replacing $A$ by $A-\sum f_i A_i$ for  appropriate $f_1,\dots,f_k\in\ci(V)^G$ we can reduce to the case that $A$ is flat at zero.  Now $A_i\circ\sigma(t,v)=t^{n_i}A_i(v)$  for $t\in\R^+$ and $v\in \Upsilon$. Hence $A\circ\sigma$ lies in the closure of sums of the form $\sum a_i(t,v)A_i(v)$, $v\in \Upsilon$, where $a_i(t,v)\in\ci(\R^+)\comptensor\ci(\Upsilon)^G$. By what we just showed, this space is complete, hence $A\circ\sigma$   has the form $\sum_i a_i(t,v) A_i(v)$ where the $a_i(t,v)$ are in $\ci(\R^+)\comptensor\ci(\Upsilon)^G$. Now $A\circ \sigma $ is flat on $\{0\}\times \Upsilon$ (since $A$ was flat at $0$). Hence  
$$
\sum_i\frac{\pt^j a_i}{\pt t^j}(0,v)A_i(v)=0,\ v\in \Upsilon,\ j=0,1,\dots.
$$
Let $h^j_i\in\ci(P)$ such that  $(\tilde p|_{\Upsilon})^*(h^j_i)(v)=({\pt^j a_i}/{\pt t^j})(0,v)$, $v\in\Upsilon$. Let $\rho$ and the $\lambda_j$ be as in  Lemma \ref{lem:Borel}  such that
$$
\sum_{j=0}^\infty h_i^j(t,y)\text{ converges to }h_i(t,y)\in\ci(\R^+\times P),\  i=1,\dots,k,
$$ 
where
$$
h_i^j(t,y)=\rho(\lambda_j t)   \frac{t^j}{j!}h^j_i(y).
$$
For each $j$, the $h_i^j(t,y)$ pull  back to a relation  of the $A_i$, hence the $h_i$ pull back to functions $k_i$ such that  $\sum k_iA_i=0$. By construction, $k_i$ and $a_i$ have the same Taylor series on $0\times\Upsilon$, $1\leq i\leq k$.  
Hence we can   reduce to the case that
$$
A\circ \sigma=\sum a_i(t,v) A_i(v), \text { where } a_i(t,v)\in\ci(\R^+,0)\comptensor\ci(\Upsilon)^G,\ i=1,\dots, k.
$$
By Corollary \ref{cor:flat}  the $t^{-n_i}a_i(t,v)$ are pullbacks of functions $f_i\in\ci(V,\NN(V))^G$, and we   have that $A=\sum_i f_iA_i$. Hence $\ci(V)^G\cdot\Der_{Q_V}(V)^G$ is complete.
The argument for  $\ci(V)^G\cdot\O(V)_R$ is similar. Hence $V$ is good.
\end{proof}

\begin{proof}[Proof of Theorem \ref{thm:closed}] This follows immediately from Proposition \ref{prop:slice}, Lemma \ref{lem:principal} and Proposition \ref{prop:Vgood}.
\end{proof}
 
\section{Reduction to type $\F$}\label{sec:reductiontoF}
We have Stein $G$-manifolds $X$ and $Y$ which are locally $G$-biholomorphic over $Q$. We show that one can deform a strong  $G$-homeomorphism $\Phi\colon X\to Y$   to a $G$-diffeomorphism of type $\F$  by making small deformations locally. The same process works if $\Phi$ is a strict $G$-diffeomorphism. Note that composition with $G$-biholomorphisms inducing $\Id_Q$ preserves the property of being of type $\F$. This will allow  us to reduce locally to the case that $X=Y$.  First another result \cite[Corollary 6.27]{Lee}.

\begin{lemma}\label{lem:extension}
Let $N$ and $P$ be smooth manifolds. Let $A$ be a closed subset of $N$ and $f\colon N\to P$ a continuous map which is smooth on $A$. Then there is a homotopy $f_t$ of $f$ with $f_t=f$ on $A$, $f_1=f$ and $f_0\colon N\to P$ smooth.
\end{lemma}

Above, $f$ smooth on $A$ means that $f|_A$ has local extensions to smooth maps from open subsets of $N$ to $P$.

We work locally on $X$ and $Y$, so that we may assume that $X=Y=S\times T_B$ are standard neighbourhoods.   We consider what transpires in a neighbourhood of $(s_0,x_0)$ where $s_0\in S$ and $x_0=[e,0]\in T_B$.   Let $L_\vb$ denote the group of $G$-vector bundle automorphisms of $T_W$ inducing the identity on the quotient. Then $L_\vb\subset\Aut_{Q_B}(T_B)^G$.
We have the action of $t\in[0,1]$ on $x=(s,[g,w])$ where $t\cdot x=(s,[g,tw])$, $s\in S$, $[g,w]\in T_B$. This induces an action $z\mapsto t\cdot z$ on the quotient.  Let $\Phi\colon X\to X$   be a strong  $G$-homeomorphism (or a $G$-diffeomorphism of type $\F$). Let $\Phi_t(x)=t\inv\cdot \Phi(t\cdot x)$, $x\in X$
, $t\in(0,1]$.   
By Corollary \ref{cor:fundam-with-S}, $\Phi_0=\lim\limits_{t\to 0}\Phi_t$ exists. From Lemma \ref{lem:fundamental}, Corollary \ref{cor:homotopy-strong} and   Lemma \ref{lem:Phi_t-is-continuous} we obtain the following.
 
 \begin{lemma} \label{lem:parameters}\label{lem:parameters.typeF}
Let $X=S\times T_B$   as above.
\begin{enumerate}
\item Let $\Phi\colon X\to X$ be a strong $G$-homeomorphism. Then the family $\Phi_t$ is a homotopy of strong  $G$-homeomorphisms. Moreover, there is a continuous map $\sigma\colon S\to L_\vb$ such that  $\Phi_0(s,[g,w])=\sigma(s)([g,w])$, $s\in S$, $[g,w]\in T_B$.
\item Let $\Phi\colon X\to X$ be a $G$-diffeomorphism of type $\F$. Then the family $\Phi_t$   is a smooth homotopy of $G$-diffeomorphisms of type $\F$. Moreover, there is a smooth map $\sigma\colon S\to L_\vb$ such that  $\Phi_0(s,[g,w])=\sigma(s)([g,w])$, $s\in S$, $[g,w]\in T_B$.
\end{enumerate}
\end{lemma}

\begin{remark} 
Let $\sigma\colon S\to L_\vb$ be continuous (resp.\ smooth). Define $\Phi(s,[g,w])=\sigma(s)([g,w])$.
We leave it to the reader to show that $\Phi$ is a strong $G$-homeomorphism (resp.\ $G$-diffeomorphism of type $\F$). Given a homotopy $\sigma_t\colon [0,1]\times S\to L_\vb$ of continuous maps, the corresponding family $\Phi_t$ 
is a homotopy of strong $G$-homeomorphisms.   Using  Lemma \ref{lem:extension} we will be able to pass from the case of continuous $\sigma$ to the case of $\sigma$ smooth, i.e., to the case of $G$-diffeomorphisms of type $\F$.
\end{remark}

We call a strong $G$-homeomorphism \emph{special\/} if it corresponds to a continuous map $\sigma\colon S\to L_\vb$.

 The next few results show that we can  locally construct homotopies of a strong $G$-homeo\-mor\-phism  
 $\Phi$ so that it becomes a $G$-diffeomorphism of type $\F$  over neighbourhoods of larger and larger compact  subsets of a stratum of the quotient.
 
Let $p_B\colon T_B\to Q_B$ denote the quotient mapping and let $p$ denote the quotient mapping of $S\times T_B$. Let $q_0=p_B([e,0])$.
We shall cut off  
the homotopy of $\Phi$ so that it is constant outside a compact set. Let $K\subset S$ be compact.   Let $\rho\colon S\times Q_B\to[0,1]$ be a smooth function which is $1$ on  a neighbourhood of $K\times\{q_0\}$ and has compact support $M$.  Let $\tau(t,z)=1+(t-1)\rho(z)$ for $t\in[0,1]$ and $z\in S\times Q_B$.  Then $\tau(t,z)=1$ outside of $M$, and $\tau(t,z)=t$ for $z$ in a neighbourhood $U$ of $K\times\{q_0\}$.  We have the  map $x\mapsto\tau(t,z)\cdot x$ where $z=p(x)$.

\begin{corollary}\label{cor:cutoffdeformation}
Let $\rho$, etc.\  be as above and let $\Phi\colon S\times T_B \to S\times T_B$ be a strong $G$-homeomorphism.  Let $\Phi_t^\rho(x)=\tau(t,z)\inv\cdot \Phi(\tau(t,z)\cdot x)$ for $x\in S\times T_B$ and $t\in(0,1]$.  Set $\Phi_0^\rho=\lim\limits_{t\to 0}\Phi_t^\rho$.  The family $\Phi_t^\rho$, $t\in[0,1]$, is a homotopy of strong $G$-homeomorphisms joining $\Phi$ to $\Phi_0^\rho$.  Moreover, for each $t\in[0,1]$, $\Phi^\rho_t$ equals $\Phi$ over the complement of $M$, and $\Phi_0^\rho=\Phi_0$  is special  over a neighbourhood of $K\times\{q_0\}$.
\end{corollary}

\begin{proof}
Let $\{f_i\}$ be a standard generating set for $\O_\gf(T_B)$ which we may also consider as a standard generating set for $\O_\gf(S\times T_B)$.  
Since $\Phi_t$ is a homotopy of strong $G$-homeomorphisms, 
$$
\Phi_t^*f_i(x)=\sum b_{ij}^t(x)f_j(x),\ x\in S\times T_B,\ t\in[0,1],
$$ 
where the $b^t_{ij}(x)$ are continuous and $G$-invariant in $x$. It follows that  
$$
(\Phi_t^\rho)^*f_i(x)=\sum b_{ij}^{\tau(t,z)}(x)f_j(x),\ x\in S\times T_B,\ t\in[0,1]
$$
where the $b_{ij}^{\tau(t,z)}(x)$ are continuous. Hence $\Phi^\rho_t$ is a homotopy of strong $G$-homeo\-mor\-phisms.
\end{proof}

\begin{lemma}   \label{lem:type.F.remains}
In the situation of Corollary \ref{cor:cutoffdeformation},  suppose that $\Phi$ is a 
$G$-diffeomor\-phism of type $\F$  on $p\inv(\Omega)$ where $\Omega$ is a neighbourhood in $S\times Q_B$ of 
a closed subset $E\times\{q_0\}\subset S\times\{q_0\}$. Then  there is a smaller neighbourhood  $\Omega'$ of $E\times\{q_0\}$  such that  $(t,z)\mapsto \tau(t,z)\cdot z$  sends $[0,1]\times \Omega'$ into $\Omega$. All the strong $G$-homeomorphisms $\Phi^\rho_t$ are $G$-diffeomorphisms of type $\F$ when restricted to $p\inv(\Omega')$. On the complement of $M$ we can arrange that $\Omega=\Omega'$.
\end{lemma}

\begin{proof} Let $\alpha_t$, $t\in[0,1]$, be the endomorphism of $S\times Q_B$ which is induced by the multiplicative action of  $\tau(t,z)=1+(t-1)\rho(z)$.    
Then  $\alpha_t$ is  the identity outside of  $M$. The $\alpha_t$  are also the identity on  $S\times\{q_0\}$. Now $\Omega$ is a neighbourhood of the compact set $E_0=M\cap (E\times\{q_0\})$, and the $\alpha_t$ are the identity on $E_0$.  Hence there is a neighbourhood $\Omega'$ of $E_0$ inside $\Omega$ such that all the $\alpha_t$ send $\Omega'$ into $\Omega$. Of course, we can arrange that $\Omega=\Omega'$ outside $M$. 
Now consider the restriction of $\Phi$ to $p\inv(\Omega')$. 
It follows from Remark \ref{rem:fundam} and the argument of Corollary \ref{cor:cutoffdeformation}   that $\Phi^\rho_t(x)$ is smooth in 
$(t,x)$. Thus it is easy to see that each $\Phi^\rho_t$ is a $G$-diffeomor\-phism inducing the identity on $S\times Q_B$ (with inverse constructed from $\Phi\inv$).
From the definition it is easy to see that  $\Phi^\rho_t$ is of type $\F$ for $t\neq 0$. For $t=0$, fix a point $x\in p\inv(\Omega')$. We may assume that $\Phi$ has a local holomorphic extension $\Psi(x',y)$ for $x'$ and $y$ in a neighbourhood of $\tau(0,z)\cdot x$. Then $\Psi_t(x',y)=t\inv\cdot\Psi(t\cdot x',t\cdot y)$ corresponds to a matrix $(b_{ij}^t(x',y))$ which is smooth in $(t, x',y)$ and holomorphic in $y$. It follows that  $\Psi_{\tau(0,z')}(x',y)$ is a local holomorphic extension of $\Phi_{\tau(0,z')}(x')$ since $(b_{ij}^{\tau(0,z')}(x',y))$   is holomorphic in $y$. Thus $\Phi^\rho_t$ is a $G$-diffeomorphism of type $\F$ over $p\inv(\Omega')$ for all $t\in[0,1]$.
\end{proof}

 \begin{lemma}\label{lem:Fextension}
Let $K\subset S$ be compact. Suppose that $\Phi$ is a $G$-diffeomorphism of type $\F$ over a neighbourhood $\Omega$ of the closed subset $E\times\{q_0\}$ of $S\times Q_B$.  Also suppose that $\Phi$ is special over  a neighbourhood $U=U'\times U''$ of $K\times\{q_0\}$  where $\overline{U}$ is compact. Then, perhaps shrinking $U'$ and $U''$, there is a homotopy $\Phi_t$ of $\Phi$ with the following properties.
\begin{enumerate}
\item $\Phi_t(x)=\Phi(x)$ for all $t$ if  $x$ is off the inverse image of a compact subset $M$ of $U$.
\item  Over a neighbourhood   of the set $(K\cup(E\cap \overline{U'}))\times U''$, $\Phi_0$ is a $G$-diffeo\-mor\-phism of type $\F$.
\item $\Phi_t=\Phi$ over a neighbourhood of $E\times\{q_0\}$ for all $t$. Hence $\Phi_0$ is a $G$-diffeomor\-phism of type $\F$ over a neighbourhood $\Omega'$ of $(E\cup K)\times\{q_0\}$. 
\item $\Omega'=\Omega$ on the complement of $M$.
\end{enumerate}
 \end{lemma}
\begin{proof}
We may assume that $\Phi$ is a $G$-diffeomorphism of type $\F$ over a closed neighbourhood  of  the set $(E\cap\overline{U'})\times\{q_0\}$. Since $E\cap\overline{U'}$ is compact, we may assume that the closed neighbourhood is of the form $E'\times\overline{U''}$.  Since $\Phi$ is special over $U$, there is a corresponding continuous $\sigma\colon U'\to L_\vb$ which is   smooth on $E'\cap U'$. By Lemma \ref{lem:extension} there is a homotopy $\sigma(t,s)$  of $\sigma$, $t\in[0,1]$, $s\in U'$, such that $\sigma(1,s)=\sigma(s)$, $\sigma(t,s)=\sigma(s)$ on $E'\cap U'$ and $\sigma(0,s)$ is smooth.
Now choose a smooth function $\alpha(t,s,q)$, $t\in[0,1]$, such that $\alpha(t,s,q)=t$ for $(s,q)$ in a neighbourhood of $K\times\{q_0\}$ and such that $\alpha(t,s,q)=1$ for $(s,q)$ off of a compact subset $M$ of $U'\times U''$. Let $(s,x)\in S\times T_B$. Then $\Phi_t(s,x)=(s,\sigma(\alpha(t,s,p_B(x)),s)(x))$ is a homotopy of strong $G$-homeomorphisms over $U'\times U''$, where $\Phi_1=\Phi$, $\Phi_t=\Phi$ off of the inverse image of  $M$,  and $\Phi_0$ is a $G$-diffeomorphism of type $\F$ over the interior of $E'\times U''$ and over a neighbourhood of $K\times\{q_0\}$, because of smoothness of the corresponding $\sigma$. We set $\Phi_t=\Phi$ off of the inverse image of $M$. By construction, $\Phi_t=\Phi$ over $(E'\cap U')\times U''$ and off the inverse image of $M$,  hence $\Phi_t=\Phi$   over a neighbourhood of $E\times\{q_0\}$ which we can take to be the same as  $\Omega$ when intersected with the complement of $M$.
\end{proof}

\begin{theorem}\label{thm:deformstrong}
Let $X$ and $Y$ be Stein $G$-manifolds locally $G$-biholomorphic over a common quotient $Q$. Let $\Phi\colon X\to Y$ be a strong $G$-homeomorphism. Then there is a homotopy of $\Phi$, through strong $G$-homeomorphisms, to a $G$-diffeomorphism of type $\F$.
\end{theorem}

\begin{proof}
Consider the stratification of $Q$ by the connected components of the Luna strata. There are at most countably many strata. Let $Z_k$ denote the union of the strata of dimension $k$.  We will inductively find homotopies of $\Phi$ such that it becomes a $G$-diffeomorphism of type $\F$ over a  neighbourhood $\Omega$ of $Z_0 \cup \cdots \cup Z_{k}$.    Each step of the finite induction will be done by a  countable induction.   
 
 Let $k\geq 0$. Then $R_{0}=Z_0\cup\cdots\cup Z_{k-1}$ is closed (and perhaps empty). Suppose by induction we have shown that, modulo a homotopy of strong $G$-homeomorphisms, $\Phi$ is a $G$-diffeomorphism of type $\F$ over a neighbourhood $\Omega_0$ of $R_0$. 
 Let $K_1$, $K_2,\dots$ be a locally finite collection of compact connected subsets of $Z_k$ whose union is $(R_0\cup Z_k)\setminus\Omega_0$. Let $U_j$ be a neighbourhood of $K_j$ in $Q$. We may assume that $X_{U_j}\simeq Y_{U_j}\simeq S_j\times T_{B_j}$ is a standard neighbourhood. Let $p_j\colon T_{B_j}\to Q_{B_j}$ denote the quotient mapping and let $q_j=p_j([e,0])$. We may assume that   $S_j\times\{q_j\}$ is the stratum of $S_j\times Q_{B_j}$ containing $K_j$. We may assume that we have a $G$-biholomorphism of $X_{U_j}$ and $Y_{U_j}$ inducing the identity on 
 $U_j$. So we consider the restriction of $\Phi$ to $X_{U_j}$ to be a strong $G$-homeomorphism of $X_{U_j}$. We may assume that the $\overline{U_j}$ are locally finite on $R_\infty=R_0\cup Z_k$ and that no $\overline U_j$ intersects $R_{0}$. By induction assume that $\Phi$ is a $G$-diffeomorphism of type $\F$ over a neighbourhood $\Omega_{n-1}$ (in $Q$) of $R_{n-1}=R_0\cup K_1\cup\cdots\cup K_{n-1}$. We have the first step of the induction with $\Omega_0$ our neighbourhood of $R_0$. 
Using Corollary \ref{cor:cutoffdeformation} and Lemmas \ref{lem:type.F.remains} and \ref{lem:Fextension} we can find a homotopy $\Phi_t$ of $\Phi$ which equals $\Phi$ off of the inverse image of a compact subset $M_n\subset U_n$ such that $\Phi_0$ is a $G$-diffeomorphism of type $\F$ over a neighbourhood $\Omega_n$ of $R_n$, where $\Omega_n=\Omega_{n-1}$ on the complement of $M_n$. We can consider our homotopy as taking place in the space of strong $G$-homeomorphisms of $X$ and $Y$.   Clearly, by local finiteness of the $\overline{U}_m$, in the limit we   construct  a homotopy $\Phi_t$ of $\Phi$ where $\Phi_0$ is a $G$-diffeomorphism of type $\F$ over a neighbourhood of $R_\infty$. This completes the induction.
 \end{proof}
 
Our procedure of deforming strong $G$-homeomorphisms to $G$-diffeomorphisms of type $\F$, applied to a strict $G$-diffeomorphism, gives a homotopy of  strict $G$-diffeomorphisms to a $G$-diffeomorphism of type $\F$.  The key technical point is Lemma \ref{lem:delta-strict}. Hence we have:
\begin{theorem} \label{thm:deformstrict}
Let $X$ and $Y$ be Stein $G$-manifolds locally $G$-biholomorphic over a common quotient $Q$. Let $\Phi\colon X\to Y$ be a strict $G$-diffeomorphism. Then there is a homotopy of $\Phi$, through strict $G$-diffeomorphisms, to a $G$-diffeomorphism   of type $\F$.
\end{theorem}

  \section{NHC-sections}\label{sec:NHC}
 We work towards proving the following theorem which, in light of Theorems \ref{thm:deformstrong} and \ref{thm:deformstrict}, completes our proof  of Theorem  \ref{thm:main4}.
\begin{theorem}\label{thm:main5} Let $X$ and $Y$ be Stein $G$-manifolds locally $G$-biholomorphic over a common quotient $Q$. 
Suppose that $\Phi\colon X\to Y$ is a $G$-diffeomorphism of type $\F$. Then there is a homotopy $\Phi_t$ of $G$-diffeomorphisms of type $\F$  where $\Phi_0=\Phi$ and $\Phi_1\colon X\to Y$ is a $G$-biholomorphism.
\end{theorem}

We  now consider parameterised families of automorphisms of $G$-saturated open subsets of $X$. We use the notation of \cite{Cartan58}. Let $C$ be a compact Hausdorff space with closed subsets $N\subset H\subset C$. We define a corresponding sheaf $\FF$ on   $Q$ as follows. Let $U\subset Q$ be open and consider the group $\FF(U)$ of $G$-diffeomorphisms $\Phi(t,x)$ of $X_U$, $t\in C$, such that:
\begin{enumerate}
\item $\Phi(t,x)$ is a continuous family of $G$-diffeomorphisms  of $X_U$ of type $\F$.  
\item For $t\in N$, $\Phi(t,x)$ is the identity, i.e., $\Phi(t,x)=x$ for all $x\in X_U$.
\item For $t\in H$, $\Phi(t,x)$ is holomorphic in $x$.
\end{enumerate}
Note that condition (1) is the same as saying that the   partial derivatives of $\Phi$ in $x$ are continuous in 
$(t,x)$. The topology on $\FF(U)$ is uniform convergence of partial derivatives on compact sets.
Similarly we  define the sheaf $\LFF$ of continuous families   of  $G$-invariant vector fields of type $\LF$ on open subsets $U$ of  $Q$. The vector fields are zero for  $t\in N$,  holomorphic for $t\in H$ and of type $\LF$ for all $t\in C$. The topology on $\LFF(U)$ is again uniform convergence of derivatives on compact sets and it is a Fr\'echet space. We can also view $\LFF(U)$ as a closed subspace of  $\CC(C)\comptensor\LF(U)$. Since $\LF(U)$ is a vector subspace of a nuclear space, hence nuclear,   the topology on the tensor product and the completion are unique (take the $\pi$ or $\varepsilon$ topology).  
Since $\LF(U)$ is Fr\'echet (Theorem \ref{thm:closed}), so is $\LFF(U)$.

We say that a continuous function $f(t,x)$ on $C\times X_U$ is an \emph{NHC-function\/} if it is $G$-invariant, zero for $t\in N$, holomorphic for $t\in H$ and smooth for all $t\in C$. The NHC-functions form a Fr\'echet space with the topology of uniform convergence of partial derivatives on compact sets. It is a closed subspace of $\CC(C)\comptensor\ci(X)^G$. We may consider  an NHC-function $f(t,x)$ as a function $\tilde f(t,q)$ for $q\in Q$. But then  $\tilde f$ may not be smooth in $q$ (see Example \ref{ex:not-smooth}).  

We now quote a lemma about surjections of Fr\'echet spaces from \cite[Appendix]{Cartan58}.
\begin{lemma}\label{lem:Cartan}
Let $\pi\colon E\to E'$ be a continuous linear surjection of Fr\'echet spaces. Let $B$ be a closed subset of the compact Hausdorff space $A$. Suppose that we have  continuous mappings $f'\colon A\to E'$ and $h\colon B\to E$ such that $f'$ agrees with $\pi\circ h$ on $B$. Then there is a continuous map $f\colon A\to E$ which extends $h$ such that $\pi\circ f=f'$.
\end{lemma}

\begin{lemma}\label{lem:NHC}
Let $A(t,x)$ be in $\LFF(U)$ where $U$ is a Stein open subset of $Q$. Suppose that $A_1,\dots,A_k$ generate $\Der_U(X_U)^G$ over $\O(U)$. Then there are   NHC-functions $a_i(t,x)$, $x\in X_U$,   such that 
$$
A(t,x)=\sum_i a_i(t,x)A_i(x)\text{ for $x\in X_U$.}
$$
\end{lemma}

\begin{proof}
For notational convenience we may suppose that $U=Q$. Let $E=\O(Q)^k$ and $E'=\Der_Q(X)^G$. Then we have the surjection $\pi$ which sends $(h_1,\dots,h_k)\in E$ to $\sum h_i A_i\in E'$. Now $A(t,x)$ is the zero mapping from $N$ to $E'$, and it lifts to  the zero mapping $h$ of $N$ to $E$. By   Lemma \ref{lem:Cartan}, $h$ extends to a continuous mapping $(h_i(t,x))$ of $H$ to $E$ which covers $A(t,x)$. Now consider the surjection of $E=(\ci(X)^G)^k$ onto $E'$, the space of smooth vector fields of type $\LF$,  sending $(a_1,\dots,a_k)$ to $\sum a_i A_i$. The $h_i(t,x)$, considered as smooth functions, cover $A(t,x)\colon H\to\Der_Q(X)^G\subset E'$. By Lemma \ref{lem:Cartan} we can find extensions of the $h_i(t,x)$ to NHC-functions $a_i(t,x)$ such that $A=\sum_i a_i(t,x)A_i$. 
\end{proof}
 
 Using the open mapping theorem one obtains:

\begin{corollary}\label{cor:nearzero}
Let $\Omega'$ be a neighbourhood of zero in the space of NHC-functions over $U$. Then there is a neighbourhood $\Omega$ of zero in $\LFF(U)$ such that any $A(t,x)\in\Omega$   is $\sum_i a_i(t,x)A_i(x)$ where  $a_i(t,x)\in\Omega'$, $i=1,\dots,k$.
\end{corollary}

Here is a basic result about sections of $\FF$ and $\LFF$.

\begin{theorem}\label{thm:logarithmPhi}
Let $K\subset Q$ be compact and  $U$  a neighbourhood of $K$. Let $U'$ be a neighbourhood of $K$ in $U$ with compact closure in $U$. Then there is a neighbourhood $\Omega$ of the identity family in $\FF(U)$ and a continuous mapping $\log\colon \Omega\to\LFF(U')$ such that $\exp\log\Phi=\Phi|_{C\times U'}$ for $\Phi\in\Omega$.  
\end{theorem}

\begin{proof}  
We may assume that we have a standard generating set $\{f_i\}$ for $\O_\gf(X)$.  By Theorem \ref{thm:Dexists} there is   a neighbourhood $\Omega'$ of the identity in $\F(U)$ such that 
any $\Psi\in\Omega'$  admits a logarithm $\log\Psi\in\LF(U')$. The mapping $\Omega'\ni\Psi\mapsto\log\Psi$ is continuous. Let $\F(C\times U)$ denote continuous families of elements of $\F(U)$. Let
$\Omega$ be the open subset of $\F(C\times U)$ of families $\Phi(t,x)$ such that $\Phi(t,x)\in\Omega'$ for all $t$. Then $\Omega$ is open in $\F(C\times U)$ and the family $t\mapsto\log\Phi(t,x)$ is a continuous family in $\LF(U')$. Now the intersection of $\Omega$ with $\FF(U)$ is open, and if $\Phi(t,x)\in\Omega\cap\FF(U)$,  then
 $\log\Phi(t,x)$ is zero if $t\in N$ and is holomorphic if $t\in H$. Hence $\log\Phi(t,x)\in\LFF(U')$. Clearly $\log\colon\FF(U)\to\LFF(U')$ is continuous.
 \end{proof}

If $U$, $U'$ and $\Omega$ are as above, we say that  every $\Phi\in\Omega$ \emph{admits a logarithm in $\LFF(U')$}.

\begin{corollary}\label{cor:small-coefficients}
Suppose that $A_1,\dots,A_k$ generate  $\Der_{U'}(X_{U'})^G$ 
over $\O(U')$. Let $\Omega'$ be a neighbourhood of zero in the space of NHC-functions over $U'$. Then there is a neighbourhood $\Omega$ of the identity family in $\FF(U)$ such that for any $\Phi\in\Omega$, $\log\Phi=\sum a_i(t,x)A_i(x)$ where the $a_i\in\Omega'$.
\end{corollary}

Using topological tensor products one establishes the following variant  of Lemma \ref{lem:bij}.  

\begin{lemma}\label{lem:bij-with-parameters}
 Let $S\times T_B$ be a standard neighbourhood in $X$ and let  $K$ be a compact subset of $S\times Q_B$. 
 Let $\{f_i\}$ be a standard generating set for $S\times T_B$. Then there is a neighbourhood $\Omega$ of the identity in $\FF(S\times Q_B )$ such that any $\Phi\in\Omega$  
 corresponds to a matrix $(b_{ij}(t,s,q))$, $t\in C$, $s\in S$, $q\in Q_B$, where $||(b_{ij}^u)-I||_{C\times K}<1/2$, $0\leq u\leq 1$.
 \end{lemma}

\begin{proposition}\label{prop:local-log-Phi-families}
Let $U\subset Q$ be open, let $q\in U$ and   $F=X_{q}$. Suppose that  $\Phi\in\FF(U)$  is the identity when restricted to $F$. Then there is a neighbourhood $U'$ of $q$ such that $\Phi$ admits a logarithm in $\LFF(U')$.
\end{proposition}

\begin{proof}  
We may assume that we have a standard generating set for $\O_\gf(X_U)$.
Shrinking $U$ we are in the situation where $X_U=S\times T_B$ is a standard neighbourhood and $q=(s_0,q_0)$ where $s_0\in S$ and $q_0$ is the image of $[e,0]\in T_B$. We have the action of $[0,1]$ on $X_U$ sending $(s,x)$ to $(s,u\cdot x)$, $s\in S$, $x\in T_B$, $u\in[0,1]$. Let $\Phi_u$ be the corresponding deformation of $\Phi$. Then $\Phi_u$ is a continuous family of elements of $\FF(S\times Q_B)$, $u\in[0,1]$. Since $\Phi$ is the identity on $F$,  $\Phi_0(t,s_0,x)$  is the identity, $t\in C$, $x\in T_B$. By Theorem \ref{thm:logarithmPhi} and Lemma \ref{lem:bij-with-parameters} there is a neighbourhood $\Omega$ of the identity section in $\FF(S\times Q_B)$ and a neighbourhood $U_0$ of $(s_0,q_0)$ such that every $\Psi\in\Omega$ admits a logarithm in $\LFF(U_0)$ and has a matrix $(b_{ij})$ such that $||(b_{ij}^u)-I||_{C\times \overline {U_0}}<1/2$, $u\in[0,1]$. For $u$ in a neighbourhood of $0$,  $\Phi_u$  lies in $\Omega$, hence  admits a logarithm $D_u\in \LFF(U_0)$. 
As in the proof of Lemma \ref{lem:LF-in-neighbourhood} this implies that $\Phi$ admits a logarithm over $U'= u\cdot U_0$.
\end{proof}

\begin{remark}\label{rem:NH-remark}
Let $K\subset Q$ be compact. Let $U$ be a neighbourhood of $K$ and let $U'$ be a relatively compact neighbourhood of $K$ in $U$. Let $\Phi\in\FF(U)$ and suppose that $\Phi$ is represented by the matrix $(a_{ij}(t,q))$ and admits a logarithm $D\in\LFF(U')$ with matrix $(d_{ij}(t,q))$ where the $d_{ij}$ are close to zero. There is no a priori guarantee that the $d_{ij}(t,q)$   are holomorphic for $t\in H$ or zero for $t\in N$. However, by Lemma \ref{lem:NHC} and Corollary \ref{cor:nearzero}, we can arrange this to be true. It follows that we can arrange that,  over $U'$,     $(a_{ij}(t,q))$ is holomorphic for $t\in H$ and the identity matrix for $t\in N$.
\end{remark}

\section{Grauert's proof}\label{sec:Grauert}
  We now show how one can modify Cartan's version \cite{Cartan58} of the proof of Grauert's Oka principle \cite{GrauertApproximation, GrauertLiesche, GrauertFaserungen} to obtain Theorem \ref{thm:main5}.  Let $N\subset H\subset C$ be compact Hausdorff spaces as before, and we consider the corresponding sheaves $\FF$ and $\LFF$ defined in the previous section, together with the topologies on their sections. We 
  recall  that an open subset $U\subset Q$ is a \emph{Runge domain\/} if it is Stein and $\O(Q)$ is dense in $\O(U)$. Here is the main theorem.
  
\begin{theorem}\label{thm:main}
Suppose that $N$ is a deformation retract of $C$. Then the following hold.
\begin{enumerate}
\item The topological group $H^0(Q,\FF)$ is pathwise connected.
\item If $U\subset Q$ is  a Runge  domain, then $H^0(Q,\FF)$ is dense in $H^0(U,\FF)$.
\item $H^1(Q,\FF)=0$.
\end{enumerate}
\end{theorem}

We actually only need (3), but the  proof of the theorem   is by an induction involving all three statements.
\begin{proof}[Proof of Theorem   \ref{thm:main5}] Let $\Phi\colon X\to Y$ be a $G$-diffeomorphism of type $\F$. We have an open cover $\{U_i\}$ of $Q$   and  $G$-biholomorphisms $\Gamma_i\colon X_i=X_{U_i}\to Y_i=Y_{U_i}$ covering $\Id_{U_i}$.  We may assume that $X_i\simeq S_i\times T_{B_i}\subset S_i\times T_{W_i}$ is a standard neighbourhood in $X$ where $S_i$ is smoothly contractible, 
say a ball. Let $\Psi=\Gamma_i\inv\circ\Phi$.
By Lemma \ref{lem:parameters.typeF} we have a homotopy $\Psi_t$ of $G$-diffeomorphisms of type $\F$ where $\Psi_0$ is special and corresponds to  a smooth mapping $\sigma\colon S_i\to L_\vb$. Here   $L_\vb$ is the subgroup of  $\Aut_\vb(T_{W_i})^G$ fixing the invariants. Since $S_i$ is smoothly contractible, we have a smooth homotopy of $\sigma$  to a constant mapping to $L_\vb$, hence we have a homotopy of $\Psi_0$ to a $G$-biholomorphism inducing the identity on $S_i\times Q_{B_i}$. Combining the two homotopies we obtain  a homotopy $\Psi_i(t,x)$ of $\Gamma_i\inv\circ\Phi$ which equals $\Gamma_i\inv\circ\Phi$ when $t=1$ and is holomorphic when $t=0$. Moreover, all the $\Psi_i(t,x)$ are of type $\F$. Let $\Phi_i(t,x)=\Gamma_i\circ\Psi_i(t,x)$. Then $\Phi_i(1,x)=\Phi$ on $X_i$ and $\Phi_i(0,x)\colon X_i\to Y_i$ is a $G$-biholomorphism. Now $\Phi_i(t,x)\inv\circ\Phi_j(t,x)$ gives an element of $H^1(Q,\FF)$ where $C=[0,1]$, $N=\{1\}$ and $H=\{0,1\}$. By (3) of Theorem \ref{thm:main} there are sections   $c_i(t,x)\in H^0(U_i,\FF)$ such that $\Phi_i(t,x)\inv\circ \Phi_j(t,x)=c_i(t,x)\inv \circ c_j(t,x)$. Then the $\Phi_i(t,x)\circ c_i(t,x)\inv$ combine to give a homotopy $\Phi(t,x)$ with $\Phi(1,x)=\Phi(x)$ and $\Phi(0,x)\colon X\to Y$ a $G$-biholomorphism inducing $\Id_Q$.
\end{proof}

We recall some definitions from \cite{Cartan58}. 
Let $a_j\leq b_j$ and $c_j\leq d_j$ be real numbers, $j=1,\dots,m$. Let $z_j=x_j+iy_j$ be the usual coordinate functions on $\C^m$. The corresponding  \emph{cube\/} in $\C^m$ is the  subset defined by the  inequalities $a_j\leq x_j\leq b_j$ and $c_j\leq y_j\leq d_j$, $j=1,\dots,m$. 
 Let $K\subset Q$ be compact and let $U$ be a neighbourhood of $K$. Suppose that we have a holomorphic mapping $f\colon Q\to\C^m$ which restricts to a biholomorphism of $U$ onto an analytic subset of a neighbourhood of a cube $\Gamma$ where $\Gamma$ has real dimension $k$. We say that $K$ is $k$-\emph{special\/} or \emph{special of dimension $k$\/} if $K=U\cap f\inv(\Gamma)$. Special compact sets $K$ have nice properties. For example, 
 $K$ is holomorphically convex, so every holomorphic function on a neighbourhood of $K$ can be uniformly approximated (on a neighbourhood of $K$) by functions holomorphic on $Q$. Also, $Q$ is the union of special compact sets $K_n$ where $K_n$ lies in the interior of $K_{n+1}$ for all $n$.
 
 Let $K\subset Q$ be a special compact set, so that we can think of it as a subset of a cube $\Gamma$ corresponding to real numbers $a_j\leq b_j$ and $c_j\leq d_j$, $j=1,\dots,m$. Let $c\in[a_1,b_1]$. Let $\Gamma'$ be the points of $\Gamma$ where $x_1\leq c$ and let $\Gamma''$ be the points where $x_1\geq c$. Let $K'$ denote $K\cap \Gamma'$ and let $K''$ denote $K\cap \Gamma''$. Then $K'$, $K''$ and $K'\cap K''$ are special. We say that the triple $(K,K',K'')$ is a \emph{special configuration\/}.
 
Let $K$ be a  compact set in $Q$. Define $H^0(K,\FF)$ to be the direct limit (with the direct limit topology) of the groups $H^0(U,\FF)$ for $U$ an open set containing $K$. Our proof of Theorem \ref{thm:main} uses the following two key results.
 
\begin{proposition}\label{prop:1}
Let $K$ be a  special compact subset of $Q$. Then the image of $H^0(Q,\FF)$ in $H^0(K,\FF)$ is dense in a neighbourhood of the identity element of $H^0(K,\FF)$.
\end{proposition}

\begin{proposition}\label{prop:2}
Let $(K,K',K'')$ be a special configuration in $Q$. Then any element $\Phi\in H^0(K'\cap K'',\FF)$, sufficiently close to the identity, can be written in the form
$$
\Phi=\Phi'\circ (\Phi'')\inv
$$
where $\Phi'\in H^0(K',\FF)$ and $\Phi''\in H^0(K'',\FF)$.
\end{proposition}

 We assume   the propositions for now. By induction on 
 $k\geq 0$ we prove the following statements.
 \begin{enumerate}
\item [(i)$_k$] If $K$ is a $k$-special compact set,   then $H^0(K,\FF)$ is pathwise connected.
\item[(ii)$_k$] If $K$ is a $k$-special compact set,   then $H^0(Q,\FF)\to H^0(K,\FF)$ has dense image.
\item[(iii)$_k$] If $(K,K',K'')$ is a special configuration where $K'\cap K''$ is  $k$-special, then every $\Phi\in H^0(K'\cap K'',\FF)$ can be written in the form $\Phi'\circ(\Phi'')\inv$ where $\Phi'\in H^0(K',\FF)$ and $\Phi''\in H^0(K'',\FF)$.
\end{enumerate}

\begin{proof} 
Consider (i)$_0$. We have that $K=\{q_0\}\subset Q$ and we need to show that $H^0(\{q_0\},\FF)$ is pathwise connected. Let $\Phi(t,x)$ be a section of $\FF$ defined in a neighbourhood of $p\inv(q_0)=F$ and let  $\lambda(t)$ denote the restriction of $\Phi(t,x)$ to $F$. Then for every 
$t_0\in C$, $\lambda(t_0)\in L_\hr$, the algebraic subgroup of $\Aut(F)^G$ of holomorphically reachable points (see  Theorem \ref{thm:reachable} and discussion preceding it). Since $N$ is a deformation retract of $C$, 
$\lambda$ has values  in the identity component of $L_\hr$ and we have a homotopy of $\lambda(t)$ to the identity section on $F$. By Remark \ref{rem:extendL0} we have a homomorphism $\nu\colon L_\hr\to\Aut_U(X_U)^G$ for a neighbourhood $U$ of $q_0$ such that $\nu(\ell)$  restricted to $F$  is $\ell$ for $\ell\in L_\hr$. This allows us to reduce to the case that $\Phi(t,x)$ restricts to the identity section on $F$.   Then it follows from Proposition \ref{prop:local-log-Phi-families} that the  logarithm $D(t,x)$ of $\Phi(t,x)$   exists  in a neighbourhood of $F$, and $D(t,x)$ is a section of $\LFF$. Then we have the homotopy $\Phi(t,u,x)=\exp(uD(t,x))$, $0\leq u\leq 1$. This is a homotopy of $\Phi(t,x)$ to the identity map on a neighbourhood of $F$.  Hence we have (i)$_0$.

Suppose that we have (i)$_k$. Let $K$ be compact and $k$-special. Then (i)$_k$ implies that every element of $H^0(K,\FF)$ is a product of finitely many elements of $H^0(K,\FF)$ which are arbitrarily close to the identity. Apply Proposition \ref{prop:1} to the elements close to the identity. Then one obtains (ii)$_k$.

 Suppose that we have (ii)$_k$. Let $(K,K',K'')$ be a special configuration such that $K'\cap K''$ is  special of dimension $k$. Let $\Phi\in H^0(K'\cap K'',\FF)$. By (ii)$_k$ one can write $\Phi=\Psi\cdot \Phi_1$ where $\Phi_1\in H^0(K'\cap K'',\FF)$ is close to the identity section and $\Psi\in H^0(K',\FF)$. By Proposition \ref{prop:2}, we can write $\Phi_1=\Phi'\cdot (\Phi'')\inv$ where $\Phi'\in H^0(K',\FF)$ and $\Phi''\in H^0(K'',\FF)$. Then $\Phi=(\Psi\cdot \Phi')\cdot(\Phi'')\inv$ and we have (iii)$_k$.
 
 Suppose that we have  (i)$_k$, 
 hence (iii)$_k$. Let $K$ be compact and   $(k+1)$-special. Let $\Phi\in H^0(U,\FF)$ for $U$ a neighbourhood of $K$. Let the  cube in $\C^m$ corresponding to $K$ be $\Gamma$ where the condition on $x_1$ is $a_1\leq x_1\leq b_1$. Let $\lambda\in[a_1,b_1]$ and let $K_\lambda$ denote the special compact subset of $K$ defined by $x_1=\lambda$. We may assume that $a_1<b_1$ so that each $K_\lambda$ has dimension $k$. By (i)$_k$, $K_\lambda$ has a neighbourhood $V_\lambda$ in $K$ (we may choose closed neighbourhoods) such that the restriction of $\Phi$ to $V_\lambda$ is homotopic to the identity section. There are a finite number of $K_{\lambda_i}$ (with $\lambda_1<\lambda_2<\cdots$) such that the corresponding $V_{\lambda_i}$ cover $K$. We can assume that the $V_{\lambda_i}$ are   $(k+1)$-special such that   $V_{\lambda_i}\cap V_{\lambda_{i+1}}$ is  $k$-special for each $i$. Let $K_i$ denote $V_{\lambda_i}$ and let $\Phi_i(u)\in H^0(K_i,\FF)$ be a section in a neighbourhood of $K_i$, depending continuously on the parameter $u\in[0,1]$, such that $\Phi_i(0)$ is the section induced by $\Phi$ and $\Phi_i(1)$ is the identity section $e$.  
 
 The homotopies $\Phi_i(u)$ do not have to agree on the intersections $K_i\cap K_{i+1}$. We now modify them to agree. One is reduced to what Cartan calls an \emph{elementary problem\/}. Let $(K,K_1,K_2)$ be a special configuration where $K_1\cap K_2$ is special of dimension $k$. Let $\Phi\in H^0(K,\FF)$ and let $\Phi_i(u)\in H^0(K_i,\FF)$ be homotopies such that $\Phi_i(0)$ is the restriction of $\Phi$ and $\Phi_i(1)$ is   $e$. Then one wants to find $\Psi(u)\in H^0(K,\FF)$ such that $\Psi(0)=\Phi$ and $\Psi(1)=e$.
 
 Now $\Phi_1(u)\inv \Phi_2(u)\in H^0(K_1\cap K_2,\FF)$ is a homotopy from the identity section of $H^0(K_1\cap K_2,\FF)$ to itself. It is an element of $H^0(K_1\cap K_2,\FF')$ where $\FF'$ is the sheaf of groups relative to the compact sets
 $$
 C'=C\times [0,1],\ N'=(N\times[0,1])\cup(C\times\{0\})\cup(C\times\{1\}),\ H'=(H\times [0,1])\cup N'
 $$
 where $N'$ is still a deformation retract of $C'$. Thus applying 
 (iii)$_k$ to $\FF'$ we have that 
 $$
 \Phi_1(u)\inv \Phi_2(u)=\Phi'(u)\Phi''(u)\inv
 $$
 where $\Phi'(u)\in H^0(K_1,\FF)$ and $\Phi''(u)\in H^0(K_2,\FF)$ depend continuously upon the parameter $u\in[0,1]$. They are the identity element for $u=0$ and $u=1$. Let $\Psi(u)=\Phi_1(u)\Phi'(u)$ in a neighbourhood of $K_1$ and let it equal $\Phi_2(u)\Phi''(u)$ in a neighbourhood of $K_2$. Then the definitions agree on the overlaps and we have $\Psi(u)\in H^0(K,\FF)$ giving the desired deformation. This solves the elementary problem and shows that we have (i)$_{k+1}$.
 \end{proof}
 
 We have established (i), (ii) and (iii) which are the statements (i)$_k$, etc.\ with $k$ omitted.
\begin{proof}[Proof of Theorem \ref{thm:main}]
Note that (2) follows from (ii). Now we consider (1). We know that there are special compact sets $K_1\subset K_2\subset  \cdots$ where $K_i$ is contained in the interior $V_{i+1}$ of $K_{i+1}$ for all $i$ and where $Q=\cup K_i$. Let $\Phi\in H^0(Q,\FF)$. By (i), the image of $\Phi$ in $H^0(K_n,\FF)$ is homotopic to the identity section. Thus  the image $\Phi_n$ of $\Phi$ in $H^0(V_n,\FF)$ is homotopic to the identity section. Let $\Phi_n(u)$ be such a homotopy with $\Phi_n(0)=\Phi_n$ and $\Phi_n(1)=e$. Then $\Phi_n(u)\inv \Phi_{n+1}(u)$, over $V_n$, is an element of $H^0(V_n,\FF')$ where $\FF'$ is the sheaf of groups associated to the sets $N'\subset H'\subset C'$ given above. Applying (ii) to $\FF'$ we see that $\Phi_n(u)\inv \Phi_{n+1}(u)$ may be   approximated over $K_{n-1}\subset V_n$ by elements of $H^0(V_{n+1},\FF')$. Hence, without changing $\Phi_n(u)$, we can modify $\Phi_{n+1}(u)$ such that $\Phi_n(u)\inv \Phi_{n+1}(u)$ is arbitrarily close to the identity section over $K_{n-1}$. Thus one can arrange that the sequence $\Phi_n(u)$ converges  over every compact subset of $Q$. Let $\Phi(u)$ denote the limit. Then $\Phi(u)$ is a homotopy of elements of $H^0(Q,\FF)$ with $\Phi(0)=\Phi$ and $\Phi(1)$ the identity section. Here we have used Corollary \ref{cor:typeFclosed}. This proves (1).

Finally, we want to prove (3), that $H^1(Q,\FF)=0$. Let $K$ be a special compact subset of $Q$. As in \cite[Section 5]{Cartan58}, using (iii) and an induction over a decomposition of $K$ into a union of smaller special compact sets, one shows that  $H^1(K,\FF)=0$.  Let $K_n$ be a sequence of special compact sets as above, with $V_n$ denoting the interior of $K_n$. Let $\{U_i\}$ be  an open cover of $Q$ and let $\{\Phi_{ij}\}$ be a cocycle   with values in $\FF$. Then for each $n$ we have sections $c_i^n\in H^0(U_i\cap V_n,\FF)$ such that $\Phi_{ij}=(c_i^n)\inv c_j^n$ on $U_i\cap U_j\cap V_n$. Hence we have 
$$
c_i^{n+1}(c_i^n)\inv=c_j^{n+1}(c_j^n)\inv\text{ on }U_i\cap U_j\cap V_n.
$$
Thus the $c_i^{n+1}(c_i^n)\inv$ define a section $\Phi_n\in H^0(V_n,\FF)$. Using (ii) as above, we can arrange that the  $c_i^n$ converge   on compact sets as $n\to\infty$. The limiting  sections $c_i\in H^0(U_i,\FF)$ have the property that 
$$
\Phi_{ij}=c_i\inv c_j\text{ on } U_i\cap U_j.
$$
Hence $H^1(Q,\FF)=0$.
 \end{proof}

We are left with the proofs of the two propositions.
 
 \begin{proof}[Proof of Proposition \ref{prop:1}]
Let $\Phi\in H^0(K,\FF)$ where $K$ is special and $\Phi$ is  near to the identity section. By Theorem \ref{thm:logarithmPhi}, there is a relatively compact neighbourhood $U$ of $K$ such that $D =\log\Phi$ exists and is in $\LFF(U)$.   We may suppose that $U$ is a Runge domain. By   Lemma \ref{lem:NHC}   we can write $D=\sum a_i(t,x) A_i$ where the $a_i$ are NHC   and the $A_i\in\Der_Q(X)^G$ generate $\Der_U(X_U)^G$. Thus it suffices to approximate NHC-functions on $U$ by global ones. Using a partition of unity on $C$ one is reduced to the problem of approximating holomorphic (resp.\ smooth $G$-invariant) functions on $U$ (resp.\ $X_U$) by functions holomorphic  on $Q$ (resp.\ smooth and $G$-invariant on $X$). There is no problem for smooth functions, and for holomorphic functions approximation is possible because $U$ is a Runge domain.
\end{proof}

  Before giving the proof of Proposition \ref{prop:2} we need some preliminary results. First we consider a result on differential equations on Lie groups from \cite[p.\ 115]{Cartan58}. Let $L$ be a Lie group and let $f(u)$, $f'(u)$ and $f''(u)$ be elements of $L$ which are $\mathcal C^1$ in $u\in [0,1]$ such that $f'(u)=f(u)  f''(u) $ and $f(0)=f'(0)=f''(0)=e$, the identity of $L$. Then $ {df}/{du}$ is a tangent vector at $f(u)$, hence it is of the form $d\rho(f(u))\cdot a(u)$ where $a(u)$ is a continuous family of elements of $T_e(L)$ and $\rho$ denotes right multiplication.   We write this as
 $$
 \frac{df}{du}=a(u)\cdot f(u).
 $$ 
 Similarly we have continuous $a'(u)$ and $a''(u)$ with
\begin{equation}\label{eq:lie}
 \frac{df'}{du}=a'(u)\cdot f'(u)\text{ and }\frac {df''}{du}=a''(u)\cdot f''(u)
\end{equation}
 Consider $a(u)$, $a'(u)$ and $a''(u)$ as elements of $\lie l$, the Lie algebra of $L$.

 \begin{lemma}\label{lem:liegroup}
 Let $f(u)$, $a(u)$ etc.\ be as above. Then we have 
\begin{equation}\label{eq:Ad}
 a'(u)=a(u)+\Ad(f(u))\cdot a''(u).
\end{equation}
 Conversely, given $f(u)$ which is $\mathcal C^1$ in $u$ with $f(0)=e$, define $a(u)$ as above and suppose that there are continuous $a'(u)$ and $a''(u)$ with values in $\lie l$ satisfying \eqref{eq:Ad}.
 Then integrating the equations \eqref{eq:lie} with the conditions that $f'(0)=f''(0)=e$ we obtain $f'(u)$ and $f''(u)$ such that $f'(u)=f(u)f''(u)$.
 \end{lemma}

Let $U\subset Q$ be open.
Let $\Phi(t,x)\in\FF(U)$ and let $A(t,x)$, $D(t,x)\in \LFF(U)$. Let $F$ be a fibre of $p\colon X_U\to U$. Then the restriction of $\Phi(t,x)$ to $F$ is a family $\Phi(t)|_F$ in $L_\hr$, the  reachable automorphisms in $\Aut(F)^G$. We have the adjoint action $\Ad \Phi(t)|_F$    on   $A(t)|_F\in\lie l_\hr$.  If
$\Phi=\exp D$, then $\Ad\Phi(t)|_F$ is   the exponential of the action of $D(t)|_F$ by $\ad$.  It is   not hard to see that  the $\Ad\Phi(t)|_F$ combine to give an action $\Ad\Phi$ of $\Phi$ on $\LFF(U)$  and, of course, the $\ad D(t)|_F$ give us an action $\ad D$ of $D$ on $\LFF(U)$ which is just bracket with $D$.

 \begin{proposition}\label{cor:Adclosetozero}
 Let $K\subset Q$ be compact and let $U$ be a relatively compact neighbourhood of $K$.
 If $\Phi\in\FF(U)$ is sufficiently close to the identity, then for any $A\in\LFF(U)$, $\Ad\Phi(A)$ is close to $A$ over a neighbourhood of $K$. In particular, if $A$ is close to zero, so is $\Ad\Phi(A)$.
 \end{proposition}
 \begin{proof}
 Let $D_1,\dots,D_k\in\Der_U(X_U)^G$ generate $\Der_U(X_U)^G$ as an  $\O(U)$-module. Let $U'$ be a neighbourhood of $K$ relatively compact in $U$. By Theorem \ref{thm:logarithmPhi} and Corollary \ref{cor:small-coefficients}, by choosing $\Phi\in\FF(U)$ close to the identity, we have  $\Phi=\exp D$, $D\in\LFF(U')$, where $D=\sum a_iD_i$ and the NHC functions $a_i$ can be chosen close to zero. Then $\Ad\Phi(A)=(\exp\sum a_i \ad D_i)(A)$ is close to $A$ in a neighbourhood of $K$.
 \end{proof}
 
The following is  essentially Hilfssatz 1 in \cite{GrauertLiesche}.  

\begin{lemma}\label{lem:matrixsplitting}
Let $(K,K',K'')$ be a special configuration. For every $t\in H$ suppose that we have a square matrix $M(t,q)$ which is holomorphic and invertible for $q$ in a neighbourhood 
$U$ of $K'\cap K''$, and suppose that $M(t,q)$ is continuous in $(t,q)$. If $M(t,q)$ is sufficiently close to the identity 
on $H\times (K'\cap K'')$, then  
$$
M(t,q)=M'(t,q)\cdot M''(t,q)\inv
$$
where $M'(t,q)$ (resp.\ $M''(t,q)$) is continuous in 
$(t,q)$, and  for fixed $t\in H$ is holomorphic and invertible for  $q$ in a neighbourhood of $K'$ (resp.\ $K''$). Moreover, $M'(t,q)$ and $M''(t,q)$ can be chosen in any given neighbourhood of the identity if $M(t,q)$ is sufficiently close to the identity.
\end{lemma}

\begin{proof}[Proof of Proposition \ref{prop:2}]
Let $\Phi\in H^0(K'\cap K'',\FF)$ be close to the identity. Then we want $\Phi'\in H^0(K',\FF)$ and $\Phi''\in H^0(K'',\FF)$ so that 
 $$
 \Phi(t,x)=\Phi'(t,x)\circ \Phi''(t,x)\inv\text{ for $p(x)$ close to }K'\cap K''.
 $$
 We first do this for $t\in H$, where the sections have to be the identity on $N$. Cartan calls this the \emph{fundamental problem}.  
If $\Phi$ is sufficiently close to the identity, then Theorem \ref{thm:logarithmPhi} shows that  $\Phi(t,x)=\exp(D(t,x))$ for $t\in H$ where $D$ vanishes on $N$ (since $\Phi$ is the identity on $N$) and is holomorphic on a relatively compact neighbourhood $U$ of $K'\cap K''$. Set $\Phi(t,u,x)=\exp(u\cdot D(t,x))$, $0\leq u\leq 1$.  
Then $\pt\Phi/\pt u=D(t,x)\cdot\Phi(t,u,x)$.
 Suppose that we have  parameterised $G$-invariant holomorphic vector fields $D'(t,u,x)$ and $D''(t,u,x)$ over the inverse image of neighbourhoods of $K'$ and $K''$, respectively, zero on $N$, annihilating the invariants, such that
 $$
 D'(t,u,x)=D(t,x)+\Ad \Phi(t,u,x)\cdot D''(t,u,x)\text{ for $p(x)$ close to }K'\cap K''. 
 $$
  By Lemma \ref{lem:liegroup},  integrating in $u$ we obtain holomorphic parameterised sections $\Phi'(t,u,x)$ and $\Phi''(t,u,x)$ such that 
 $$
 \Phi(t,u,x)=\Phi'(t,u,x)\circ \Phi''(t,u,x)\inv
 $$
For $u=1$ we get a solution of the fundamental problem. We will construct $D'$ and $D''$ close to zero so that $\Phi'$ and $\Phi''$ are close to the identity.
 
 Let $\{A_i\}$ be $\O(U)$-module generators of $\Der_{U}(X_U)^G$.  Then
 $$
\Ad  \Phi(t,u,x)\cdot A_i=\sum m_{ij}A_j
 $$
 where the coefficients $m_{ij}(t,u,q)$ are holomorphic for $q\in U$, and continuous for $t\in H$ and $u\in [0,1]$. This uses  Lemma \ref{lem:NHC}. If $\Phi(t,x)$ is sufficiently close to the identity section, then   $\Phi(t,u,x)\cdot A_i$ is close to $A_i$ and   we may choose  $(m_{ij})$ close to the identity. By Lemma \ref{lem:matrixsplitting}  we have
 $$
 \sum_i m_{ji}m'_{ik}=m''_{jk}
 $$
 where the matrices $m'_{ij}(t,u,q)$ and $m''_{ij}(t,u,q)$ are invertible and depend continuously on $t\in H$ and $u\in[0,1]$ and are close to the identity. The first is holomorphic near $K'$ and the second near $K''$. We have
 $$
 D(t,x)=\sum a_i(t,q)A_i
 $$
 where the $a_i(t,q)$ are holomorphic near $K'\cap K''$. By Corollary \ref{cor:nearzero} we may assume that the $a_i(t,q)$ are near zero. It suffices to find $a'_i(t,u,q)$ and $a''_i(t,u,q)$ holomorphic near $K'$ and $K''$, respectively, near zero, such that
 $$
 a'_i(t,u,q)=a_i(t,q)+\sum_j m_{ji}a''_j(t,u,q)
 $$
 for $q$ near $K'\cap K''$. For this it suffices that
 $$
 \sum_i m'_{ik}a'_i=\sum_i m'_{ik}a_i+\sum_i m''_{ik}a_i''.
 $$
 Since our matrices are invertible, in place of the $a'_i$ and $a''_i$ one can take  as unknowns the terms
 $$
 b'_k=\sum_i m'_{ik}a'_i\text{ and } b''_k=\sum_i m''_{ik}a''_i.
 $$
  Set
 $$
 b_k=\sum m'_{ik}(t,u,q)a_i(t,q).
 $$
 Then our equation becomes 
 $$
 b'_k(t,u,q)-b_k''(t,u,q)=b_k(t,u,q)
 $$
 which we can always solve with $b'_k(t,u,q)$ and $b''(t,u,q)$ small if $b_k(t,u,q)$ is small \cite[footnote p.\ 117]{Cartan58}.
 
 So we have solved the fundamental problem. Given $\Phi(t,x)$ near the identity we find $\Phi'(t,x)$ and $\Phi''(t,x)$, near the identity, such that $\Phi=\Phi'(\Phi'')\inv$ for $t\in H$.    Since $\Phi'$ and $\Phi''$ are near the identity  we may write them as $\exp D'$ and $\exp D''$, respectively, where $D'$ and $D''$ are near zero.   We have   continuous maps $a_i'(t,q)$ and $a''_i(t,q)$ from $H$ to holomorphic functions which are small on neighbourhoods $U'$ of $K'$ and $U''$ of $K''$, respectively, zero on $N$, such that 
 $$
 D'=\sum_i a_i'(t,q)A_i\text{ and }D''=\sum a_i''(t,q)A_i.
 $$
 Now the  space of holomorphic functions on $U'$ is nuclear. Since $\CC(C)\to \CC(H)$ is surjective, we find that
 $$
 \CC(C,\O(U'))\simeq \CC(C)\comptensor \O(U')\to \CC(H)\comptensor \O(U')\simeq \CC(H,\O(U'))
 $$ 
 is surjective. See \cite[Proposition 43.9]{Treves}.
 Thus we have 
 an extension  of $D'$   to an  invariant holomorphic vector field  annihilating the invariants for all $t\in C$, and similarly for $D''$. The extensions may not have coefficients close to zero for $t\in C$, but we can arrange this by multiplying by continuous functions $g'(t,x)$ and $g''(t,x)$  which are $1$ for $t\in H$ and smooth and invariant in $x$. Thus we have smooth extensions $D'\in\LFF(U')$  and $D''\in\LFF(U'')$ which have  exponentials 
 $\widetilde \Phi'(t,x)$ and $\widetilde \Phi''(t,x)$, respectively, 
 near the identity. Consider the product
 $$
 \Psi(t,x) =\widetilde\Phi'(\widetilde\Phi'')\inv \Phi\inv.
 $$
  Then $\Psi$ is a section of $\FF$ on a neighbourhood of $K'\cap K''$ which is the identity for $t\in H$. Since $\Psi$ is near to the identity,  $\log \Psi$ exists for $p(x)$ near $K'\cap K''$. Multiply $\log\Psi(t,x)$ by a smooth invariant cutoff function which is 1 on a neighbourhood of $K'\cap K''$   such that the closure of the support is  compact in our original neighbourhood of $K'\cap K''$. Then the corresponding automorphism $\widetilde \Psi(t,x)$ is defined everywhere, in particular, on a neighbourhood of $K'$. Now we set
  $$
  \Phi'(t,x)=\widetilde \Psi(t,x)\inv \widetilde \Phi'(t,x)\text{ for $p(x)$ in a neighbourhood of $K'$,}
  $$
  $$
  \Phi''(t,x)=\widetilde \Phi''(t,x) \text{ for $p(x)$ in a neighbourhood of $K''$.}
  $$
  Then on a neighbourhood of $K'\cap K''$ we have
  $$
  \Phi(t,x)=\Phi'(t,x)\cdot \Phi''(t,x)\inv\text{ for all $t\in C$.}
  $$
  This completes the proof of the proposition.
\end{proof}

\bibliographystyle{amsalpha}
\bibliography{Oka.paperbib}

\end{document}